\documentclass{article}
\usepackage[top=3cm,bottom=3cm]{geometry}
\usepackage{amssymb,amsthm,amsmath,amsxtra,dsfont}
\usepackage{bbm}
\usepackage[nottoc]{tocbibind}   
\usepackage{graphicx,tikz,esint,pgfplots}
\usepackage{hyperref}
\usepackage{mathtools}
 
\hypersetup{colorlinks=true, citecolor=blue, linkcolor=black}
\usepackage{geometry,graphicx}

\usepackage{mathrsfs}
\numberwithin{equation}{section}

\newcommand{\sgn}{\operatorname{sgn}}

\newcommand{\tr}{\operatorname{Tr}}

\renewcommand{\P}[1]{\mathbb{P}\left[#1\right]}
\newcommand{\E}[1]{\mathbb{E}\left[#1\right]}

\newcommand{\Oo}[2]{\underset{#1}{\mho_{#2}}}

\newcommand{\1}{\mathds{1}}
\newcommand{\A}{\mathfrak{A}}

\newcommand{\C}{\mathbb{C}}
\newcommand{\Ci}{\operatorname{Ci}}
\newcommand{\Cin}{\operatorname{Cin}}

\newcommand{\Cau}{\mathcal{C}}

\newcommand{\f}{\mathrm{f}}
\newcommand{\F}{\mathscr{F}}
\newcommand{\g}{\mathrm{g}}

\newcommand{\Id}{\mathrm{I}}

\newcommand{\J}{\chi}

\newcommand{\Hi}{\mathscr{H}}

\newcommand{\N}{\mathbb{N}}

\newcommand{\pii}{2 \pi i}
\newcommand{\Pt}{\widetilde{\mathbb{P}}}

\newcommand{\R}{\mathbb{R}}
\renewcommand{\S}{\mathcal{S}}

\newcommand{\T}{\mathbb{T}}

\newcommand{\Res}{L}
\renewcommand{\u}{\mathbf{u}}
\newcommand{\X}{\widetilde{X}}

\def\Xint#1{\mathchoice
{\XXint\displaystyle\textstyle{#1}}%
{\XXint\textstyle\scriptstyle{#1}}%
{\XXint\scriptstyle\scriptscriptstyle{#1}}%
{\XXint\scriptscriptstyle%
\scriptscriptstyle{#1}}%
\!\int}
\def\XXint#1#2#3{{\setbox0=\hbox{$#1{#2#3}{\int}$ }
\vcenter{\hbox{$#2#3$ }}\kern-.6\wd0}}
\def\dashint{\Xint-}

\newtheorem{theorem}{Theorem}[section]

\newtheorem{proposition}[theorem]{Proposition}
\newtheorem{corollary}[theorem]{Corollary}
\newtheorem{conjecture}[theorem]{Conjecture}
\newtheorem{lemma}[theorem]{Lemma}

 \newtheorem{remark}[theorem]{Remark}

\newtheorem{assumption}{Assumption}[section]

\begin{document}
\date{}

\title{\bf Subcritical multiplicative chaos for regularized counting statistics from random matrix theory}
\author{Gaultier Lambert\thanks{Department of Mathematics, Universit\"at Z\"urich, gaultier.lambert@math.uzh.ch. G. L.~gratefully acknowledges the support of the grant
KAW 2010.0063 from the Knut and Alice Wallenberg Foundation while at KTH, Royal Institute of Technology.}\ , Dmitry Ostrovsky\thanks{Stamford, CT, USA. dm\_ostrov@aya.yale.edu.}\, and Nick Simm\thanks{School of Mathematical and Physical Sciences, Department of Mathematics, University of Sussex, Falmer Campus, Brighton, BN1 9QH, UK. N.J.Simm@sussex.ac.uk. N. S. gratefully acknowledges financial support from a Leverhulme Trust Early Career Fellowship ECF-2014-309.}}
\maketitle

\begin{abstract}
For an $N \times N$ Haar distributed random unitary matrix $U_{N}$, we consider the random field  defined by counting the number of eigenvalues of $U_{N}$ in a mesoscopic arc centered at the point $u$ on the unit circle. We prove that after regularizing at a small scale $\epsilon_{N}>0$, the renormalized exponential of this field converges as $N \to \infty$ to a Gaussian multiplicative chaos measure in the whole subcritical phase. We discuss implications of this result for obtaining a lower bound on the maximum of the field. We also show that the moments of the total mass converge to a Selberg-like integral and by taking a further limit as the size of the arc diverges, we establish part of the conjectures in \cite{Ost16}. 
By an analogous construction, we prove that the multiplicative chaos measure coming from the sine process has the same distribution, which strongly suggests that this limiting object should be universal. 
Our approach to the $L^{1}$-phase is based on a generalization of the construction in Berestycki \cite{Berestycki15} to random fields which are only \textit{asymptotically} Gaussian. In particular, our method could have applications to other random fields coming from either random matrix theory or a different context. 
\end{abstract}

\section{Introduction}
\subsection{Background and results for the CUE} \label{sect:BR}
The study of Gaussian fields with logarithmic correlations has seen many recent developments in the last few years. One of those concerns a relation to the eigenvalues of random matrices, which can be traced back to a work of Hughes, Keating and O'Connell \cite{HKC01}. They studied the characteristic polynomial of  random $N \times N$ matrices $U_N$ sampled from the unitary group according to the Haar measure, also known as the Circular Unitary Ensemble (CUE). One of their key results was to prove that $Z_{N}(\theta) = \sqrt{2}\log|\det(e^{i\theta}-U_N)|$ converges in law as $N \to \infty$ to a log-correlated Gaussian field which corresponds to the restriction of the two-dimensional Gaussian Free Field on the unit circle. This field can be represented as the Fourier series:
\begin{equation}
Z(\theta) = \frac{1}{\sqrt{2}}\mathrm{Re}\bigg\{\sum_{k=1}^{\infty}\frac{e^{-ik\theta}}{\sqrt{k}}\,Z_{k}\bigg\} \label{1overf} ,
\end{equation}
where $Z_{k}$ are i.i.d. standard complex Gaussian variables. In particular, almost surely, this series does not converge in $L^2(\T)$, and $Z(\theta)$ only exists as a random distribution, or a generalized random function with correlation kernel
$\mathbb{E}(Z(\theta)Z(\theta')) = -\log|e^{i\theta}-e^{i\theta'}|$.

More recent developments concern the extreme value statistics of the field $Z_{N}(\theta)$ and on the related issue of making sense of its exponential in the limit as $N \to \infty$. The authors of \cite{FHK12,FK14} gave a very precise conjecture for the maximum value of $Z_{N}(\theta)$, which has recently seen significant progress \cite{ABB16, PZ16, CNM16}. This conjecture is intimately related to the following\footnote{For convenience we state conjecture 1.1 for $q \in \mathbb{N}$. The full conjecture given by these authors involves an analytic continuation to $q \in \mathbb{C}$ which allowed them to calculate the extreme value statistics of $Z_{N}(\theta)$ by invoking the so-called `freezing scenario' expected of log--correlated Gaussian processes.} concerning the total mass of the field $Z_{N}(\theta)$.

\begin{conjecture}[Fyodorov and Keating \cite{FK14}]
Let $\gamma>0$ and define
\begin{equation}
M^{\gamma}_{N} := \int_{0}^{2\pi}e^{\gamma Z_{N}(\theta)}\,d\theta \label{totalmass}
\end{equation}
For any $q \in \mathbb{N}$ such that  $\gamma^{2} q < 2$, we have
\begin{equation}
\begin{split}
\lim_{N \to \infty}N^{-\frac{\gamma^{2} q}{2}}\mathbb{E}[(M^{\gamma}_{N})^{q}] &= C_{\gamma,q}\int_{[0,2\pi]^{q}}\prod_{1 \leq j < k \leq q}|e^{i\theta_j}-e^{i\theta_k}|^{-\gamma^{2}}\,d\theta_1\ldots d\theta_q\\
&= C_{\gamma,q}\frac{\Gamma(1-\gamma^{2} q/2)}{\Gamma(1-\gamma^{2}/2)^{q}} \label{FKconj}
\end{split}
\end{equation}
where $C_{\gamma,q} = (2\pi)^{q}\frac{G(1+\gamma/\sqrt{2})^{2q}}{G(1+\gamma\sqrt{2})^{q}}$ and $G(z)$ is the Barnes G-function. \label{co:FKconj}
\end{conjecture}
The case $q=2$ of Conjecture \ref{co:FKconj} was solved by Claeys and Krasovsky \cite{CK15} by a rather delicate asymptotic analysis of Toeplitz determinants with merging Fisher-Hartwig singularities\footnote{Claeys and Krasovsky also obtain the asymptotics for the so-called critical and super-critical regimes of the second moment ($\gamma=1$ and $\gamma > 1$ respectively). For general $q$ with $\gamma^{2}q>2$ the asymptotic order of magnitude was conjectured by Fyodorov and Keating \cite{FK14}.}. For $q \neq 1, 2$, this conjecture remains an open problem. On the other hand, the normalization by $N^{-\frac{\gamma^{2}q}{2}}$ on the left-hand side of \eqref{FKconj} suggests that $N^{-\gamma^{2}/2}M^{\gamma}_{N}$ might converge in distribution to a non-trivial limiting random variable as $N \to \infty$, at least for some range of $\gamma$ values. This has been shown by Webb \cite{W15} for any $\gamma < 1$. The condition $\gamma < 1$ is called the $L^{2}$-phase because it is precisely the regime where the second moment (case $q=2$ of \eqref{FKconj}) is finite. Webb showed that the limiting random variable can be described in terms of \textit{Gaussian multiplicative chaos} (GMC), a theory devoted to properly defining the exponential of a log-correlated Gaussian field. These exponentials are interpreted as random measures and are naturally linked to the geometric properties of the fields.
This topic is currently under intense investigation, partly because of its applications in subjects such as Liouville quantum gravity, turbulence or mathematical finance, see \cite{RV14} for a review. We provide a  detailed discussion of GMC in Appendix \ref{sect:GMC}, but the basic idea can be summarised as follows. Consider a  Gaussian field $G(u)$ on a domain $\A \subset \mathbb{R}^{d}$ with mean zero and covariance
\begin{equation}
\mathbb{E}(G(u)G(v)) = \log\frac{1}{|u-v|}+g(u,v) \label{Gcov}
\end{equation}	
where $g(u,v)$ is bounded and continuous. Because of the logarithmic singularity in \eqref{Gcov}, $G(u)$ is not an ordinary function, so one wishes to regularize it in some way, obtaining $G_{\epsilon}(u)$ which is a smooth Gaussian field for every finite $\epsilon>0$. Then one may consider the exponential of the field $G_{\epsilon}(u)$ and remove the diverging part in order to take the limit as $\epsilon\to 0$. Typically, one finds that this limit does not depend on the regularization procedure. 
This was first established by Kahane, \cite{KP76,Kah85}, using martingale methods and more recently extended to other regularizations such as convolutions.

\begin{theorem}[Robert and Vargas, \cite{RV10}, Berestycki \cite{Berestycki15}]
\label{th:GMCintro}
 Let $G$ be a log--correlated Gaussian field and $\phi$ a smooth probability density function (mollifier) on $\R^d$. For any $\epsilon>0$, let $G_{\epsilon} = G \ast \phi_{\epsilon}$ be the convolution of $G$ with $\phi_{\epsilon}(x) := \frac{1}{\epsilon}\phi(\frac{x}{\epsilon})$. For any function $w \in L^{1}(\R^d)$ uniformly bounded with compact support,  define the family of random variables
\begin{equation}
\nu^{\gamma}_{\epsilon}(w) = \int e^{\gamma G_{\epsilon}(u) - \frac{\gamma^{2}}{2}\mathrm{Var}(G_{\epsilon}(u))}w(u)\,du .  \label{gmcnu}
\end{equation}
 Then for every $\gamma < \sqrt{2d}$,  $\nu^{\gamma}_{\epsilon}(w) $ converges in $L^1(\mathbb{P})$ to a non-trivial random variable  $\nu^{\gamma}(w) $ as $\epsilon \to 0$. 
 This limit does not depend on the mollifier $\phi$ and $\nu^{\gamma}$ defines a random measure which is called the GMC measure associated with the field $G$.
 \end{theorem}

The aim of this paper is to establish proofs of this convergence for fields which are only \textit{asymptotically} Gaussian and log-correlated. Our main interest in such fields will be those which naturally arise in random matrix theory, but our theory is likely to apply in other situations. 

We now define our main objects of study for the CUE. Let $e^{i\theta_1},\ldots,e^{i\theta_N}$ denote the eigenvalues of a Haar distributed random unitary matrix $U_{N}$, with the convention that $\theta_{1},\ldots,\theta_{N} \in [-\pi,\pi)$. We consider the random process $u \to W_{N}(u)$ which counts the number of eigenangles in an interval centered around $u$:
\begin{equation}
W_{N}(u) = \sum_{j=1}^{N}\chi_{u}(N^{\alpha}\theta_j) \label{indicator} , 
\end{equation}
where $\J_u(x) = \pi \1_{|x-u| \le \ell/2}$ for some fixed $\ell>0$. We emphasize that throughout the paper, the role played by $u$ is that of the spatial variable of the random process under consideration. The parameter $\alpha$ is called the spectral scale and the process \eqref{indicator} can be studied in three different regimes. 
In the microscopic regime $\alpha=1$, the field $W_N(u)$ converges weakly as $N\to\infty$ to a counting function for the sine process which is not a Gaussian process, but we will come back to this in Section \ref{sect:sine}. 
On the other hand, either in the mesoscopic regime\footnote{Mesoscopic means that the number of objects (zeroes, eigenvalues) 
that are being counted goes to infinity, whereas the length of the interval, 
over which they are being counted, goes to zero.} $0 < \alpha < 1$ or in the global regime $\alpha=0$, if centered, the process $W_{N}(u)$ will converge in a certain sense to a log--correlated Gaussian field. 
We will focus on the mesoscopic regime where the geometry of the spectrum is unimportant, although our theory also applies for $\alpha=0$ with minor changes. Inspired by Theorem \ref{th:GMCintro}, we will work with the following family of regularizations instead of working directly with 
the counting statistics \eqref{indicator}. Let $\phi$ be a mollifier and define
\begin{equation}
X_{N,\epsilon}(u) :=  \sum_{j=1}^{N}( \J_u\ast \phi_{\epsilon})(N^{\alpha}\theta_j) \label{smoothing}
\end{equation}
where $\phi_{\epsilon}(\theta) = \frac{1}{\epsilon}\phi\left(\frac{\theta}{\epsilon}\right)$. Throughout the paper, we use the notation $\overline{X}_{N,\epsilon_N} := X_{N,\epsilon_N}-\mathbb{E}X_{N,\epsilon_N}$ to denote a recentering by the expectation. The smoothed fields \eqref{smoothing} were introduced in \cite{Ost16} in the context of counting statistics of the Riemann zeros. 
 The parameter $\epsilon > 0$ controls the scale of regularization for $X_{N}(u)$, which we will allow to converge to zero as the dimension $N\to\infty$. Roughly speaking, the speed at which the sequence $\epsilon_N\to0$ controls how close the centered field $\overline{X}_{N,\epsilon_N}$ is to a Gaussian process for large $N$.

For fixed $\epsilon>0$, \eqref{smoothing} is a smooth linear statistic and by Soshnikov's central limit theorem (CLT) \cite{S00a}, for every $u\in\R$, we have the convergence in distribution to a Gaussian random variable,
\begin{equation}
\overline{X}_{N,\epsilon}(u)  \Rightarrow  \mathcal{N}\left(0,\int_{\mathbb{R}}|k||\widehat{\J_u\ast \phi_{\epsilon}}(k)|^{2}\,dk\right),  \label{CLT}
\end{equation} 
as $N\to\infty$, see formula \ref{Fourier_T} for our normalization of the Fourier transform. A direct calculation with the limiting variance in \eqref{CLT} easily shows (using \textit{e.g.} the fact that the Fourier transform of $\J_u$ decays as $1/k$ for large $k$) that $\mathrm{Var}(X_{N,\epsilon}(u)) \sim \log\frac{1}{\epsilon}$ as $N\to\infty$ followed by $\epsilon \to 0$. In general, the covariance structure coming from Soshnikov's CLT is associated with an $H^{1/2}$ Sobolev space in the following way: 
for any functions  $g,h\in C^1_0(\R)$,
\begin{equation}\label{Sos}
\begin{split}
\lim_{N \to \infty}\mathrm{Cov}\left(\sum_{j=1}^{N}g(N^{\alpha}\theta_j),\sum_{j=1}^{N}h(N^{\alpha}\theta_j)\right) &= \int_{\mathbb{R}}\,|k|\hat{g}(k)\hat{h}(-k)\,dk\\
&= \frac{1}{2\pi^{2}}\int_{\mathbb{R}^{2}}g'(x)h'(y)\log\frac{1}{|x-y|}\,dx\,dy . 
\end{split}
\end{equation}

This is suggestive of a logarithmic covariance structure underlying the mesoscopic statistics of CUE eigenvalues. Applied to the statistic \eqref{smoothing}, one can easily show that for $u,v \in \mathbb{R}$ fixed $(u \neq v)$, we have
\begin{equation} \label{cov'}
\lim_{\epsilon \to 0}\lim_{N \to \infty}\mathrm{Cov}(X_{N,\epsilon}(u),X_{N,\epsilon}(v)) = \log\frac{1}{|u-v|}+\log|\ell^{2}-(u-v)^{2}|^{1/2}. 
\end{equation}

By analogy with \eqref{gmcnu}, we are going to study the convergence of the following quantity as $N \to \infty$, where $\epsilon$ may converge to zero with $N$.
\begin{equation}
\mu^{\gamma,\mathrm{CUE}}_{N,\epsilon}(w) := \int_{\mathbb{R}}\,e^{\gamma \overline{X}_{N,\epsilon}(u)-\frac{\gamma^{2}}{2}\mathrm{Var}(\overline{X}_{N,\epsilon}(u))}w(u)\,du, \label{muintro}
\end{equation}
where $w \in L^{1}(\mathbb{R})$ has compact support. 

\begin{theorem}
\label{th:gmc}
Consider the regularized statistic \eqref{smoothing} with $\ell>0$ and mesoscopic scale $0 < \alpha < 1$ fixed. Suppose that $\epsilon_{N} \to 0$ in such a way that for some $\kappa>0$, 
\begin{equation}
N^{\alpha-1}\epsilon_{N}^{-1} = \underset{N\to\infty}{O}(N^{-\kappa})  \label{c1}.
\end{equation} 
Then for every $0 < \gamma < \sqrt{2}$ and $w\in L^1(\R)$ uniformly bounded with compact support, we have the convergence in distribution
\begin{equation}
\mu^{\gamma,\mathrm{CUE}}_{N,\epsilon_N}(w)  \Rightarrow \nu^{\gamma}(w), \qquad N \to \infty , \label{convdist}
\end{equation}
where $\nu^{\gamma}$ is the GMC measure defined in Theorem \ref{th:GMCintro} with $g(u,v) = \log|\ell^{2}-|u-v|^{2}|^{1/2}$ in the covariance formula \eqref{Gcov}.
\end{theorem}

\begin{remark}
The reader may observe that $g(u,v)$ in Theorem \ref{th:gmc} is not bounded below, as in the hypothesis of Theorem \ref{th:GMCintro}. However, it is easy to see that precisely the same steps carried out in \cite{Berestycki15} for constructing the multiplicative chaos measure show that this assumption is not required, the main crucial point being that $g(u,v)$ is bounded above.
\end{remark}
\begin{remark}
In the following, we interpret $\nu^{\gamma}_{\epsilon}$ in \eqref{gmcnu} and $\mu^{\gamma, \mathrm{CUE}}_{N,\epsilon}$ in \eqref{muintro} as absolutely continuous random measures and we will use the notation $\nu^{\gamma}_{\epsilon}(S) = \nu^{\gamma}_{\epsilon}(\1_S)$ for any compact Borel subset $S \subset \mathbb{R}^d$ and similarly for $\mu^{\gamma, \mathrm{CUE}}_{N,\epsilon_N}$. 
\end{remark}

 The condition \eqref{c1} is rather natural and we will discuss its meaning and necessity in the next subsection, but we mention that its importance was first emphasized for smoothed statistics of the Riemann zeros in work of the second author \cite{Ost16}. There it is shown that \eqref{c1} is the natural slow decay condition both for keeping the statistic 
mesoscopic and for preserving its $H^{1/2}$-Gaussianity 
under smoothing.
  
It turns out that the appearance of the $H^{1/2}$-Gaussian noise is remarkably universal in the mesoscopic limit of one dimensional ensembles with random matrix type repulsion and this problem has attracted renewed interest in the last couple of years. For instance, the analogue of Soshnikov's CLT \eqref{Sos} was obtained for the GUE \cite{FKS16}, more general invariant ensembles \cite{BD16, Lam15}, Wigner matrices \cite{EK15,LS15,HK16}, $\beta$-ensembles \cite{BEYY15, BL16} and for zeros of the Riemann zeta function \cite{BK14,R14}. 
 It is likely the counterpart of Theorem \ref{th:gmc} continues to hold for these models as well. To add some weight to this assertion, we will present the proof of Theorem~\ref{th:gmc} under some general criteria (see Theorem~\ref{th:GMCgeneral} below) and provide an analogous result for the sine process, that is the random point process which describes the microscopic limit of a wide class of Hermitian random matrices, see Section~\ref{sect:sine}.
Thus, from the point of view of Gaussian multiplicative chaos, one may view the different ensembles of random matrices as alternative ways of regularizing a log-correlated Gaussian field.\\

 We also prove a result establishing the convergence of the $q$-moments of \eqref{muintro} and a relation to Selberg integrals. Although $\ell$ is fixed in Theorem \ref{th:gmc}, we now consider the case where $\ell = L(N) \to \infty$ as $N \to \infty$. Similarly, the rate at which $L(N) \to \infty$ is naturally restricted by the mesoscopic scale,
\begin{equation}
L(N)\log(N) = \underset{N\to\infty}{o}(N^{\alpha}) \label{c2}.
 \end{equation}
Under this restriction we prove that the moments of the total mass are given by ratios of Gamma functions in the limit $N \to \infty$.

\begin{theorem}(Moments of the total mass and Selberg integrals) \label{th:moments}
Let $0 < \alpha < 1$ and suppose that the conditions \eqref{c1} and \eqref{c2} are operative. Then for any $r>0$ and $q \in \mathbb{N}$ such that $\gamma^{2}q <  2$, we have
\begin{equation}
\begin{split}
& \lim_{N \to \infty}L(N)^{-\gamma^2\binom{q }{2}}\E{(\mu^{\gamma,\mathrm{CUE}}_{N,\epsilon_N}([0,r]))^{q}} = \lim_{\ell \to \infty}\lim_{\epsilon \to 0}\lim_{N \to \infty} \ell^{-\gamma^2\binom{q}{2}}\E{(\mu^{\gamma,\mathrm{CUE}}_{N,\epsilon}[0,r])^{q}}\\
&\qquad= \int_{[0,r]^{q}}\,|\Delta(\vec{u})|^{-\gamma^2}\,du_{1}\ldots du_{q} = r^{q+\gamma^2\binom{q}{2}}\prod_{j=0}^{q-1}\frac{\Gamma(1-(j+1)\gamma^2/2)\Gamma^{2}(1-j\gamma^2/2)}{\Gamma(1-\gamma^2/2)\Gamma(2-(q+j-1)\gamma^2/2)}, \label{largeint} 
\end{split}
\end{equation}
where $\Delta(\vec{u}) = \prod_{1 \leq j < k \leq q}(u_{k}-u_{j})$.
\end{theorem}
This theorem may be viewed as a smoothed analogue of Conjecture \ref{co:FKconj} and solves the case $q \in \mathbb{N}$ of the conjectures posed by the second author in \cite{Ost16}. In the latter paper, the double limit $N \to \infty$ followed by $\epsilon \to 0$ is referred to as the \textit{weak conjecture}, while the more fundamental single limit in the first equality is the \textit{strong conjecture}. Theorem \ref{th:moments} demonstrates the equivalence of the strong and weak limits for positive integer $q \in \mathbb{N}$. The full statement of these conjectures consist of assertions valid for any $q \in \mathbb{C}$; such generality is beyond the scope of this paper, though one expects Theorem \ref{th:moments} to hold for general $q \in \mathbb{C}$ with the appropriate analytic continuation of the Selberg integral as explained below, \textit{c.f.} equation (49) in \cite{Ost16}. 
The exponent of the factor $r^{q+\gamma^{2}\binom{q}{2}}$ in \eqref{largeint} is known as the \textit{structure exponent} and its quadratic dependence on $q$ are associated with the multi-fractal nature of the underlying random measures \cite{RV14}. It is natural  to ask what happens if we take a different interval in \eqref{indicator}; indeed in \cite{Ost16} the examples $[0,u]$ and $[-L(N),u]$ are studied in detail and the analogous results conjectured there for $q \in \mathbb{N}$ follow easily with our approach. In fact the second example $[-L(N),u]$ behaves similarly to our interval $[u-L(N)/2,u+L(N)/2]$ with $L(N) \to \infty$. However, it turns out that the random measures corresponding to these examples are \textit{not normalizable} in the usual GMC sense of \eqref{muintro}. We leave it as an interesting open problem to find a mathematically rigorous way of normalizing the GMC measures corresponding to these interval statistics.

The integral on the second line of \eqref{largeint} is a particular case of a multi-dimensional integral known as \textit{Selberg's integral}, due to its explicit evaluation by Selberg in 1944, see \cite{ForWar08} for a detailed historical review (similarly the integral in \eqref{FKconj} is referred to as the \textit{Dyson integral}). It is known that
the Dyson and Selberg integrals describe the moments of the total mass of the Bacry-Muzy Gaussian multiplicative chaos measure on the circle and the interval, respectively. Hence the problem of extending these integrals to meromorphic functions of $q\in\C$ in such a way that the resulting function is the Mellin transform of a probability distribution is of fundamental importance as the probability
distribution is then naturally conjectured to be that of the total mass. This problem was solved by Fyodorov and Bouchaud
\cite{FB08} for the Dyson integral, heuristically by Fyodorov \emph{et. al.} \cite{FlDR09} and independently and rigorously
by the second author \cite{Ost09,Ost13,Ost14} for the Selberg integral, see also \cite{Ost16b} for the Morris integral.
These analytic extensions are fundamental in the theory of log-correlated Gaussian processes \cite{FB08, FlDR09, FHK12,FK14,FS16,FlD16} as they lead to precise conjectures about the explicit form of the extreme value statistics, by an analogy with the so-called derivative martingale coming from branching processes. The connection between the derivative martingale and the distribution of the maximum has been rigorously proved for a certain class of log-correlated Gaussian fields \cite{DRZ15}. 

In view of our Theorems \ref{th:gmc} and \ref{th:moments}, it is reasonable to ask whether they tell us anything new about eigenvalue statistics of the CUE. It turns out that Theorem \ref{th:gmc} can be used to obtain a lower bound on the maximum of the process $W_{N}(u)$, \eqref{indicator}. 
 Namely, by first proving a lower bound for the maximum of the field $\overline{X}_{N,\epsilon_N}(u)$ where $\epsilon_N=N^{\alpha-1+\kappa}$ and $\kappa>0$ is small and then by taking care of removing the regularization, it can be shown that 
 \begin{equation} \label{max_LB_meso}
\lim_{N\to\infty}  \P{\max_{|u| \le C} W_N(u) \ge \frac{\ell}{2}N^{1-\alpha}+   (\sqrt{2}-\delta) \log N^{1-\alpha}}=1 ,
\end{equation} 
for any mesoscopic scale $0< \alpha <1$, any small $\delta>0$, and for any large $C>0$.   
Note that $\E{W_N(u)} = \frac{\ell}{2}N^{1-\alpha}$ and the constant $\sqrt{2}$ corresponds to the GMC critical value in dimension 1. 
In the global regime, our method produces the same bound with $\alpha=0$ and $C=\pi$, providing a result analogous to \cite[Theorem~1.2]{ABB16} about the leading order of the maximum of the imaginary part of the characteristic polynomial. In this case, it is known that the GMC critical value  describes the leading order of the maximum of the field in the sense that  for any $\delta>0$,
\begin{equation} \label{max_LB}
\lim_{N\to\infty}  \P{\max_{|u|\le\pi} W_N(u) \le  \frac{\ell}{2}N^{1-\alpha}+ (\sqrt{2}+\delta) \log N}=1 . 
\end{equation} 
In general, it is a difficult task to establish a lower bound for the extreme values of log--correlated fields. On the other hand, we will prove in a future publication that the bound \eqref{max_LB_meso} follows from the convergence \eqref{convdist} provided that it holds throughout the subcritical phase.
The main idea is simple and relies on the fact that the random measure $\mu_{N, \epsilon_N}^{\gamma,\mathrm{CUE}}$ {\it lives on the set of $\gamma$-thick points}:
 \begin{equation} \label{gamma_points}
\mathscr{T}_N^\gamma := \left\{  x\in \A : \overline{X}_{N,\epsilon_N}(x) \ge \gamma \E{ \overline{X}_{N,\epsilon_N}(x)^2 } \right\}  ,
\end{equation}
for large $N$, so that if its limit is non-trivial, there must exist (random) points where the field takes atypically large values. 
 It is an interesting question whether the convergence of the exponential measure in the critical case, $\gamma=\sqrt{2d}$, would yield some further information about extreme value statistics and whether it is possible to extend the validity of Theorem~\ref{th:gmc} to this case.

An interesting consequence of the lower-bound \eqref{max_LB_meso} concerning eigenvalue rigidity\footnote{By rigidity, we mean that the number of CUE eigenvalues contained in any (mesoscopic) arc $A$ on the unit circle concentrates near its mean $|A|\frac{N}{2\pi}$ with high probability when the dimension $N$ is large.} is that for any $\delta>0$ and scale $\alpha<1$, when $N$ is large there exists an arc $\mathrm{A}$ of size $N^{-\alpha}$ so that 
\begin{equation}
\#\{ \theta_j \in\mathrm{A}\}  \ge \E{\#\{ \theta_j \in\mathrm{A}\} } + (\sqrt{2}-\delta) \log\big(N|A| \big), \label{crowd}
\end{equation}
with high probability. That is the eigenvalues are over-crowded in this arc, since one typically expects the second term in \eqref{crowd} to be of order $\sqrt{\log N}$. We are planning in future work to further discuss the consequences of our results for extreme value statistics of log-correlated fields and eigenvalue rigidity in the random matrix context. In particular, for the sine process using the results of Section~\ref{sect:sine}.

\subsection{Strategy of the proof}

In order to prove Theorem~\ref{th:gmc}, we develop a general method for constructing multiplicative chaos measures in models where the log-correlated fields are only asymptotically Gaussian.
 In general, this requires asymptotics for the exponential moments of the field, which is much stronger than the usual approximation given by a CLT such as \eqref{CLT}.
The precise assumptions of our theory are rather technical and we present them in full detail in Section~\ref{sect:proof} together with the proof of Theorem~\ref{th:GMCgeneral} below. In this introduction, for the sake of transparency, we formulate our results in a simpler setting which is closer to that described in Section~\ref{sect:BR}, but remains rather general.

\begin{theorem} \label{th:GMCgeneral}
As in Theorem \ref{th:GMCintro}, let $G$ be a log--correlated Gaussian field with covariance \eqref{Gcov}, $\phi$ be a mollifier, and $G_{\epsilon} = G\,\ast\,\phi_{\epsilon}$ for any $\epsilon>0$.  
For each $N \in \mathbb{N}$, let $(X_{N,\epsilon}(u))_{u \in \A, \epsilon\in(0, 1]}$ be a centered random field on $\A \subset \R^{d}$ (not necessarily Gaussian).
Suppose that there exists a sequence $\delta_N \to 0$ as $N\to\infty$ so that for any $q\in\N$ and $\boldsymbol{\gamma}\in \R^q $, 
\begin{equation}  \label{qmoms}
\log\mathbb{E}\bigg[\mathrm{exp}\bigg( \sum_{j=1}^{q} \gamma_j X_{N,\epsilon_j}(u_{j})\bigg)\bigg] = \frac{1}{2}\sum_{j,k=1}^{q} \gamma_j\gamma_k\E{G_{\epsilon_j}(u_j) G_{\epsilon_k}(u_k)} + \underset{N\to\infty}{o(1)}
\end{equation}
uniformly for all $\u\in \A^{q}$ and $\boldsymbol{\epsilon}\in (\delta_N, 1]^q$. Then,  for any $0<\gamma<\sqrt{2d}$ and $w\in L^1(\A)$ uniformly bounded, the sequence of random variables
\begin{equation}
\mu^{\gamma}_{N}(w) := \int e^{\gamma X_{N,\delta_N}(u)-\frac{\gamma^{2}}{2}\E{ X_{N,\delta_N}(u)^2}}\,w(u)\,du 
\label{GMCgeneral}
\end{equation}
converges in distribution to the random variable $\nu^{\gamma}(w)$ where $\nu^{\gamma}$ is the GMC measure defined in Theorem \ref{th:GMCintro} associated with the field $G$.
\end{theorem}
Let us note that in the above, the field $X_{N,\epsilon}(u)$ need not be defined on the same probability space for different $N$. Theorem~\ref{th:GMCgeneral} establishes the convergence to a GMC measure in the whole subcritical phase $\gamma<\sqrt{2d}$.
 The proof is inspired by the elementary argument introduced by Berestycki \cite{Berestycki15} for establishing Theorem \ref{th:GMCintro}.
 Our main contributions are to single out the assumptions needed to apply these ideas in a setting where the fields are only asymptotically Gaussian and find a way to replace the Gaussian techniques (like, for instance,  the use of Girsanov's theorem) with other estimates using the strong asymptotics\footnote{In fact, the argument presented in Section~\ref{sect:proof} does not require that the asymptotics \eqref{qmoms} hold in full generality.} \eqref{qmoms}. 
 The strategy of the proof of Theorem~\ref{th:GMCgeneral} consists in constructing a (random) set $\mathfrak{S} \subset \mathfrak{A}$ which depends on the parameters $0<\gamma<\sqrt{2d}$ and $N\in\N$  in such a way that $\mu^{\gamma}_{N}(w\1_\mathfrak{S})$ is uniformly bounded in $L^2(\mathbb{P})$ and
$\lim_{N\to\infty} \E{\mu^{\gamma}_{N}(w\1_{\A\setminus\mathfrak{S}})} =0$. This second moment computation makes essential use of the uniformity of the error term in the asymptotics \eqref{qmoms}. These properties show that the sequence of random variables $\mu^{\gamma}_{N}(w)$ is tight and we identify its limit  by showing that
 $$
 \lim_{N\to\infty} \mu^{\gamma}_{N}(w) \overset{d}{=}  \lim_{\epsilon\to0} \lim_{N\to\infty} \mu^{\gamma}_{N, \epsilon}(w)  \overset{d}{=} \nu^\gamma(w) . 
  $$
 By $"\overset{d}{=}"$ we mean that these random variables have the same law. 
 This last idea is taken from the work of Webb, \cite{W15}, on the convergence of the characteristic polynomial of the CUE. However, Webb only established the convergence in the $L^2$-phase, $\gamma<\sqrt{d}$, by showing that in this regime,  $ \mu^{\gamma}_{N}(w)$ is uniformly bounded in $L^2(\mathbb{P})$ and is therefore uniformly integrable -- see Section~\ref{sect:L^2}. In the so--called 
 $L^1$-phase, $\sqrt{d}<\gamma < \sqrt{2d}$,
 $ \E{ \mu^{\gamma}_{N}(1)^2} \to\infty$ as $N\to\infty$ 
and, in effect, we need to restrict the random measure $\mu^{\gamma}_{N}$ to the set of points which are {\it not $\alpha$-thick}, see \eqref{gamma_points}, for some parameter $\alpha$ slightly bigger than $\gamma$, 
 to restore the uniform integrability property. This is the main idea to construct the set  $\mathfrak{S}$.  \\

 As we already mentioned, Theorem~\ref{th:gmc} follows from an application of Theorem~\ref{th:GMCgeneral} to the random field \eqref{smoothing} where $G$ is a stationary Gaussian field on $\R$ with covariance structure:
 \begin{equation}
\E{G(u) G(v))} = Q(u-v)= \log\frac{1}{|u-v|}+\log|\ell^{2}-(u-v)^{2}|^{1/2} ,  \label{cov}
\end{equation}
see formula \eqref{cov'}. For the CUE, since the random variables $X_{N,\epsilon}(u)$ are linear statistics, the strong Gaussian asymptotics \eqref{qmoms} can be checked by exploiting the connection between the CUE and Toeplitz or Fredholm determinants. In particular, we use the following remarkable formula.

\begin{theorem}[Borodin-Okounkov-Case-Geronimo formula]
Let $g$ be a function on the unit circle satisfying the regularity condition
\begin{equation}
\sigma^{2}(g) := \sum_{k=1}^{\infty}k|\hat{g}(k)|^{2} < \infty,\quad \text{where} \quad \hat{g}(k) = \frac{1}{2\pi}\int_{0}^{2\pi}g(e^{i\theta})e^{-ik\theta}\,d\theta.
\end{equation}
Then we have the exact formula
\begin{equation}
\mathbb{E}\bigg[\mathrm{exp}\bigg(\sum_{j=1}^{N}g(e^{i\theta_{j}})\bigg)\bigg] = \mathrm{exp}\left( N \hat{g}(0)+\frac{1}{2}\sigma^{2}(g)\right)\mathrm{det}\left(1+R_{N}V(g)V(g)^{*}R_{N}\right) \label{BOintro}
\end{equation}
where $V(g)$ is a certain Hilbert-Schmidt operator acting on $\ell^2(\mathbb{N},\mathbb{C})$ (see Section \ref{sect:BO} for full details) and $R_{N}$ is the orthogonal projection onto the subspace $\ell^{2}(\{0,\ldots,N-1\},\mathbb{C})$. 
\end{theorem}

For a comprehensive account of this formula, which originally appeared in \cite{GC79}, see the monograph of Simon \cite{Simon05}. The first factor on the RHS of formula \eqref{BOintro} corresponds to the Laplace transform of a Gaussian random variable with mean  $N \hat{g}(0)$ and variance $\sigma^{2}(g)$. The second factor is a Fredholm determinant of an operator acting on the sequence space $\ell^{2}(\mathbb{N},\mathbb{C})$. For fixed $g$, the operator $R_{N}V(g)V(g)^{*}R_{N}$ converges to zero in the trace norm which implies the Strong Szeg\H{o} limit theorem.
 
In the context of obtaining estimate \eqref{qmoms}, we have\footnote{Here $X^{(2\pi)}_{N,\epsilon}$ is an appropriate periodization of $X_{N,\epsilon}$ which does not change the Gaussian asymptotics \eqref{qmoms}, see Section \ref{sect:BO}} 
 $$
  \sum_{j=1}^{q} \gamma_j X^{(2\pi)}_{N,\epsilon_j}(u_{j}) = \sum_{j=1}^{N}g_N(e^{i\theta_{j}}) 
 $$
where $g_N$ is a smooth function on the unit circle varying with $N$.  Hence, in order to obtain the asymptotics \eqref{qmoms}, the main challenge is to show that, uniformly in the various parameters, 
 $\lim_{N\to\infty} \det(1+R_{N}V(g_N)V(g_N)^{*}R_{N}\big)=1$. This is precisely where the condition \eqref{c1} comes into play. Namely, there is a transition in the asymptotics when the regularization scale~$\epsilon_N$ becomes of the order of the mean eigenvalue spacing $N^{\alpha-1}$. When $\epsilon_{N} = O(N^{\alpha-1})$, the Fredholm determinant on the RHS of formula \eqref{BOintro} does not converge to $1$ and we enter the so-called Fisher--Hartwig regime characterized by the case $\epsilon_{N} \equiv 0$. In fact, if we consider the counting statistics \eqref{indicator} and set $\overline{W}_N= W_N - \mathbb{E}W_N$, we deduce from \cite[Theorem 2.2]{Kra11} that for any $q\in\N$ and $\boldsymbol{\gamma}\in (-1,\infty]^q$, 
	\begin{align}  \label{kraasympt}
\mathbb{E}\bigg[\mathrm{exp}\bigg(\sum_{j=1}^{q}\gamma_{j} &\overline{W}_N(u_j) \bigg) \bigg]= \mathrm{exp}\left(\frac{1}{2}\sum_{j=1}^{q}\gamma_{j}^{2}\log N-\sum_{j < k}\gamma_{j}\gamma_{k}\log|e^{iu_k}-e^{iu_j}|\right.& \\
	&\left.+\sum_{j < k}\frac{\gamma_{j}\gamma_{k}}{2}\log\left(|e^{i(u_k-\ell)}-e^{iu_j}||e^{i(u_k+\ell)}-e^{iu_j}|\right)\right)\prod_{j=1}^{q}\left|G(1+\frac{\gamma_{j}}{2i})\right|^{4}(1+o(1)) , \notag
	\end{align}  	
for fixed parameters $0< u_1<\cdots < u_q$ (the error term is no longer uniform). For the proof of Theorem~\ref{th:GMCgeneral}, it is crucial that the asymptotics \eqref{qmoms} are uniform when the points $u_j$ merge. When $q=2$ this merging has been studied by Claeys and Krasovsky \cite{CK15} and very precise asymptotics are known in the various regimes. 

However, for $q>2$, there are no results about merging singularities and this is known to be a complicated problem. This lack of uniformity is the main obstacle in establishing the analogues of Theorems \ref{th:gmc} and~\ref{th:moments} without the regularization procedure ($\epsilon_N =0$). 
Therefore, the condition \eqref{c1}  simplifies the asymptotics by assuring that the regularized field remains in a Gaussian regime and prevents the technical complications due to the emergence of Fisher--Hartwig singularities.  
Nevertheless, as \eqref{max_LB_meso} shows, such a condition still allows us to recover a sharp lower bound for the leading order behavior of the maximum of the field.\\

Finally, let us remark that the Barnes G-functions in \eqref{kraasympt} seem to indicate non-Gaussianity of the process $\overline{W}_N$ beyond the leading order for large $N$. However, we expect that for a fixed $\gamma$ the normalization used by Webb \cite{W15} $\frac{e^{\gamma \overline{W}_N(u)}}{\mathbb{E} (e^{\gamma\overline{W}_N(u))})}$ and the usual normalization
$e^{\gamma \overline{W}_N(u) - \frac{\gamma^2}{2}\mathbb{E}(\overline{W}_N(u)^2)}$ converge to the same limiting random variable up to a constant factor depending only on $\gamma$.

\subsection{Results for the sine process} \label{sect:sine}

In this subsection, we explain how to apply Theorem \ref{th:GMCgeneral} to regularized counting statistics of the sine process. The results are analogous to those stated for the CUE in Section \ref{sect:BR}. Hence, we expect that similar results hold for well-behaved unitary invariant ensembles as well, but we leave the task of proving the sufficient asymptotics open for a future project. The sine process, denoted by $\Lambda_N$, is the determinantal point process on $\R$ with correlation kernel 
  \begin{equation}\label{sine_kernel}
K_N(u,v) = \frac{\sin\big( \pi N (u-v)\big)}{\pi(u-v)}  .
\end{equation}
We refer to Section~\ref{sect:DPP} below for some background on determinantal processes.
This is a translation invariant point process whose density is $N$ times the Lebesgue measure on~$\R$.  
Recall that, given a mollifier $\phi$, we denote $\phi_\epsilon(x)=\epsilon^{-1}\phi(x/\epsilon)$ and $\J_{u}(x) = \pi \1_{|x-u| \le \ell/2}$ for all $u, x\in\R$.
Let us consider the linear statistics:
\begin{equation} \label{X_field}
X_{N,\epsilon}(u) = \sum_{\lambda\in \Lambda_1} \J_u*\phi_\epsilon(\lambda/N) . 
\end{equation}
As $N\to\infty$, the random variable \eqref{X_field} gives an approximation of the number of eigenvalues in a mesoscopic box in the bulk of the GUE, or say of another unitary invariant ensemble. To see the parallel with the CUE, note that the scaling property of the sine kernel implies that  
\begin{equation} \label{X_scaling}
X_{N,\epsilon}(u) \overset{d}{=} \sum_{\lambda\in \Lambda_N} \J_u*\phi_\epsilon(\lambda) . 
\end{equation}
      
We will focus on the strong regime where $\epsilon= \epsilon_N \to 0$  as $N\to\infty$
 and, viewing $X_{N,\epsilon}$ as an asymptotically  Gaussian field on $\R$, we will construct its chaos measure in the subcritical phase. 
The advantage of working with the random variables \eqref{X_field} instead of the RHS of \eqref{X_scaling} is that it introduces a natural coupling which allows us to obtain a stronger mode of convergence than for the CUE. 
For technical reasons, we will restrict ourselves to the following class of real-analytic mollifiers.
 
 \begin{assumption} \label{phi_assumption}
Suppose that the function $\phi$ is analytic in $| \Im z | <c $ and that $|\Im \phi| < \pi/\ell$ in this strip. We also assume that $\phi \ge 0$ on $\R$, $\displaystyle \int_\R \phi(x) dx =1$,  and  $\sup\big\{  \| \phi\|_{L^1(\R+i s)}  : s \le c/2\big\} = C_\phi <\infty$.   
 \end{assumption}

For any $\gamma>0$, we now consider the random measure:
\begin{equation} \label{X_measure}
\mu_{N,\epsilon}^{\gamma,\mathrm{Sine}}(w) = \int_{\mathbb{R}}e^{\gamma \overline{X}_{N,\epsilon}(u) - \frac{\gamma^{2}}{2}\mathbb{E}((\overline{X}_{N,\epsilon})^{2})}w(u)\,du,
\end{equation}
where, as for the CUE, $\overline{X}_{N,\epsilon} := X_{N,\epsilon}-\mathbb{E}X_{N,\epsilon}$. We also define a space of mollifiers that will enter into the regularized field \eqref{X_scaling}: For any $\alpha\ge 0$, let
\begin{equation} \label{D}
\mathcal{D}_\alpha = \left\{ \phi \in L^1\cap L^2(\R)  : \phi \ge 0 , \int_\R \phi(x) dx =1 \text{ and }   \int_\R |x|^\alpha\phi(x) dx < \infty  \right\}
\end{equation}
and 
$\mathcal{D} = \bigcup_{\alpha>0} \mathcal{D}_\alpha.$ For instance, the functions $\phi(z) = e^{-z^{2}/2}/\sqrt{2\pi}$ or $\phi(z) = \frac{1/\pi}{1+z^2}$ satisfy Assumption~\ref{phi_assumption} and belong to the set $\mathcal{D}$. 
Finally recall that $G$ is the stationary Gaussian field on $\R$ with zero mean  and covariance kernel \eqref{cov}. As in the case of the CUE, we let $\nu^\gamma$ be the GMC measure associated to the field $G$. 

\begin{theorem} \label{thm:sine_GMC}
Let $w\in L^1\cap L^{\infty}(\R)$, $\phi\in\mathcal{D}$ be a function which satisfies Assumption~\ref{phi_assumption}, and let $\epsilon_N$ be a sequence which converges to 0 as $N\to\infty$ in such a way that $\epsilon_N \ge N^{-1}(\log N)^{1+\kappa}$ for some $\kappa>0$.   For any $0<\gamma<\sqrt{2}$, $\mu_{N,\epsilon_N}^{\gamma,\mathrm{Sine}}(w)$ converges in $L^1(\mathbb{P})$ as $N\to \infty$ to a random variable $\mu^\gamma(w)$ which has the same law as $\nu^\gamma(w)$, where $\nu^\gamma$ is the GMC measure defined in Theorem \ref{th:GMCintro} and $g(u,v) = \log|\ell^{2}-(u-v)^{2}|^{1/2}$. 
\end{theorem}

This shows that, in the subcritical phase, the law of the random measure $\mu^\gamma$ does not depend on the mollifier $\phi$ and it is the same GMC measure as for the CUE and as for Theorem \ref{th:GMCintro}. The proof of Theorem~\ref{thm:sine_GMC} is given at the end of Section~\ref{sect:L^1} and is a direct consequence of our general result, Theorem~\ref{thm:L^1_phase}. 
The main assumption to obtain Theorem~\ref{thm:sine_GMC} is again the strong Gaussian approximation \eqref{qmoms}. The needed exponential moments of \eqref{X_scaling} can be expressed in terms of a Fredholm determinant; see formula \eqref{Fredholm_det} below. In Section~\ref{sect:DPP}, we explain how the asymptotics of this Fredholm determinant can be related to a $2\times 2$ Riemann--Hilbert problem. Then, the condition $\epsilon_N \ge N^{-1}(\log N)^{1+\kappa}$ guarantees that we can straightforwardly apply the Deift--Zhou steepest descent method to this problem.

\begin{assumption}\label{h_assumption}
Let $0<\alpha<\pi$. Suppose that $h$ is a function which is analytic and satisfies  
\begin{equation*}
 |\Im h(z)| <  \alpha ,
 \end{equation*}
 in a strip $| \Im z | <\delta$.
We also assume that $h:\R \to\R$, that $h'\in L^1(\R)$, and that the following constants are finite:
\begin{equation} \label{constants_1}
C_\infty = \sup\big\{  \exp | h(z)| : | \Im z | \le\delta/2\big\}\
\text{ and }\
 C_1=   \sup\big\{ \| h\|_{L^\infty(\R+i s)} \vee  \| h\|_{L^1(\R+i s)}  : s \le\delta/2\big\} .  
\end{equation}
\end{assumption}

 We need to work with real-analytic mollifiers in order to apply the Riemann-Hilbert machinery. In principle, it could be possible to work with a more general class of mollifiers by using an argument analogous to the one in \cite{BerggrenD16}. It is based on constructing an $N$-dependent approximation $\phi^{(N)}$ of the mollifier $\phi$ which is real-analytic so that we can solve the Riemann-Hilbert problem for $\phi^{(N)}$, c.f.~Lemma~\ref{thm:RHP}, and argue that the Laplace transform of the two regularizations are sufficiently close as $N\to\infty$.

  \begin{proposition} \label{thm:sine_asymp}
If $h$ is a test function which satisfies Assumption~\ref{h_assumption}, then
\begin{equation}\label{Laplace_asymp}
\log \E{\exp\left( \sum_{\lambda\in\Lambda_N} h(\lambda)  \right)  }
  =    N \int  h(x) dx 
+\frac{1}{2}\| h \|_{H^{1/2}(\R)}^2
  +   \underset{N\to\infty}{O}\left(  \frac{C_\infty^6 C_1^2}{\delta^{3/2}|\sin\alpha|}  e^{-\pi \delta N} \right) .
  \end{equation}
	where $\| h \|_{H^{1/2}(\R)}^2$ is defined in Appendix \ref{appendix}, equation \eqref{IP_1}. The implied constant in the big-O term of \eqref{Laplace_asymp} does not depend on $h$, $\delta$ or $N$.
  \end{proposition} 
  The proof of Proposition~\ref{thm:sine_asymp} is given in Section~\ref{sect:RHP}. 
 
This is a classical result, \cite{Deift99}, but we need to take extra care to control the error term uniformly, especially because we will consider $N$-dependent test functions. 

Specifically, we can consider any regime where
$\delta(N) \to 0$ as $N\to\infty$ almost as fast as $N^{-1}$ (i.e.~up to the critical regime where the Gaussian approximation fails). 
In fact, our asymptotics are sufficiently strong to  strengthen the convergence of the measure $\mu_{N,\epsilon}^{\gamma,\mathrm{Sine}}$ when the parameter $\gamma$ is sufficiently small. In particular,  motivated by the conjectures of \cite{Ost16}, beyond the $L^2$-phase, we  establish the convergence of all the existing moments of the multiplicative chaos measure~$\mu_{N,\epsilon}^{\gamma,\mathrm{Sine}}$.

\begin{theorem} \label{thm:sine_moments}
Under the same assumptions as Theorem~\ref{thm:sine_GMC}, 
if $q$ is an even integer and $\gamma \ge 0$ so that $q\gamma^2<2$, then
 the random variable $\mu_{N,\epsilon}^\gamma(w)$ converges in $L^q(\mathbb{P})$ to $\mu^\gamma(w)$. Moreover, for any $q\in\N$ and  $\gamma \ge 0$ such that $q\gamma^2 <2$, we have
\begin{equation} \label{cvg_moments}
\lim_{N\to\infty}\E{\mu_{N,\epsilon}^{\gamma,\mathrm{Sine}}(w)^q} 
=  \int_{\R^q} \exp\Bigg(\gamma^2\hspace{-.15cm}\sum_{1\le j<k\le q}\hspace{-.15cm} Q(u_j-u_k) \Bigg)\prod_{k=1}^q w(u_k) du_k .
\end{equation}
where the correlation kernel $Q$ is given by formula \eqref{cov}.
\end{theorem}

The proofs of Theorem~\ref{thm:sine_GMC} and Theorem~\ref{thm:sine_moments} are given at the end of Sections~\ref{sect:L^1} and~\ref{sect:L^2} respectively.
Finally, let us mention that the main challenge to extend Theorem~\ref{th:gmc} beyond the CUE or sine process boils down to obtaining the strong Gaussian asymptotics \eqref{qmoms}. This is an interesting problem, also of independent interest, for both Hermitian unitary invariant or Wigner ensembles.

\subsection{Overview of the paper}

The paper is organized as follows. In Section~\ref{sect:DPP}, we begin with a brief review of determinantal point processes associated with integrable operators, of which the CUE and sine process are special cases. In Section~\ref{sect:L^2} and~\ref{sect:L^1}, we develop the multiplicative chaos theory for random fields which satisfy the strong Gaussian approximation \eqref{qmoms} and prove Theorem \ref{th:GMCgeneral}. In Section~\ref{sect:covariance}, we analyze the covariance structure of the regularized field $G_{\epsilon}(u)$ arising from the $H^{1/2}$ noise \eqref{Sos}, proving some preparatory results to apply the general theory to the CUE and sine process. In Sections~\ref{sect:RHP} and~\ref{sect:BO}, we establish the required asymptotics \eqref{qmoms} for the sine and CUE point processes, using a Riemann-Hilbert problem and formula \eqref{BOintro}, respectively. Finally, in the appendix (section~\ref{sect:GMC}), we provide a review of some of the recent developments in the theory of Gaussian multiplicative chaos of relevance to the present article. \\

In what follows, $C>0$ denotes a numerical constant which may change from line to line and we use the notation $a \ll  b$ to specify that the quantity $a \le Cb$. 
We also define for all $x\in\R$, 
\begin{equation}
\log^+(x)= \log\big(1 \vee |x|), \label{logplus}
\end{equation}
where throughout the article we use the notation $x\wedge y := \mathrm{min}\{x,y\}$ and $x \vee y := \mathrm{max}\{x,y\}$.
\subsection{Determinantal point processes and integrable operators} \label{sect:DPP}

The aim of this section is to provide,  in a general context,  a short introduction to the theory of determinantal point processes which focuses on the connection between linear statistics  and Fredholm determinants. We also briefly review the concept of integrable operators introduced in~\cite{IIKS} and
how this relates the Laplace transform of a linear statistic to a 
 Riemann-Hilbert problem.  \\

Let $\Sigma$ be a Polish space equipped with a Radon measure $\eta$.  A point configuration $\Upsilon\subset \Sigma$ is a discrete set  which is locally finite (i.e.~the set $\Upsilon \cap B$ is finite for any compact set $B\subset\Sigma$). A point process is a probability measure on the space of point configurations. This definition can be made mathematically precise, see for instance \cite{S00b, Johansson06, Borodin11}, and a point process can be described by its intensity measures or correlation functions $\{\rho_n\}_{n=1}^\infty$ which are defined by the formulae:
\begin{equation} \label{correlation_1}
\E{\sum_{ (\lambda_1,\dots, \lambda_n ) \subset \Upsilon} \prod_{k=1}^n f_k(\lambda_k)  } = \int_{\Sigma^n} \prod_{k=1}^n f_k(x_k)  \rho_n(dx_1,\dots, dx_n)  ,
\end{equation}
for any functions $f_1,\dots, f_n \in L^\infty(\Sigma\to \R_+)$ with compact support.  Note that the LHS of formula \eqref{correlation_1} consists of a sum over all ordered subsets of the random configuration $\Upsilon$ of size $n\in\N$.   
A point  process is  called determinantal if all its intensity measures are of the form
\begin{equation}\label{correlation_2}
 \rho_n(dx_1,\dots, dx_n)  = \det_{n\times n}[K(x_i,x_j)] \eta(dx_1)\cdots \eta(dx_n).
\end{equation}
The function $K :\Sigma\times \Sigma \to \C $ is called the correlation kernel. It is obviously not unique, but it encodes the law of the random configuration $\Upsilon$. There are many interesting examples of determinantal processes coming from probability theory, combinatorics, and mathematical physics such as the eigenvalues of unitary invariant random matrices, free fermions, zeros of Gaussian analytic functions, non-intersecting random walks, uniform spanning trees, random tilings, etc.
We refer to the surveys \cite{Johansson06, HKPV06, Borodin11} for further examples. 
Let us just mention the following criterion which goes back to the beginning of the theory, \cite{Macchi75},  and describes a natural class of correlation kernels. 

\begin{theorem}[Macchi \cite{Macchi75}, Soshnikov \cite{S00b}]
\label{thm:Macchi}
If a kernel $K$ determines a self-adjoint integral operator acting on $L^2(\Sigma)$ which is locally trace-class, then $K$ defines a determinantal point process if and only if its spectrum is contained in $[0,1]$.
\end{theorem}

In the following, we shall assume that the kernel $K$ is a continuous function  on $\Sigma\times\Sigma$ and satisfies the hypothesis of Theorem~\ref{thm:Macchi}. 
In this case, this kernel defines an operator, also denoted by $K$, which is locally trace-class if and only if for any compact set $B \subseteq \Sigma$,
$$
\tr K =  
\int_B K(x,x)dx <\infty ;
$$
 see \cite[Theorem~2.12]{Simon05}. 
We let $\Upsilon$ be the point configuration of the determinantal process with kernel $K$ and for any function $\varphi \in L^\infty(\Sigma \to \R_+)$, we denote
\begin{equation} \label{kernel}
K_\varphi(x, x') = \sqrt{\varphi(x)} K(x,x') \sqrt{\varphi(x')} . 
\end{equation}
The condition that the function $\varphi \ge 0$ is not necessary but rather convenient. 
In particular, this implies that the operator $K_\varphi$ is also self-adjoint, non-negative, and it is trace class if
\begin{equation} \label{expectation}
\tr K_\varphi = 
\int_{\Sigma} K(x,x) \varphi(x)dx
= \E{\sum_{\lambda\in\Upsilon} \varphi(\lambda)}  <\infty .
\end{equation}
Note that this condition holds if for instance $ \varphi$ has compact support. 
The last equality in \eqref{expectation} follows from the definition of the first intensity measure and the function $x\mapsto K(x,x)$ is called the density of the point process $\Upsilon$. 
The reason to consider the kernel \eqref{kernel} is that, using formulae \eqref{correlation_1} and \eqref{correlation_2}, it is a simple combinatorial exercise to show that if $K_\varphi$ is trace-class, then  for any $t\ge 0$, 
\begin{equation} \label{Fredholm_det}
\E{\prod_{\lambda\in\Upsilon}\big(1+t  \varphi(\lambda)\big)} = \det[\Id +t K_\varphi]_{L^2(\Sigma)} ,
\end{equation}
where the RHS is a Fredholm determinant; c.f.~\cite{Johansson06}. 
In particular, taking the usual logarithm, we obtain 
\begin{equation} \label{trace}
\log \E{\prod_{\lambda\in\Upsilon}\big(1 + t\varphi(\lambda)\big)} = \tr \log(\Id + t K_\varphi) ,
\end{equation}
and this function is differentiable for all $t>0$:
$$ \frac{d}{dt}  \tr \log(\Id + t K_\varphi) = \tr \left[ \frac{K_\varphi}{\Id+ tK_\varphi} \right] . $$ 
Hence, if we define $ \Res_{t}  := \frac{K_\varphi}{1+tK_\varphi}$, this implies that
\begin{equation} \label{resolvent}
\log \E{\prod_{\lambda\in\Upsilon}\big(1+\varphi(\lambda)\big)} 
= \int_0^1 \tr[\Res_t ] dt . 
 \end{equation} 

For instance, taking $\varphi(x) = \1_{x\in B}$ for some compact subset $B \subseteq \Sigma$, one can investigate the distribution of
the random variable $|\Upsilon\cap B|$ and in particular the probability that there are no points in the set $B$. 
More generally,  if $h \in L^\infty( \Sigma \to \R_+)$, taking $\varphi(x) =  e^{ h(x)} -1$, this gives an explicit formula for the exponential moments or Laplace transform of the linear statistic $\sum_{\lambda\in\Lambda} h(\lambda)$. 
As the density of the point process  converges to infinity, this reduces the question about the statistical properties of the random variable $\sum_{\lambda\in\Upsilon} h(\lambda)$ to a question about the 
 asymptotics of the resolvent operator $\Res_t$.
There is a special class of determinantal processes, those for which the correlation kernel gives rise to an integrable operator, which are particularly interesting because computing the resolvent $\Res_t$ turns out to be equivalent to solving a Riemann-Hilbert problem; see Proposition~\ref{thm:Deift} below. In particular, it allows to use the so-called {\it Deift-Zhou steepest descent} method introduced in \cite{DZ93} to obtain the asymptotics of formula~\eqref{resolvent}. 
The theory of integrable operators and the auxiliary Riemann-Hilbert problem originates in the context of statistical field theory, \cite{IIKS}, but this approach has also been used to answer different types of questions about the statistics of eigenvalues of unitary invariant matrix ensembles. For instance, one can find a proof of the Strong Szeg\H{o} limit theorem in \cite{Deift99} and, in \cite{BerggrenD16}, the authors extended Deift's method  to  investigate a transition for smooth mesoscopic statistics of the so-called thinned CUE and thinned sine process. Mesoscopic statistics were also studied using a Riemann-Hilbert problem in \cite{FKS16}.

 In this paper, we will use an analogous method to derive the necessary estimates to construct a multiplicative chaos measure which arises naturally from the sine process. In particular, we will make use of the following result from \cite{IIKS}, see also \cite{Deift99}. 
 
\begin{theorem}\label{thm:Deift}
Suppose that $\Sigma$ is a closed (oriented) curve on the Riemann sphere. 
Let $\Upsilon$ be a determinantal process on $\Sigma$ with Hermitian correlation kernel of the form
\begin{equation}\label{integrable_kernel}
K(z,z') =- \frac{\f(z)^*\g(z')}{\pii (z-z')} , 
\end{equation}
where $\f : \Sigma\to \C^k $ and $\g : \Sigma\to \C^k $ are continuously differentiable functions so that $\f(z)^*\g(z)=0$ for all $z\in\Sigma$. 
If $\varphi: L^\infty(\Sigma\to \R)$ is a test function so that both $\sqrt{\varphi}\f,  \sqrt{\varphi}\g \in L^2(\Sigma)$, then we have
\begin{equation}\label{Laplace_T}
\log \E{ \prod_{\lambda\in\Upsilon}\big(1+ \varphi(\lambda) \big)  }
= \frac{-1}{\pii} \int_0^1 \int_\R \left(\frac{d\sqrt{\varphi}\mathrm{F}_t}{dx}  (x)\right)^* \left(\sqrt{\varphi(x)} \mathrm{G}_t(x)\right) dx dt
\end{equation}
where $\mathrm{F}_t= m_+\f $ and $\mathrm{G}_t = (m_+^{-1})^* \g$ and the matrix $m$ is the (unique) solution of the Riemann-Hilbert problem:
\begin{itemize}
	\item $m(z) $ is analytic on $\mathbb{C}\backslash\Sigma$.
	\item If we let $v = \Id + t \varphi \f \g^*$, then $m(z)$ satisfies the jump condition:
\begin{equation}  \label{RHP_0}
m_{+}(z) = m_{-}(z)v(z), \qquad z \in \Sigma .
\end{equation}
\vspace{-.7cm}
	\item $m(z) \to \Id$ as $z \to \infty$ in $\mathbb{C}\backslash\Sigma$. 
\end{itemize}
\end{theorem}

For instance, when $\Sigma=\R$,  $m_\pm(x)=\lim_{\delta\to0} m(x+i\delta)$ denotes the boundary value  
of the matrix $m$ and are typically assumed to be continuous functions.
 In practice,  $\Sigma=\{|z|=1\}$ for the CUE, or $\Sigma=\R$ for the sine process and  the eigenvalue processes coming from  unitary invariant ensembles of Hermitian matrices.
   Moreover,  the correlation kernels of these processes all give rise to integrable operators. According to formula \eqref{integrable_kernel},  we may choose for the CUE:
$$
\f(z)  = \begin{pmatrix}z^{N+1} \\ 1 \end{pmatrix}  
\quad\text{ and }\quad
 \g(z) = \begin{pmatrix}z^{N+1} \\ -1 \end{pmatrix}, \qquad |z|=1  
$$
 and for the sine process:
$$
\f(x) =\begin{pmatrix} e^{i \pi N x} \\  e^{-i \pi Nx}     \end{pmatrix}
\quad\text{ and }\quad
\g(x) = \begin{pmatrix} e^{i \pi N x} \\  -e^{-i \pi Nx}  \end{pmatrix}, \qquad x \in \mathbb{R}     .
$$
For this example, the Riemann-Hilbert problem \eqref{RHP_0} is solved in Section~\ref{sect:RHP} for a large class of analytic test functions $\varphi$; see in particular Proposition~\ref{thm:RHP} for the asymptotics of the solution. 
As a last comment about universality, if one considers the eigenvalues of an $N\times N$ Hermitian random matrix sampled according to the weight $e^{-N \tr V(H)}$ for a real-analytic external field $V:\R\to\R$, then, by the Christoffel-Darboux formula, the correlation kernel of the eigenvalue process is also integrable with
\begin{equation} 
\f_V(x) =\begin{pmatrix} \pi_N(x) \\  -\pii \gamma_{N-1}^2 \pi_{N-1}(x)    \end{pmatrix} e^{-NV(x)/2}
\hspace{.6cm} \text{and} \hspace{.6cm} 
\g_V(y) =  \begin{pmatrix}   - \pii \gamma_{N-1}^2 \pi_{N-1}(x)\\ \pi_{N}(x) \end{pmatrix} e^{-NV(x)/2}.
\end{equation}
where $\pi_N$ and $\pi_{N-1}$ are the monic polynomials with respect to the measure $e^{-NV(x)}dx$ on $\R$ of degree $N$ and $N-1$ respectively and 
$$
\gamma_{N-1}^{-2} = \int_{\mathbb{R}} \pi_{N-1}(x)^2 e^{-NV(x)} dx . 
$$ 
Observe that, apart from the weight $e^{-NV(x)/2}$, $\f_V$ is exactly the first column of the solution $Y_N$
of the orthogonal polynomial Riemann-Hilbert problem, whose solution is derived in great detail in \cite{DKMVZ}. 
In particular from their results, one can extract the universal oscillatory behavior of the functions $\f_V$ and $\g_V$ in the bulk, which indicates strongly that the approach we present for the sine process could be generalized and would provide a way to show that the limiting  chaos measure has the same law for a large class of potentials $V$.  
However, turning this heuristic into a rigorous computation is rather technical and we leave it as an open problem for future work.

  \section*{Acknowledgements}
 We thank Christian Webb for his interest in our work, for useful discussions and carefully reading the first draft of this paper which helped fix typos and mistakes in the proofs and greatly improved the presentation. 
G. L. thanks Juhan Aru for stimulating discussions about Gaussian multiplicative chaos. N. S. would like to thank Yan Fyodorov for suggesting to study the conjectures of \cite{Ost16} in the context of random matrix theory. Finally, we express our gratitude to the anonymous referees for their detailed comments and suggestions.

\section{Proof of the main results} \label{sect:proof}

Even though the applications discussed in this paper are concerned with random processes defined on $\R$, we will formulate our convergence results in an abstract setting under some general assumptions. Let $\A$ be a compact set in $\R^d$, $d\ge 1$. 
In certain cases, such as the sine process, one might also consider the case where $\A$ is not compact, this requires only slight modifications of our proof.
We consider a real-valued  generalized Gaussian process $G$ defined on $\A$  with a covariance kernel:
$$T(x,y) = - \log|x-y| + g(x,y) ,  $$
where the function $g:\A^2\to\R\cup\{-\infty\}$ is continuous and such that there exists a constant $C>0$ so that for all $x,y\in \R$,
\begin{equation*}
 g(x,y) -\log^+(x-y)   \le C .
\end{equation*}

We also consider a family of real-valued random fields 
$X_{N,\epsilon}(u)$ defined on $\A$ which are centered, depend on two parameters $N,\epsilon>0$, and behave asymptotically like the Gaussian process~$G$. Specifically, we should assume that for any $N>0$,  $\big(u, \epsilon \mapsto X_{N,\epsilon}(u)\big)_{u\in \A , \epsilon>0}$ are random processes defined on the same probability space and that they satisfy Assumptions~\ref{finite_dist} -- \ref{exp_moments} below. For any $\gamma\in\R$,  we consider the normalized process
\begin{equation} \label{normalization}
\X_{N,\epsilon}^\gamma(u) =\gamma X_{N,\epsilon}(u)  -\frac{\gamma^2}{2} \E{X_{N,\epsilon}(u)^2} ,
\end{equation} 
and our goal is to construct the limit of the random measure 
 $$
\mu_{N,\epsilon}^\gamma(du) = \exp\big(\X_{N,\epsilon}^\gamma(u)\big)du
$$
when $\epsilon_N \to 0$ as $N\to\infty$ sufficiently slowly.
To begin with, in Section~\ref{sect:L^2}, we shall prove that $\mu_{N,\epsilon}^\gamma$ converges to a GMC measure in the $L^2$-phase ($\gamma<\sqrt{d}$) and compute the limit of its moments in view of proving Theorem ~\ref{th:moments}. 
Then, in Section~\ref{sect:L^1}, we tackle the more challenging task of showing that $\mu_{N,\epsilon}^\gamma$ converges in the whole subcritical regime ($\gamma<\sqrt{2d}$), thus establishing the proof of Theorem~\ref{th:GMCgeneral}.

\subsection{Convergence of the multiplicative chaos measure in the $L^2$-phase}\label{sect:L^2}

\begin{assumption}[Finite-dimensional convergence in the weak regime] \label{finite_dist} 
For any given $\epsilon, \delta>0$,
we have 
$$
\lim_{N\to\infty} \E{X_{N,\epsilon}(u)X_{N,\delta}(v)} =
T_{\epsilon, \delta}(u,v), 
$$
and the field $u\mapsto X_{N,\epsilon}(u)$ converges in the sense of finite-dimensional distributions to a mean-zero Gaussian process $G_\epsilon$ with covariance structure: 
\begin{equation}\label{regularized_kernel}
T_{\epsilon, \delta}(u,v) = \E{G_\epsilon(u)G_\delta(v) }, 
\end{equation}
for any $u,v\in \A$ and $\epsilon, \delta>0$.
\end{assumption}

In the context of random matrix theory described in the introduction, Assumption~\ref{finite_dist} follows from the CLT for smooth linear statistics and $G_\epsilon = G*\phi_\epsilon$ for some  nice mollifier $\phi$.
In this abstract context, $G_\epsilon$ is a $d+1$ dimensional Gaussian field which is a smooth approximation of a log--correlated $G$ coming from a possibly different regularization procedure. To construct a multiplicative chaos measure out of the field $X_{N,\epsilon}$, one also needs the existence of the GMC measure $\nu^\gamma$ associated with the field $G$. 
As discussed in the proof of Proposition~\ref{L^2_GMC} and Remark~\ref{rk:Berestycki} below it, 
the convergence  follows from the following conditions for the correlation kernels. 

\begin{assumption}[Covariance kernel asymptotics] \label{kernel_asymp}
 Suppose that for all $(u,v)\in \A^2$, 
\begin{equation}  \label{domination}
T_{\epsilon,\delta}(u,v) \le \log^+\left(|u-v|^{-1}\wedge \epsilon^{-1}\wedge\delta^{-1}\right) + C ,
\end{equation}
 and that for almost all $(u,v)\in\A^2$,
\begin{equation} \label{assumption_2}
T_{\epsilon,\delta}(u,v) \to T(u,v) \quad\text{ as }\epsilon,\delta \to0.
 \end{equation}
We also suppose the bound \eqref{domination} is sharp,  in the sense that if $\epsilon\ge \delta\ge 0$ and  $|u-v| \le \exp(- \epsilon^{-1})$, then 
\begin{equation}  \label{assumption_5}
T_{\epsilon,\delta}(u,v) = \log\epsilon^{-1} + \underset{\epsilon\to0}{O}(1) .
\end{equation}
 \end{assumption}

For any $\gamma, \epsilon>0$, let 
$$
\nu^\gamma_{\epsilon}(dx) = \exp\left(\gamma G_\epsilon(x) - \frac{\gamma^2}{2} \E{G_\epsilon(x)^2}\right) dx .
$$
The assumption~\ref{kernel_asymp} guarantees that for any $q \in\N$ such that $\gamma^2q<2d$ and for any $w\in L^1\cap L^\infty(\A)$, the random variable $\nu^\gamma_{\epsilon}(w)$  converges in $L^{q}(\mathbb{P})$  as $\epsilon\to0$. The purpose of the next assumption is to identify the limit.

\begin{assumption}[Convergence of the GMC measure]  \label{GMC_existence}
For any $\gamma<\sqrt{2d}$, let $\nu^\gamma$ be the GMC measure associated with the field $G$ as defined in Theorem~\ref{th:GMCintro}. Then, $\nu^\gamma_{\epsilon} \Rightarrow \nu^\gamma$ as $\epsilon \to 0$.
\end{assumption}

Finally, in order to apply  the {\it second moment method} considered by Webb in \cite{W15}, we will also need to control some exponential moments of the field $(u, \epsilon) \mapsto \X_{N,\epsilon}(u)$. The idea of \cite{W15}  consists in proving that both in the weak regime (when we consider successive limits as $N\to\infty$ and $\epsilon\to0$) and in the strong regime (when $\epsilon_N\to0$ as $N\to\infty$), the limiting random measures coincide and have the same law as the GMC measure $\nu^\gamma$. In particular, we will need the following asymptotics.
For any $q\in\N$ and  $\delta>0$, define
\begin{equation}\label{simplex}
\triangle_q(\delta)=\big\{ \boldsymbol{\epsilon}\in\R^q : \epsilon_1 \ge \cdots \ge \epsilon_q \ge \delta\big\}  .
\end{equation}

\begin{assumption}[Exponential moments asymptotics] \label{exp_moments}
Let $q\in\N$.  We suppose that there exists a sequence $\delta_N\to0$ as $N\to\infty$ so that for any $\boldsymbol{\epsilon}\in \triangle_q(\delta_N)$ and for any  $\mathbf{t} \in\R^q$,  we have   uniformly for all $\u\in \A^q$,   
\begin{equation}  \label{assumption_4}
\log \E{\exp\left( \sum_{k=1}^{q}  \X^{t_k}_{N,\epsilon_k}(u_k) \right)  }
  = 
\hspace{-.1cm}   \sum_{1\le k<j\le q} \hspace{-.1cm}  
t_k t_j T_{\epsilon_j, \epsilon_k}(u_j,u_k)
  +   \underset{N\to\infty}{o(1)}  .
    \end{equation}
\end{assumption}

In the CUE case (Theorem~\ref{th:gmc}), at any mesoscopic scale $0<\alpha<1$,  one may choose the parameter  $\delta_N = N^{\alpha-1+\kappa}$ for some small $\kappa>0$ so that the condition \eqref{c1} is satisfied. On the other hand, for the sine process with density $N$, we may choose $\delta_N=N^{-1}(\log N)^{1+\kappa}$ for any $\kappa>0$.

\begin{theorem} \label{thm:L^2_phase}
Suppose that Assumptions \ref{finite_dist} -- \ref{GMC_existence} hold, as well as Assumption~\ref{exp_moments} for $q=1,2$. 
If  $\gamma^2<d$,  then for any  $w\in L^1\cap L^\infty(\A)$,  the random variable $\mu_{N,\delta_N}^\gamma(w)$ converges in distribution as $N\to \infty$ to $\nu^\gamma(w)$.
Moreover, if Assumption~\ref{exp_moments} is also satisfied for  $q\in\N$ and $\gamma^2q< 2d$, then
\begin{equation} \label{moments_limit}
\lim_{N\to\infty}
\E{\mu_{N,\delta_N}^\gamma(w)^q} 
=  \int\exp\Bigg(\gamma^2\hspace{-.15cm}\sum_{1\le j<k\le q}\hspace{-.15cm} T(u_j,u_k)  \Bigg)\prod_{k=1}^q w(u_k) du_k .
\end{equation}
\end{theorem}

As mentioned in the introduction, the condition $\gamma^{2}q < 2d$ for the existence of the limiting moments \eqref{moments_limit} is sharp and is related to the convergence of certain multiple integrals, which in case $d=1$ are related to Selberg integrals.
The remainder of this section is devoted to the proof of Theorem~\ref{thm:L^2_phase} and of an extension, Lemma~\ref{thm:L^q_weak}, in the case where there is a coupling of the fields $X_{N,\epsilon}(u)$ for different $N>0$. 
This applies for instance to the sine process discussed in Section~\ref{sect:sine}. 
 Note that to prove the convergence of  the measure $\mu^\gamma_{N,\delta_N}$ and its moments in the $L^2$-phase, it is clear from our assumptions that one can use the same argument as in the proof of Proposition~\ref{L^2_GMC}. However, to identify that this measure has the same distribution as $\nu^\gamma$, it is simpler to first establish that $\mu_{N,\epsilon}^\gamma$ converges in distribution in the weak regime 
   and then to show that the strong and weak limits  coincide; c.f.~Lemma~\ref{thm:weak_cvg} and Lemma~\ref{thm:L^2_Cauchy} respectively. 

\begin{lemma}\label{thm:weak_cvg}
Suppose that $\A$ is compact in $\R^{d}$. Let $w \in L^{1}(\A)$ and fix $\epsilon>0$. 
For any $\gamma \ge 0$, the random variable $\mu_{N,\epsilon}^\gamma(w)$ converges in distribution as $N\to\infty$ to $\nu^\gamma_{\epsilon}(w)$.   
\end{lemma}

\proof
By assumption \ref{finite_dist} and continuity of the exponential function, the finite dimensional distributions of the process $\xi_{N,\epsilon}(u) = \exp\big(\X^{\gamma}_{N,\epsilon}(u)\big)$ converge to those of $\xi_{\epsilon}(u) = \exp\big( \widetilde{G}^\gamma_{\epsilon}(u)\big)$ as $N \to \infty$ where
$$
\widetilde{G}^\gamma_{\epsilon}(u) = \gamma G_{\epsilon}(u)-\frac{\gamma^{2}}{2}\mathbb{E}(G_{\epsilon}(u)^{2}) . 
$$
 We also claim that $\xi_{N,\epsilon}(u)$ is tight in $L^{1}(\A, |w(u)|du)$, so that, by Prokhorov's theorem, $\xi_{N,\epsilon}\Rightarrow \xi_{\epsilon}$ as $N \to \infty$.
The tightness follows from a criteria established in \cite{CKtight86} which shows that when $\A$ is compact it suffices that there exists a constant $C_\epsilon>0$ so that
 \begin{equation} \label{tightness_assumption}
\sup_{u\in\A} \E{|\xi_{N,\epsilon}(u)|} \le C_\epsilon .
 \end{equation}
Notice that, since $G_\epsilon$ is a Gaussian process,  $ \E{|\xi_{\epsilon}(u)|} =1$
and the estimate \eqref{tightness_assumption} follows directly from Assumption \ref{exp_moments}. Since the functional $\xi \to \int \xi(u)w(u)du$ is obviously continuous on $L^{1}(\A, |w(u)|\,du)$, we conclude that as $N\to\infty$, 
$$ \mu_{N,\epsilon}^\gamma(w) = \int \xi_{N,\epsilon}(u)w(u)du 
\Rightarrow \nu_{\epsilon}^\gamma(w) =\int \xi_{\epsilon}(u)w(u)du  . \qed$$ 

\begin{remark} This proof relies on the fact that the sequence $(\xi_{N,\epsilon})_{N>0}$ is tight in $L^{1}(\A, |w(u)|du)$ for any $\epsilon>0$ and that the condition \eqref{tightness_assumption} is straightforward to check. However, for the CUE or sine statistics (c.f.~\eqref{smoothing} and \eqref{X_field} respectively), using the specific form of the test function $\J_u \ast \phi_{\epsilon}$, it is also possible to verify that the criterion (4) of \cite[Theorem~16.5]{Kallenberg} holds, which implies that $X_{N,\epsilon} \Rightarrow X_{\epsilon}$ as random elements of $C(\A\to\R)$. 
\end{remark}

\begin{lemma}\label{thm:L^2_Cauchy}
 Let $q$ be an even integer such that $\gamma^{2}q<2d$. Then, for any $w\in L^1\cap L^\infty(\A)$, 
\begin{equation*} 
\lim_{\epsilon\to0, N\to\infty } \E{\left| \mu_{N,\delta_N}^\gamma(w) -\mu_{N,\epsilon}^\gamma(w) \right|^q} =0 .
\end{equation*}
\end{lemma}

\proof  Suppose that $\delta_N\le \epsilon$. 
For any $0\le i \le q$, define 
$$
\boldsymbol{\epsilon}^i = \big(\underbrace{\epsilon,\dots,\epsilon}_{ i \#},
\underbrace{\delta_N,\dots,\delta_N}_{ (q-i) \# }\big) 
$$
and let for any $\u\in\A^q$,
$$ \Theta_{\epsilon,N}^i(\u) =
\hspace{-.2cm}   \sum_{1\le k<j\le q} \hspace{-.2cm}  
T_{\epsilon_j^i,\epsilon_k^i}(u_j, u_k)  .
$$
By Fubini's theorem and Assumption~\ref{exp_moments} with $\mathbf{t} =(\gamma, \dots, \gamma) $, we obtain
\begin{equation}  \label{integral_3}
\E{\left| \mu_{N,\delta_N}^\gamma(w) -\mu_{N,\epsilon}^\gamma(w) \right|^q}
= \sum_{i=0}^q (-1)^i \binom{q}{i} 
\int_{\A^q}  \exp\left(\gamma^2 \Theta^i_{\epsilon,N}(\u) +  \underset{N\to\infty}{o(1)} \right)
 \prod_{k=1}^q w(u_k) du_k  ,
\end{equation}
where the error term is uniform.
Moreover, the condition \eqref{assumption_2} implies that  for any $i\in[q]$ and for almost all $\u\in \A^q$, 
\begin{equation*} 
 \lim_{\epsilon\to0, N\to\infty }  \Theta_{\epsilon,N}^i(\u) = 
  \hspace{-.2cm}   \sum_{1\le k<j\le n} \hspace{-.2cm}   
   T(u_j,u_k) .
\end{equation*}
Finally, the condition~\eqref{domination} shows that for all $\u\in \A^{q}$, 
\begin{equation} \label{domination'}
\exp\bigg( \gamma^2 \Theta_{\epsilon,N}^i(\u)  \bigg) \ll
\hspace{-.2cm} \prod_{1\le j<k\le q}\hspace{-.15cm}  1\vee |u_j-u_k|^{-\gamma^2} .
\end{equation}
Hence, since the RHS of \eqref{domination'}
is locally integrable on $(\R^d)^q$  when $\gamma^2q< 2d$, 
by the dominated convergence theorem, the integrals on the RHS of formula \eqref{integral_3} 
converge for all $i \in [q]$ to the same finite value while taking the limit as $N\to\infty$ and then as $\epsilon\to0$.  Since $\sum_{i=0}^q (-1)^i \binom{q}{i} =0$, this proves the claim.\qed\\

\noindent
{\it Proof of Theorem~\ref{thm:L^2_phase}.}
By Lemma~\ref{thm:weak_cvg}, for any $\epsilon>0$, we have 
 $\mu_{N,\epsilon}^\gamma(w) \Rightarrow \nu^\gamma_{\epsilon}(w)$ as $N\to\infty$ and Assumption~\ref{GMC_existence} guarantees that $\nu^\gamma_{\epsilon}(w) \Rightarrow \nu^\gamma(w) $
as $\epsilon\to0$. 
Hence, by  \cite[Theorem~4.28]{Kallenberg} and Lemma~\ref{thm:L^2_Cauchy} with $q=2$, this implies that  
$ \mu_{N,\delta_N}^\gamma(w) \Rightarrow   \nu^\gamma(w)$ as $N\to\infty$.
To complete the proof, it remains to establish convergence of the moments of the random variable $\mu_{N,\delta_N}^\gamma(w)$. 
We proceed as in the proof of Lemma~\ref{thm:L^2_Cauchy}.
By Fubini's theorem, for any $q\in\N$ and $\epsilon>0$, we have
\begin{equation}\label{Fubini}
\E{\mu_{N,\epsilon}^\gamma(w)^q} 
= \int_{\R^q} \E{\exp\bigg( \sum_{k=1}^q \X_{N,\epsilon}^\gamma(u_j)\bigg)}\prod_{k=1}^q w(u_k) du_k .
\end{equation}
Then, Assumptions~\ref{kernel_asymp} and \ref{exp_moments} imply that for almost all $\u\in \A^q$, 
\begin{equation}  \label{assumption_3}
\lim_{\begin{subarray}{c} N\to\infty \\ \epsilon\to0 \end{subarray}} \E{\exp\bigg( \sum_{k=1}^q \X_{N,\epsilon}^\gamma(u_j)\bigg)} = \exp\Bigg(\gamma^2\hspace{-.15cm}\sum_{1\le j<k\le q}\hspace{-.15cm} \E{G(u_j) G(u_k)} \Bigg) 
\end{equation}
in both the weak and strong regime (as long as $\epsilon(N) \ge \delta_N$). 
Moreover, the condition~\eqref{domination} guarantees that for all $\u\in \R^q$, 
\begin{equation*}
\E{\exp\bigg( \sum_{k=1}^q \X_{N,\epsilon}^\gamma(u_j)\bigg) }\le C
\hspace{-.2cm} \prod_{1\le j<k\le q}\hspace{-.15cm}  1\vee |u_j-u_k|^{-\gamma^2} .
\end{equation*}
Hence, if $\gamma^2q< 2$, formula~\eqref{moments_limit} follows directly from \eqref{Fubini}--\eqref{assumption_3} and the dominated convergence theorem.\qed\\

Note that in the context of Theorem~\ref{thm:L^2_phase} we did not require that the  fields $X_{N,\epsilon}$, $X_{N+1,\epsilon}, \dots$ are defined on the same probability space. However, if such a coupling is available, as in the case of the sine process, then we can upgrade the topology of convergence in Theorem~\ref{thm:L^2_phase} by replacing Lemma~\ref{thm:weak_cvg} by the following result. 

\begin{lemma} \label{thm:L^q_weak}
Using the notation \eqref{X_field}--\eqref{X_measure} where $\Lambda_1$ is the sine process. 
For any $w\in L^1\cap L^\infty(\R)$ and for any $q\ge 1$, the random variable $\mu_{N,\epsilon}^{\gamma,\mathrm{Sine}}(w)$ converges  in  $L^q(\mathbb{P})$   as $N\to\infty$ to a limit $\mu^{\gamma}_\epsilon(w)$ whose law is the same as $ \nu^\gamma_\epsilon(w)$.   
\end{lemma}

Before proving Lemma~\ref{thm:L^q_weak}, let us first use the previous results to obtain  Theorem~ \ref{thm:sine_moments}. \\

\noindent{\it Proof of  Theorem~ \ref{thm:sine_moments}.}
Let $\delta_N = N^{-1}(\log N)^{1+\kappa}$ where $\kappa>0$,  $\phi$ be a function which satisfies Assumption~\ref{phi_assumption},  and define for any  $\mathbf{t}\in\R^n$, $\u\in\R^n$ and $\boldsymbol{\epsilon} \in \triangle_n(\delta_N) $,
\begin{equation*} 
h_{\u, \boldsymbol{\epsilon}}(z) = \pi
\sum_{k=1}^n t_k \int_\R \1_{|x-u_k|\le \ell/2}\ \phi_{\epsilon_k}(z-x) dx . 
\end{equation*}
This function is analytic in  $|\Im z| \le c \delta_N $
and we claim that it satisfies  Assumption~\ref{h_assumption} in this strip with 
$$
C_\infty(h_{\u, \boldsymbol{\epsilon}}) = e^{\pi C_\phi |\mathbf{t}|}
\quad\text{ and }\quad
C_1(h_{\u, \boldsymbol{\epsilon}}) =  \pi C_\phi |\mathbf{t}| (\ell \vee 1) ,
$$
where $|\mathbf{t}| = |t_1|+\cdots+|t_n|$. 
To check this assumption, we can use the bounds:
$$
\left|\int_\R \1_{|x-u|\le \ell/2}\ \phi_{\varepsilon}(z-x) dx \right|
\le \int_\R \big| \phi( x+ i\Im z) \big| dx 
$$
and 
$$
\iint_{\R\times\R} \1_{|x-u|\le \ell/2}\ \big|\phi_{\varepsilon}(t-x +is)\big|dx  dt
\le \ell \int_\R \big| \phi( x+ is) \big| dx ,
$$
which hold for any $u\in\R$ and $\varepsilon>0$. 
This implies that we can apply Proposition~\ref{thm:sine_asymp}.
Moreover,  since the sine process  has constant density $N$ on $\R$, the leading term in formula~\eqref{Laplace_asymp}  corresponds  to the expected value of the linear statistic $\sum_{\lambda\in\Lambda_N} h_{\u, \boldsymbol{\epsilon}}(\lambda)$, and by definition of the $H^{1/2}$ Gaussian noise $\Xi$ (see Appendix~\ref{appendix}), the second order term corresponds to the variance of the Gaussian random variable $\Xi(h_{\u, \boldsymbol{\epsilon}})$.
Thus, we get
\begin{equation*} 
\log \E{\exp\left( \sum_{\lambda\in\Lambda_N} h_{\u, \boldsymbol{\epsilon}}(\lambda)  \right)  }
  =     \E{ \sum_{\lambda\in\Lambda_N} h_{\u, \boldsymbol{\epsilon}}}
+\frac{\E{\Xi\big(h_{\u, \boldsymbol{\epsilon}})^2}}{2}
  +   \underset{N\to\infty}{O}\left(\delta_N^{-3/2}   (|\mathbf{t}|\ell)^2 e^{\pi (6C_\phi|\mathbf{t}|-\delta_N N)} \right) . 
  \end{equation*}
  By definition of the random field \eqref{X_field} and using the scaling property 
  \eqref{X_scaling}, we have the equality in law
$$\sum_{\lambda\in\Lambda_N} h_{\u, \boldsymbol{\epsilon}}(\lambda) \overset{d}{=} \sum_{k=1}^n t_k X_{N, \epsilon_k}(u_k) , $$
and using the 
representation~\eqref{G_regularization}, we also have
$\Xi(h_{\u, \boldsymbol{\epsilon}})  \overset{d}{=}  \sum_{k=1}^n G_{\phi, \epsilon_k}(u_k)$. 
Hence,
in the regime where $|\mathbf{t}|$ and $\ell>0$ are independent of the parameter $N$, this implies that for any $\beta<1+\kappa$ and for any $\boldsymbol{\epsilon} \in \triangle_n(\delta_N) $ with $\delta_N = (\log N)^{1+\kappa}/N$, we have
  \begin{equation} \label{full_asymp}
\log \E{\exp\left( \sum_{k=1}^n t_k \X_{N,\epsilon_k}(u_k) \right)  }
  = 
  \sum_{1 \le k< j \le n}  t_kt_j
\E{G_{\phi, \epsilon_j}(u_j)G_{\phi, \epsilon_k}(u_k)  }
  +   \underset{N\to\infty}{O}\left(e^{-(\log  N)^\beta} \right) , 
  \end{equation}
  uniformly for all $\u\in\R^n$.
In particular, this  immediately shows that the random field $u\mapsto X_{N, \epsilon}(u)$ 
  satisfies Assumptions~\ref{finite_dist} and~\ref{exp_moments} for all $q\in\N$. 
  Moreover, if the mollifier $\phi \in \mathcal{D}$, by Corollary~\ref{covariance_assumptions}, Assumption~\ref{kernel_asymp} holds too.
  Consequently, by Theorem~\ref{thm:L^2_phase},  we obtain the convergence of the moments, formula~\eqref{cvg_moments}. Then, if $w\in L^1 \cap L^\infty(\R) $ and $q$ is an even integer, according to Lemma~\ref{thm:L^q_weak}, we have  for any given $\epsilon>0$, $\mu_{N,\epsilon}^{\gamma,\mathrm{Sine}}(w)$ converges in $L^q(\mathbb{P})$ as $N\to\infty$ to the random variable  $\mu_{\epsilon}^{\gamma}(w)$.
In addition, if $\gamma^2q<2$, by Proposition~\ref{L^2_GMC} and Remark~\ref{rk:Berestycki} below it, the random variable $\mu^\gamma_\epsilon(w)$ constructed above converges in $L^q(\mathbb{P})$ as $\epsilon\to0$  to a random variable $\mu^\gamma(w)$ which has the same law as $\nu^\gamma(w)$. In other words, we have
\begin{equation*} 
\lim_{\epsilon\to0, N\to\infty } \E{\left| \mu_{N,\epsilon}^{\gamma,\mathrm{Sine}}(w) -\mu^\gamma(w) \right|^q} =0 .
\end{equation*}
On the other hand, by Lemma~\ref{thm:L^2_Cauchy} and the triangle inequality, this implies that
\begin{equation*} 
\lim_{N\to\infty } \E{\left| \mu_{N,\delta_N}^{\gamma,\mathrm{Sine}}(w) -\mu^\gamma(w) \right|^q} =0 
\end{equation*}
which concludes the proof of Theorem~\ref{thm:sine_GMC}. \qed\\

\noindent{\it Proof of Lemma~\ref{thm:L^q_weak}.}
To keep the notation simple, we will prove the Proposition when $q=2$ which is the most interesting case. It is straightforward to generalize the argument to any even $q$.
 As in the proof of Theorem~\ref{thm:sine_GMC} (c.f.~formula \eqref{full_asymp}), it is easy to check that the asymptotics of Proposition~\ref{thm:sine_asymp} implies that for any $\eta, \eta' \in \{0,1\}$, 
    \begin{equation*} 
\log \E{\exp\left(  \X_{N+\eta,\epsilon}^\gamma(u) +   \X_{N+\eta',\epsilon}^\gamma(v)  \right)  }
  =   \gamma^2 T_{\epsilon,\epsilon}(u,v)
  +   \underset{N\to\infty}{O_{\eta,\eta'}}\left(e^{-(\log  N)^\beta} \right) 
  \end{equation*}
 uniformly for all $u,v\in\R$, where 
$T_{\epsilon,\epsilon}(u,v) = \E{G_{\phi, \epsilon}(u)G_{\phi, \epsilon}(v) }$ and $\beta<1+\kappa$ .  
By expanding the square, this implies that 
\begin{align} 
\E{\left| \mu_{N,\epsilon}^{\rm \gamma, Sine}(w) -\mu_{N+1,\epsilon}^{\rm \gamma, Sine}(w) \right|^2} 
&\label{Cauchy_property}=
  \int_{\R}   e^{\gamma^2 
T_{\epsilon,\epsilon}(u, v)} \Big\{ \hspace{-.3cm} \sum_{\eta,\eta' \in \{0,1\}} 
(-1)^{\eta+\eta'}\exp \underset{N\to\infty}{O_{\eta,\eta'}}\left(e^{-(\log  N)^\beta} \right) \Big\} w(u)w(v)dudv .
\end{align}
Note that the leading terms on the RHS of formula \eqref{Cauchy_property} cancel and the error terms are uniform. Moreover, by  \eqref{domination},
$
T_{\epsilon,\epsilon}(u,v) \le \log^+(\epsilon^{-1}) + C 
$ on $\R^2$ and since $w\in L^1(\R)$, we obtain for any $\epsilon>0$, 
$$
\E{\left| \mu_{N,\epsilon}^{\rm \gamma, Sine}(w) -\mu_{N+1,\epsilon}^{\rm \gamma, Sine}(w) \right|^2} 
=\underset{N\to\infty}{O_{w,\epsilon}}\left(e^{-(\log  N)^\beta} \right) .
$$
Thus, since we may choose the parameter $\beta>1$, by the triangle inequality, 
this shows that $\big(\mu_{N,\epsilon}^{\rm \gamma, Sine}(w)\big)_{N>0}$ is a Cauchy sequence in $L^2(\mathbb{P})$. 
Let us denote by $\mu^\gamma_\epsilon(w)$ its limit.  
We will use Lemma~\ref{thm:weak_cvg} to identify the law of this limit. 
For any  $M>0$, we let $w_M(u) = w(u)\1_{\{ |u| \le M\}}$
 and 
$w_M^*(u) = w(u)\1_{\{ |u| > M\}} $.
On the one hand, by formula~\eqref{Fubini} and using Assumption \ref{exp_moments}  and the bound \eqref{domination}, we obtain
\begin{align*} 
\lim_{N\to\infty}\E{ \mu_{N,\epsilon}^{\gamma,\mathrm{Sine}}(w^*_M)^2} = \E{\mu_{\epsilon}^\gamma(w^*_M)^{2}}
\ll \epsilon^{-1}  \int_{\begin{subarray}{c} |u|>M \\ |v|>M \end{subarray}}  w(u) w(v) dudv . 
\end{align*}
In addition, we easily check that $\mu^\gamma_\epsilon(w) = \mu^\gamma_\epsilon(w_M) + \mu^\gamma_\epsilon(w_M^*) $  so that by taking the limit as $M\to\infty$, 
$$
\mu^\gamma_\epsilon(w_M) \Rightarrow \mu^\gamma_\epsilon(w) .
$$
An analogous computation shows that for the regularized GMC measure:
$  \nu^\gamma_\epsilon(w_M) \Rightarrow    \nu^\gamma_\epsilon(w)$ as $M\to\infty$.  
On the other hand, since the functions $w_M$ have compact support, by Lemma~\ref{thm:weak_cvg}, we know that
$\mu_{\epsilon}^\gamma(w_M) \overset{d}{=} \nu^\gamma_\epsilon(w_M)$ for any $M>0$ . Hence, we conclude that $\mu^{\gamma}_{\epsilon}(w) \overset{d}{=} \nu^\gamma_\epsilon(w)$. \qed

\subsection{Convergence of the multiplicative chaos measure in the $L^1$-phase}\label{sect:L^1}

We work in the same context as in the previous section and  the goal is to establish the convergence of the random measure 
$\mu_{N,\epsilon}^\gamma(du) = \exp\big(\X_{N,\epsilon}^\gamma(u)\big)du$ 
throughout the subcritical phase $0<\gamma<\sqrt{2d}$. 

\begin{theorem}\label{thm:L^1_phase}
Suppose that Assumptions~\ref{finite_dist}--\ref{cvg} hold for $q\le 3$ in \ref{exp_moments}.  If $w \in L^1\cap L^{\infty}(\A)$ and  $\gamma<\sqrt{2d}$, then for any sequence $\epsilon_N \to 0$ in such a way that $\epsilon_N \ge \delta_N$,   the random variable $\mu_{N,\epsilon_N}^\gamma(w)$ converges in distribution to  $\nu^\gamma(w)$ as $N\to \infty$. 
\end{theorem}

The condition $\gamma <\sqrt{2d}$ is sharp in the sense that we expect that
$\mu_{N,\epsilon}^\gamma(du) \to 0$ for any $\gamma \ge \sqrt{2d}$.
The proof of Theorem~\ref{thm:L^1_phase} follows the elementary argument introduced \cite{Berestycki15} to prove convergence of the GMC measure in the $L^1$-phase.  
In fact, the asymptotics \eqref{assumption_4} are so strong that Berestycki's method can be applied to the field 
$ u\mapsto X_{N,\epsilon}(u)$ modulo a few technical issues. The main idea stems from the fact that the measure $
\mu_{N,\epsilon}^\gamma$ is supported on the so-called $\gamma$-thick points, \eqref{gamma_points}. 
More specifically, we will proceed to show that, if the parameter $\alpha>\gamma$,  the mass
$$
\mu_{N,\epsilon_N}^\gamma\left( u\in\A : \frac{X_{N,\epsilon}(u)}{\log\epsilon^{-1}}  >\alpha \text{ for some } \epsilon \in \{ e^{-k} :  L  \le k \le \log \delta_N^{-1} \}\right)
$$
 converges to 0 in $L^1(\mathbb{P})$  as $N\to\infty$ and then $L\to\infty$. Then we will show that the random measures $\mu_{N,\epsilon_N}^\gamma$ restricted to the good set 
$ \left\{ u \in \A : \frac{X_{N,\epsilon}(u)}{\log\epsilon^{-1}}  \le \alpha \text{ for all } \epsilon \in \{ e^{-k} :  L  \le k \le \log \delta_N^{-1} \}\right\}$ form a Cauchy net in $L^2(\mathbb{P})$ (in the sense of  Proposition~\ref{thm:L^1_Cauchy}) if $\alpha$ is sufficiently close to $\gamma$ and $\gamma <\sqrt{2d}$.
Like in Section~\ref{sect:L^2}, we will rely on the fact that the weak and strong limits coincide to identify that the law of the random variable $\mu^\gamma$ is the same law as the GMC measure $\nu^\gamma$.  
First of all, we need to introduce further notation and establish some preparatory lemmas. 
Then, we will give the proofs of  Theorem~\ref{thm:L^1_phase}, Theorem~\ref{th:GMCgeneral} and finally of Theorem~\ref{thm:sine_GMC}.\\

For any $u\in\A$ and $N>0$, let $Z^N(u)$ be the random variable taking values in $\R^\N$  (measurable with respect to the process $(X_{N, \epsilon})_{\epsilon>0}$) 
given by
\begin{equation} \label{BBM}
Z^N_k(u) =
 X_{N, e^{-k}}(u) . 
\end{equation}
Here, $\R^\N$ is equipped with the usual product topology and its Borel $\sigma$-algebra.
For any $u\in\A$, let
 $\mathbb{P}^{u}_{N, \epsilon}$ be the probability measure with Radon-Nykodym derivative proportional to
$\exp\big( \X_{N,\epsilon}(u)\big) $ with respect to the probability measure~$\mathbb{P}$. 
Finally, for any $(u,v)\in\A^2$ such that $u\neq v$, let $\mathbb{P}^{u,v}_{N, \epsilon}$ be the probability measure with Radon-Nykodym derivative proportional to
$\exp\big(\X^\gamma_{N,\epsilon}(u)+\X^\gamma_{N,\epsilon}(v)\big) $ and, in the mixed regime,  let $\Pt^{u,v}_{N, \epsilon}$ be the probability measure with Radon-Nykodym derivative  proportional to
$\exp\big(\X^\gamma_{N,\delta_N}(u)+\X^\gamma_{N,\epsilon}(v)\big) $  with respect to $\mathbb{P}$.  We make the following assumption :

\begin{assumption}[Weak convergence of  finite-dimensional distributions  under the biased measures] \label{cvg}
For any $(u,v)\in\A^2$ such that $u\neq v$, one has
\begin{align*} 
& {\bf Law}_{\mathbb{P}^{u,v}_{N, \epsilon_N}}\big( Z^N(u), Z^N(v) \big) \rightarrow \mathbb{G}_{u,v} 
\quad \text{ as } N\to\infty , \\
&{\bf Law}_{\mathbb{P}^{u,v}_{N, \epsilon}}\big( Z^N(u), Z^N(v)\big) 
 \rightarrow
  \mathbb{G}_{u,v}  
  \quad\text{ as }N\to\infty\text{ and then }\epsilon\to0 , \\
& {\bf Law}_{\widetilde{\mathbb{P}}^{u,v}_{N, \epsilon}}\big( Z^N(u), Z^N(v)\big) 
 \rightarrow
  \mathbb{G}_{u,v}  
  \quad\text{ as }N\to\infty\text{ and then }\epsilon\to0 , 
 \end{align*} 
 where the convergence holds weakly for finite dimensional marginals and $\mathbb{G}_{u,v}$ is  a Gaussian measure on $\R^\N\times \R^\N$  with mean
 \begin{equation} \label{G_1}
 \mathbf{E}_{\mathbb{G}_{u,v}}[ Z_k(x)]= \gamma \big( T_{0,e^{-k}}(x,x)+ T_{0,e^{-k}}(u,v)  \big) , 
\quad \forall x\in \{u,v\} ,  
\end{equation}
and covariance structure:
\begin{equation} \label{G_2}
\langle Z_k(x);Z_j(y) \rangle_{\mathbb{G}_{u,v}}=  T_{e^{-k},e^{-j}}(x,y) ,  
\qquad \forall x,y\in \{u,v\} . 
\end{equation}
\end{assumption}

For any $\alpha>0$ and $0< L <M$, we  define the following events:
\begin{equation} 
\textstyle
\mathrm{A}_{L,M}^\alpha(Z) = \bigcap_{k=L}^M \{ Z_k < \alpha k\}
\quad \text{ and }\quad
\mathrm{A}_{L,M}^{\alpha *}(Z) =  \{ \exists k \in[ L, M] :  Z_k  \ge \alpha k  \} .
\end{equation}
Observe that one has $ \mathrm{A}_{L,\infty}^\alpha(Z) =  \bigcap_{M=L}^\infty \mathrm{A}_{L,M}^\alpha(Z) $ and  for any  $0< L < L' < M$, 
\begin{equation}
\mathrm{A}_{L,L'}^\alpha(Z) \cap \mathrm{A}_{L,M}^{\alpha *}(Z) = \mathrm{A}_{L,L'}^\alpha(Z) \cap \mathrm{A}_{L',M}^{\alpha *}(Z). \label{intersection} 
\end{equation}

\begin{lemma} \label{thm:GT_1}
Let $M_N = \lfloor \log \delta_N^{-1} \rfloor$. 
If $\gamma< \alpha$ and $L$ is sufficiently large, one has for any $\epsilon>0$, 
$$
\limsup_{N\to\infty} \mathbb{P}^{u}_{N, \epsilon} \left[  \mathrm{A}_{L,M_N}^{\alpha *}(Z^N(u))  \right]  \ll e^{-\frac{(\alpha-\gamma)^2}{4}L} . 
$$
\end{lemma}
\proof 
By a union bound, 
\begin{align*}
\mathbb{P}^{u}_{N, \epsilon} \left[  \mathrm{A}_{L,M_N}^{\alpha *}(Z^N(u))  \right] 
 \le \sum_{k=L}^\infty  \mathbb{P}^{u}_{N, \epsilon}\left[ Z^N_k(u) \ge \alpha k \right] . 
 \end{align*}
 By Assumption~\ref{finite_dist},  under the law $ \mathbb{P}^{u}_{N, \epsilon}$, $Z^N_k(u)$ converges weakly to Gaussian random variable with mean  $\gamma T_{\epsilon,e^{-k}}(u,u) $ and variance  $T_{e^{-k},e^{-k}}(u,u) $ for any $k\in \N$. Therefore, by the Portmanteau Theorem and  a Gaussian tail-bound, we obtain
 \begin{align*}
\limsup_{N\to\infty} \mathbb{P}^{u}_{N, \epsilon} \left[  \mathrm{A}_{L,M_N}^{\alpha *}(Z^N(u))  \right]  
&\le  \sum_{k=L}^\infty  \exp\left(- \frac{\big(\alpha k -\gamma T_{\epsilon,e^{-k}}(u,u) \big)^2}{2 T_{e^{-k},e^{-k}}(u,u)} \right) . 
\end{align*}
Thus, by Assumption~\ref{kernel_asymp}, if $\alpha>\gamma$ and the parameter $L$ is sufficiently large, we can assume that  for all $k\ge L$, 
$$\alpha k -\gamma T_{\epsilon,e^{-k}}(u,u) \ge  \sqrt{3}(\alpha-\gamma) k/2
\quad\text{ and }\quad
T_{e^{-k},e^{-k}}(u,u) \le 3k/2 .$$
This implies that for any $u\in\A$,
\begin{align*}
\limsup_{N\to\infty} \mathbb{P}^{u}_{N, \epsilon} \left[  \mathrm{A}_{L,M_N}^{\alpha *}(Z^N(u))  \right]  
&\le \sum_{k \ge L} \exp\left(-\frac{(\alpha - \gamma)^2}{4}k \right) 
\end{align*}
which  completes the proof. \qed\\

\begin{lemma} \label{thm:GT_2}
Let $M_N = \lfloor \log \delta_N^{-1} \rfloor$ and suppose that $\gamma< \alpha< 2\gamma $.  One has for any $u\in\A$,
$$ \lim_{L\to\infty} \lim_{N\to\infty}  \mathbb{P}^{u}_{N, \epsilon_N} \left[  \mathrm{A}_{L,M_N}^{\alpha *}(Z^N(u))  \right] = 0 .  $$
Similarly, one has for any points $u,v \in \A$ and  $x \in\{ u,v\}$, 
$$
 \lim_{L\to\infty} \lim_{N\to\infty}  \mathbb{P}^{u,v}_{N, \epsilon_N}\left[\mathrm{A}_{L,M_N}^{\alpha *}(Z^N(x))\right]  = 0 . 
$$
\end{lemma}

\proof By a union bound  and Markov's inequality, we obtain for any $t>0$, 
\begin{align*}
\mathbb{P}^{u}_{N, \epsilon_N} \left[  \mathrm{A}_{L,M_N}^{\alpha *}(Z^N(u))  \right] 
& \le \sum_{k=L}^{M_N}  \mathbb{P}^{u}_{N, \epsilon_N}\left[ Z^N_k(u) \ge \alpha k \right]\\
&\le  \sum_{k=L}^{M_N}   e^{-t\alpha k + \frac{t^2}{2} \E{X_{N, e^{-k}}(u)^2}} \E{ e^{  
\X^t_{N, e^{-k}}(u) + \X^\gamma_{N, \epsilon_N}(u)}}
\end{align*}
Using the assumptions, in particular the asymptotics \eqref{assumption_4}, this implies that
\begin{align*}
\mathbb{P}^{u}_{N, \epsilon_N} \left[  \mathrm{A}_{L,M_N}^{\alpha *}(Z^N(u))  \right] 
\ll  \sum_{k=L}^{M_N}   e^{-t\alpha k + \frac{t^2}{2} T_{e^{-k}, e^{-k}}(u,u) + t\gamma  T_{e^{-k}, \epsilon_N}(u,u)  }.
\end{align*}
Using the upper-bound \eqref{domination} and choosing $t = \alpha-\gamma$, this shows that 
\begin{align*}
\limsup_{N\to\infty}\mathbb{P}^{u}_{N, \epsilon_N} \left[  \mathrm{A}_{L,M_N}^{\alpha *}(Z^N(u))  \right] 
&\ll  \sum_{k=L}^\infty   e^{-t k \big( (\alpha- \gamma) - t/2 \big)} \\
&\ll e^{-L(\alpha- \gamma)^2/ 2 } .
\end{align*}
Letting $L\to\infty$, this completes the first part of the proof. 
Now, for the second estimate, an analogous argument shows that 
$$
\mathbb{P}^{u,v}_{N, \epsilon_N}\left[\mathrm{A}_{L,M_N}^{\alpha *}(Z^N(x))\right]  
\ll  \sum_{k=L}^\infty   e^{-t\alpha k + \frac{t^2}{2} T_{e^{-k}, e^{-k}}(u,u) + t\gamma  \big( T_{e^{-k}, \epsilon_N}(x,u) + T_{e^{-k}, \epsilon_N}(x,v)  \big)  + \gamma^2 T_{\epsilon_N, \epsilon_N}(u,v) }.
$$
Then, if $L \ge  \zeta = \log|u-v|^{-1}$, still choosing $t = \alpha-\gamma < \gamma$,  we obtain
\begin{align*}
\limsup_{N\to\infty}
\mathbb{P}^{u,v}_{N, \epsilon_N}\left[\mathrm{A}_{L,M_N}^{\alpha *}(Z^N(x))\right]   
&\ll e^{2\gamma^2 \zeta}  \sum_{k=L}^\infty   e^{-t k \big( (\alpha- \gamma) - t/2 \big)} \\
&\ll e^{- L(\alpha- \gamma)^2/ 2  + 2\gamma^2 \zeta} 
\end{align*}
which converges to 0 as $L\to\infty$. 
\qed \\

\begin{proposition} \label{thm:L^1_Cauchy}
Let $M_N = \lfloor \log \delta_N^{-1} \rfloor$ and  $0<\gamma<\sqrt{2d}$.
There exists $\gamma<\alpha <2\gamma$ so that, under the assumptions of Theorem~\ref{thm:L^1_phase},  one has
$$ \lim_{\epsilon\to0} \lim_{N\to\infty}
\E{\left| \mu_{N,\epsilon_N}^\gamma\big(w\1_{\mathrm{A}_{L,M_N}^\alpha(Z^N)}\big) -\mu_{N,\epsilon}^\gamma\big(w\1_{\mathrm{A}_{L,M_N}^\alpha(Z^N)}\big) \right|^2} =0 . 
$$
\end{proposition}

\proof In this proof, the parameter $L$ is assumed to be large but fixed. 
For any $(u,v) \in \A^2$, $\epsilon>0$ and $N>0$, let us define
$$
 \mathscr{W}_{N,\epsilon}^{u,v} = \E{\exp\big(\X_{N,\epsilon}(u)+\X_{N,\epsilon}(v)\big) }
\quad\text{ and }\quad
 \widetilde{\mathscr{W}}_{N,\epsilon}^{u,v} = \E{\exp\big(\X_{N,\epsilon_N}(u)+\X_{N,\epsilon}(v)\big) } . 
$$
We also let
$$
 \mathcal{I}_{N,\epsilon} = \iint_{\A^2} \mathbb{P}^{u,v}_{N, \epsilon}[\mathrm{A}_{L,M_N}^\alpha(Z^N(u)) \cap \mathrm{A}_{L,M_N}^\alpha(Z^N(v))]  \mathscr{W}_{N,\epsilon}^{u,v}   w(u)w(v)dudv
$$
and 
$$
 \widetilde{\mathcal{I}}_{N,\epsilon} = \iint_{\A^2} \widetilde{\mathbb{P}}^{u,v}_{N, \epsilon}[\mathrm{A}_{L,M_N}^\alpha(Z^N(u)) \cap \mathrm{A}_{L,M_N}^\alpha(Z^N(v))]  \widetilde{\mathscr{W}}_{N,\epsilon}^{u,v}   w(u)w(v)dudv . 
$$

As for the proof of Proposition~\ref{thm:L^2_Cauchy}, we may expand  
\begin{equation} \label{expansion_2}
\E{\left| \mu_{N,\epsilon_N}^\gamma(w\1_{{\rm A}^\alpha_{L,M_N}(Z^N)}) -\mu_{N,\epsilon}^\gamma(w\1_{{\rm A}^\alpha_{L,M_N}(Z^N)}) \right|^2} 
=
\mathcal{I}_{N,\epsilon_N} +  \mathcal{I}_{N,\epsilon} -2  \widetilde{\mathcal{I}}_{N,\epsilon}
\end{equation}
and we would like to prove that all the terms converge to the same limit when $N\to\infty$ and then $\epsilon\to0$. 
We will focus on computing  the limit  of the integral $\mathcal{I}_{N,\epsilon_N}$ as $N\to\infty$. 
On the one hand by \eqref{intersection} and elementary algebra, we obtain for any $L< L' < M_N$,  
\begin{align*}
\mathbb{P}^{u,v}_{N, \epsilon_N}\left[\mathrm{A}_{L,M_N}^\alpha(Z^N(u)) \cap \mathrm{A}_{L,M_N}^\alpha(Z^N(v))\right] 
= &\ \mathbb{P}^{u,v}_{N, \epsilon_N}\left[\mathrm{A}_{L,L'}^\alpha(Z^N(u))\cap \mathrm{A}_{L,L'}^\alpha(Z^N(v))\right]    \\
&\hspace{-1cm}- \mathbb{P}^{u,v}_{N, \epsilon_N}\left[\mathrm{A}_{L,L'}^\alpha(Z^N(u))\cap \mathrm{A}_{L,L'}^\alpha(Z^N(v))\cap \mathrm{A}_{L',M_N}^{\alpha *}(Z^N(u))\right]\\
&\hspace{-1cm}- \mathbb{P}^{u,v}_{N, \epsilon_N}\left[\mathrm{A}_{L,L'}^\alpha(Z^N(v))\cap\mathrm{A}_{L',M_N}^{\alpha *}(Z^N(v)) \cap \mathrm{A}_{L,M_N}^\alpha(Z^N(u))\right] . 
\end{align*}

The Assumption~\ref{cvg} implies that, under the law $\mathbb{P}^{u,v}_{N, \epsilon_N}$, $\big( Z^N_k(u),Z^N_k(v)\big)_{k=L}^{L'}$ converges weakly to a certain Gaussian vector. Then,  by Lemma~\ref{thm:GT_2}, this implies that for any $u\neq v$, 
\begin{align*}
\lim_{N\to\infty}\mathbb{P}^{u,v}_{N, \epsilon_N}\left[\mathrm{A}_{L,M_N}^\alpha(Z^N(u)) \cap \mathrm{A}_{L,M_N}^\alpha(Z^N(v))\right] 
=    \mathbb{G}_{u,v}  \left[\mathrm{A}_{L,L'}^\alpha(Z(u))\cap \mathrm{A}_{L,L'}^\alpha(Z(v))\right] + \underset{L'\to\infty}{o(1)} . 
\end{align*}
Consequently, by the monotone convergence theorem, we obtain 
\begin{align} \label{weak_cvg}
\lim_{N\to\infty}\mathbb{P}^{u,v}_{N, \epsilon_N}\left[\mathrm{A}_{L,M_{N}}^\alpha(Z^N(u)) \cap \mathrm{A}_{L,M_{N}}^\alpha(Z^N(v))\right] 
=  \mathbb{G}_{u,v}  \left[\mathrm{A}_{L,\infty}^\alpha(Z(u))\cap \mathrm{A}_{L,\infty}^\alpha(Z(v))\right] . 
\end{align}
On the other hand, Assumption \ref{exp_moments} implies that for all $(u,v) \in\A^2$,
\begin{equation}\label{partition}
\mathscr{W}_{N,\epsilon_N}^{u,v} \sim \exp\big(\gamma^2 T_{\epsilon_N,\epsilon_N}(u,v)\big)
\quad\text{ as }N\to\infty.
\end{equation}
So, in the regime  where $\gamma>\sqrt{d}$, we expect the limit of the integral $\mathcal{I}_{N,\epsilon_N}$ to be finite only if the probability $\mathbb{G}_{u,v}\left[ \mathrm{A}_{L,\infty}^\alpha(Z(u)), \mathrm{A}_{L,\infty}^\alpha(Z(v))\right]$ 
converges to 0 sufficiently fast as  $|u-v| \to 0 $. 
In order to prove this, we use that $\mathrm{A}_{L,M_N}^\alpha(Z)
\subset \{ Z_\zeta \le \alpha\zeta \}$ where we choose
$$
\zeta = \begin{cases}
\log |u-v|^{-1} &\text{if } L \le \log |u-v|^{-1} \le \log \epsilon_N^{-1} \\
L &\text{if }  \log |u-v|^{-1} \le L \\
 \log \epsilon_N^{-1} &\text{if }   |u-v| \le \epsilon_N
\end{cases}
$$
so that
$$
\mathbb{P}^{u,v}_{N, \epsilon_N}\left[\mathrm{A}_{L,M_N}^\alpha(Z^N(u))\cap \mathrm{A}_{L,M_N}^\alpha(Z^N(v))\right]  \le \mathbb{P}^{u,v}_{N, \epsilon_N}\left[Z^N_\zeta(u) \le \alpha\zeta \right] . 
$$
By Markov's inequality and the asymptotics \eqref{assumption_4}, this implies that for any $t>0$, 
\begin{align*}
 \mathbb{P}^{u,v}_{N, \epsilon_N}\left[Z^N_\zeta (u)\le \alpha\zeta \right]
\le &\frac{\E{\exp\left(\gamma X_{N,\epsilon_N}(u)+\gamma X_{N,\epsilon_N}(v) - tX_{N,e^{-\zeta}}(u)   \right) }}{\E{\exp\big(\gamma X_{N,\epsilon_N}(u)+ \gamma X_{N,\epsilon_N}(v)\big) }}e^{ \alpha t \zeta} \\  
&=\exp\left(   \alpha t \zeta -
\gamma t \left(T_{\epsilon_N, e^{-\zeta}}(u,u)+T_{\epsilon_N, e^{-\zeta}}(u, v) \right)
  +  \frac{t^2}{2} T_{ e^{-\zeta},  e^{-\zeta}}(u,u) + \underset{N\to\infty}{o(1)} \right) . 
\end{align*} 
Then, choosing 
$t = t_*(N, \gamma, \alpha,u,v) := \frac{ 
\gamma T_{\epsilon_N, e^{-\zeta}}(u,u)+\gamma T_{\epsilon_N, e^{-\zeta}}(u,v) -  \alpha \zeta }{T_{ e^{-\zeta},  e^{-\zeta}}(u,u) }, $
 we have
\begin{equation} \label{estimate_3}
 \mathbb{P}^{u,v}_{N, \epsilon_N}\left[Z^N_\zeta \le \alpha\zeta \right]
\ll   \exp\left( -\frac{ t_*^2}{2} T_{ e^{-\zeta},  e^{-\zeta}}(u,u) \right) . 
\end{equation}
Note that by definition of $\zeta$, Assumption~\ref{kernel_asymp} implies that
$T_{\epsilon_N, e^{-\zeta}}(u, v)  = \zeta +\hspace{-.1cm} \underset{L\to\infty}{O}\hspace{-.1cm}(1)$ when  $|u-v| \le e^{-L}$. In particular, in the regime $|u-v| \le e^{-L}$, by \eqref{assumption_5}, we see that
$
t_*=(2\gamma-\alpha)+ \underset{L\to\infty}{o(1)}. 
$
Hence, if  the parameter $L$ is sufficiently large and $\alpha<2\gamma$,  the parameter $t^*>0$ so that the bound \eqref{estimate_3} holds and we obtain 
\begin{equation*} \notag
 \mathbb{P}^{u,v}_{N, \epsilon_N}\left[Z^N_\zeta \le \alpha\zeta \right]
\ll   \exp\left( -\frac{ (2\gamma-\alpha)^2}{2} \zeta \right) . 
\end{equation*}
This shows that for all $u,v \in\A$ such that $|u-v| \le e^{-L}$, 
\begin{equation*}
 \mathbb{P}^{u,v}_{N, \epsilon_N}\left[\mathrm{A}_{L,M_N}^\alpha(Z^N(u))\cap \mathrm{A}_{L,M_N}^\alpha(Z^N(v))\right]  
\ll  \big(  |u-v|  \vee \epsilon_N \big)^{\frac{ (2\gamma-\alpha)^2}{2}} . 
\end{equation*}
Moreover, in the regime  $|u-v| > e^{-L}$, the RHS of \eqref{partition} remains bounded as $N\to\infty$. 
Hence, using the bound \eqref{domination}, these estimates show that  
\begin{align}  \notag
 \mathbb{P}^{u,v}_{N, \epsilon}[\mathrm{A}_{L,M_N}^\alpha(Z^N(u))\cap \mathrm{A}_{L,M_N}^\alpha(Z^N(v))]  \mathscr{W}_{N,\epsilon_N}^{u,v} 
& \ll  \big(  |u-v|  \vee \epsilon_N \big)^{\frac{ (2\gamma-\alpha)^2}{2} -\gamma^2} \\
&  \label{estimate_5} 
\ll   |u-v|^{ -\gamma^2 +\frac{ (2\gamma-\alpha)^2}{2}} .
\end{align}
The LHS of \eqref{estimate_5} is locally integrable on $\R^{d}\times \R^d$ if 
$\gamma^2 - \frac{(2\gamma-\alpha)^2}{2} < d$. 
Hence, as long as $\gamma^2<2d$, it is possible to choose 
$\alpha >\gamma$ so that this condition is satisfied.
By the dominated convergence theorem,  formulae \eqref{weak_cvg}--\eqref{partition} and \eqref{assumption_2}, we conclude that 
$\displaystyle
\lim_{N\to\infty}  \mathcal{I}_{N,\epsilon_N} = \mathcal{I}_{\infty}^{\gamma, L} 
$
where
\begin{equation*} 
\mathcal{I}_{\infty}^{\gamma, L} :=  \iint_{\R^2}  \mathbb{G}_{u,v}\left[ \mathrm{A}_{L,\infty}^\alpha(Z(u))\cap \mathrm{A}_{L,\infty}^\alpha(Z(v))\right]   e^{\gamma^2 T(u,v)} w(u)w(v) dudv <\infty .
\end{equation*}
We  can apply the same argument in the weak and mixed regimes as well and obtain
\begin{equation*}
\lim_{\epsilon\to0}\lim_{N\to\infty}  \mathcal{I}_{N,\epsilon} 
= \lim_{\epsilon\to0}\lim_{N\to\infty}  \widetilde{\mathcal{I}}_{N,\epsilon}
=  \mathcal{I}_{\infty}^{\gamma, L} . 
\end{equation*}
In the end, combining  these limits with formula \eqref{expansion_2}, we have completed the proof. \qed\\

\noindent{\it Proof of Theorem~\ref{thm:L^1_phase}.}
Recall that by Lemma~\ref{thm:weak_cvg}, for any $\epsilon>0$, 
 $\mu_{N,\epsilon}^\gamma(w) \Rightarrow  \nu^\gamma_\epsilon(w)$ as $N\to\infty$ and, according to Assumption~\ref{GMC_existence}, if $\gamma<\sqrt{2d}$, the random variable $ \nu^\gamma_\epsilon(w) \Rightarrow  \nu^\gamma(w) $
as $\epsilon\to0$. Therefore, by  \cite[Theorem~4.28]{Kallenberg}, it suffices to establish that \begin{equation*} 
 \limsup_{\epsilon\to0} \limsup_{N\to\infty}\E{\left| \mu_{N,\epsilon_N}^\gamma(w) -\mu_{N,\epsilon}^\gamma(w) \right|}  = 0 .
\end{equation*}
By the triangle inequality, we have for any $L, \alpha>0$ and $\epsilon>0$,
\begin{align*}  \notag
 \E{\left| \mu_{N,\epsilon_N}^\gamma(w) -\mu^\gamma_{N, \epsilon}(w) \right|}  
&\le \E{\left| \mu_{N,\epsilon_N}^\gamma\left(w\1_{\mathrm{A}_{L,M_N}^\alpha(Z^N)}\right) -\mu_{N,\epsilon}^\gamma\left(w\1_{\mathrm{A}_{L,M_N}^\alpha(Z^N)}\right) \right|} \\
& \quad
+\E{\left| \mu_{N,\epsilon_N}^\gamma\left(w\1_{\mathrm{A}_{L,M_N}^{\alpha*}(Z^N)}\right) \right|}
 + \E{\left|\mu_{N,\epsilon}^\gamma\left(w\1_{\mathrm{A}_{L,M_N}^{\alpha*}(Z^N)}\right) \right|} .
 \end{align*}
Then, by Proposition~\ref{thm:L^1_Cauchy}, if $\gamma<\sqrt{2d}$, by choosing the parameter $\alpha>\gamma$ appropriately,   we obtain
\begin{align} \label{estimate_4}
 \limsup_{\epsilon\to0} \limsup_{N\to\infty}\E{\left| \mu_{N,\epsilon_N}^\gamma(w) -\mu_{N,\epsilon}^\gamma(w) \right|}  
\le  &\limsup_{ N\to\infty} \E{\left| \mu_{N,\epsilon_N}^\gamma\left(w\1_{\mathrm{A}_{L,M_N}^{\alpha*}(Z^N)}\right) \right|}\\
&\notag +  \limsup_{\epsilon\to0} \limsup_{N\to\infty} \E{\left|\mu_{N,\epsilon}^\gamma\left(w\1_{\mathrm{A}_{L,M_N}^{\alpha*}(Z^N)}\right) \right|}.
 \end{align}

Since the LHS of \eqref{estimate_4} does not depend on the parameter $L>0$, in order to complete the proof, it suffices to show that 
\begin{align} 
\label{limit_1}
\lim_{L\to\infty} \lim_{N\to\infty}\E{\left|\mu_{N,\epsilon_N}^\gamma\left(w\1_{\mathrm{A}_{L,M_N}^{\alpha*}(Z^N)}\right) \right|}  & =0 ,\\
\label{limit_2}
\lim_{L\to\infty}\lim_{\epsilon\to0} \lim_{N\to\infty}
\E{\left|\mu_{N,\epsilon}^\gamma\left(w\1_{\mathrm{A}_{L,M_N}^{\alpha*}(Z^N)}\right) \right|} 
& =0 . 
\end{align}  
 By definition of the probability  measure
 $\mathbb{P}^{u}_{N, \epsilon}$,  one has for any $\epsilon>0$, 
  $$
\E{\big| \mu_{N,\epsilon}^\gamma\left(w\1_{\mathrm{A}_{L,M_N}^{\alpha*}(Z^N)}\right) \big| }   \le \int_{\A}\mathbb{P}^{u}_{N, \epsilon}\left[\mathrm{A}_{L,M_N}^{\alpha*}(Z^N(u))\right]  \E{e^{ \X_{N,\epsilon}(u)}} |w(u)| du .
$$
Note that by Assumption \ref{exp_moments}, $ \E{e^{ \X_{N,\epsilon}(u)}}\to 1$ as $N\to\infty$ and $\epsilon\to 0$ uniformly for all $u\in\A$, so that by Lemma~\ref{thm:GT_1}, one obtains \eqref{limit_2}. Finally, \eqref{limit_1} follows from an analogous argument using the estimate of Lemma~\ref{thm:GT_2}. \qed\\

In principle, it is possible to verify Assumption~\ref{cvg} without knowing the asymptotics of all exponential moments of the random field $X_{N,\epsilon}(u)$. 
However, if the Assumption \ref{exp_moments} holds for all $q\in\N$, then the  Assumption~\ref{cvg} follows from a routine computation. 
 In particular, as a consequence of  the asymptotics obtained in Section~\ref{sect:covariance} and Section~\ref{sect:asymptotics}, this applies to  the CUE and sine process. \\

\noindent{\it Proof of Theorem~\ref{th:GMCgeneral}.}
It suffices to check that the fields in question satisfy the assumptions of Theorem~\ref{thm:L^1_phase}. 
First of all, the strong Gaussian asymptotics \eqref{qmoms} implies directly Assumption~\ref{exp_moments}  as well as the fact that, for fixed $\epsilon>0$,  the field $u\mapsto X_{N,\epsilon}(u)$ converges in the sense of finite-dimensional distributions to a mean-zero Gaussian process $G_\epsilon$. In addition, one has for any $\mathbf{t}\in\R^q$ and $\u\in\{u,v\}^q$, 
\begin{align*}
&\mathbf{E}_{\mathbb{P}^{u,v}_{N, \epsilon_N}}\left[ \exp\left( \sum_{k=1}^{q} t_k
 X_{N, e^{-k}}(u_k) \right) \right] \\
& \hspace{2cm}=\frac{ \E{ \exp\left( \gamma X_{N,\delta_N}(u)+ \gamma X_{N,\delta_N}(v)
+\sum_{k=1}^{q} t_k X_{N, e^{-k}}(u_k) \right) }}{\E{\exp\big( \gamma X_{N,\delta_N}(u)+\gamma X_{N,\delta_N}(v)\big) }} \\
&\hspace{2cm}= \exp\left( \sum_{k=1}^{q} t_k \mathbf{E}_{\mathbb{G}_{u,v}}[Z_k(u_k)]
+   \frac{1}{2} \sum_{ k, j=1}^q  t_kt_j \langle Z_k(u_k);Z_j(u_j) \rangle_{\mathbb{G}_{u,v}}
  +   \underset{N\to\infty}{o(1)}   \right) 
\end{align*}
according to \eqref{G_1}--\eqref{G_2}. This establishes, the first part of the Assumption~\ref{cvg}; the other assertions in the weak and mixed regimes are checked in a similar fashion. Finally, when $G_\epsilon = G\,\ast\,\phi_{\epsilon}$, we check in Section~\ref{sect:covariance} that
the covariance kernel satisfies both Assumptions~\ref{finite_dist} and \ref{kernel_asymp} (the computation is given for the stationary field $G$ with correlation kernel \eqref{cov} but they can be easily generalized in other situations -- see for instance \cite{Berestycki15}). 
In this context, the condition~\ref{GMC_existence} (which, in general, is a non-trivial issue) follows from Theorem~\ref{th:GMCintro} presented in the introduction.
\qed \\

\noindent{\it Proof of Theorem~\ref{thm:sine_GMC}.}
If $\phi\in\mathcal{D}$ is a mollifier which satisfies Assumption~\ref{phi_assumption},
formula \eqref{full_asymp} shows that the random field $X_{N,\epsilon}$ given by  \eqref{X_field} satisfies the strong Gaussian asymptotics of Assumption~\ref{exp_moments}  with $G_\epsilon = G\,\ast\,\phi_{\epsilon}$. In this context, by  Theorem~\ref{th:GMCintro}, we know that for any  $w \in L^{1}(\R)\cap L^{\infty}(\R)$ and $\gamma <\sqrt{2d}$, the random variable 
$\nu^{\gamma}_{\epsilon}(w)$ converges in $L^1(\mathbb{P})$ to  $\nu^{\gamma}(w) $ as $\epsilon \to 0$. 
In addition, we  have already checked (see the proof of Theorem~\ref{th:GMCgeneral} above) that the Assumptions~\ref{finite_dist}--\ref{cvg}  are also satisfied. Then, by Lemma~\ref{thm:L^q_weak}, for any $\epsilon>0$, $\mu_{N,\epsilon}^{\gamma,\mathrm{Sine}}(w)$ converges in $L^1(\mathbb{P})$ to some random variable  $\mu^\gamma_\epsilon(w)$
and, since $\mu^\gamma_\epsilon(w) \overset{d}{=} \nu^\gamma_\epsilon(w)$,  this shows that for any  $\gamma <\sqrt{2d}$, there exists a random variable $\mu^\gamma(w) \overset{d}{=} \nu^\gamma(w)$ so that by taking successive limits, we obtain
\begin{equation} \label{limit_3}
\lim_{\epsilon \to 0} \lim_{N\to\infty} \E { \left| \mu_{N,\epsilon}^{\gamma,\mathrm{Sine}}(w) - \mu^\gamma(w) \right| } =0 . 
\end{equation}
Then, using Lemmas~\ref{thm:GT_1}, \ref{thm:GT_2} and Proposition~\ref{thm:L^1_Cauchy}, proceeding exactly as in the proof of Theorem~\ref{thm:L^1_phase}, we can show that if $\gamma <\sqrt{2d}$ and $\epsilon_N\ge \delta_N N (\log N)^{1+\kappa}$ for some $\kappa>0$, we have
\begin{equation} \label{limit_4}
 \lim_{\epsilon\to0} \lim_{N\to\infty}\E{\left| \mu_{N,\epsilon_N}^{\gamma,\mathrm{Sine}}(w) -\mu_{N,\epsilon}^{\gamma,\mathrm{Sine}}(w) \right|}   =0 . 
 \end{equation}
 Using the triangle inequality and combining the limits \eqref{limit_3} and \eqref{limit_4}, we conclude that the random variable $ \mu_{N,\epsilon_N}^{\gamma,\mathrm{Sine}}(w) $ converges in $L^1(\mathbb{P})$ to $\mu^{\gamma}(w)$  as $N\to\infty$ and  $\mu^\gamma(w) \overset{d}{=} \nu^\gamma(w)$. \qed

\subsection{Asymptotics of the covariances} \label{sect:covariance}

Let $G$ be the stationary Gaussian process on $\R$ with covariance function
$Q$, \eqref{cov}, and recall that, for any mollifier $\phi$, we denote $G_{\phi, \epsilon} = G*\phi_\epsilon$.  
In this section, we derive the asymptotics of the covariance 
\begin{equation} \label{covariance_4}
\E{G_{\phi, \epsilon}(u)G_{\psi, \delta}(v)}
=\int_\R e^{-\pii (u-v)\kappa} \hat\psi(\epsilon\kappa) \overline{\hat\phi(\delta\kappa)} \widehat{Q}(\kappa)  d\kappa  , 
\end{equation}
 as $\epsilon, \delta\to0$, for all $u,v\in\R$
and for a large class of mollifiers. This is relevant to check that the conditions \eqref{assumption_2}, \eqref{domination}, and \eqref{assumption_5} are satisfied and apply Theorems~\ref{thm:L^2_phase} and~\ref{thm:L^1_phase} to conclude that the multiplicative chaos measures associated with counting statistics of the CUE and sine process exist and have the same law. For any $\ell, \epsilon>0$, we define the function
\begin{equation} \label{Q_epsilon}
Q_\epsilon(x) = - \log\left(\frac{\epsilon}{2\pi}\vee|x|\right)  
+  \log\left(\frac{\epsilon}{2\pi}\vee \sqrt{|\ell^2-x^2|} \right) . 
\end{equation}
In the following, we will use the notation $\Oo{\epsilon\to0}{\u}$ to specify a function of the variable $\u$ and the parameter $\epsilon>0$ which is uniformly bounded (by a universal constant) and converges to $0$ as $\epsilon\to0$ for almost all $\u\in\R^q$.\\

Recall that we defined
$\mathcal{D} = \bigcup_{\alpha>0} \mathcal{D}_\alpha$, c.f.~\eqref{D}, 
and recall also the definition of the cosine integrals:
\begin{equation*}
\Cin(\omega)  = \int_{0}^{\omega}\frac{1-\cos( z)}{z}\,dz 
\quad\text{ and }\quad
\Ci(x) = \int_{1}^{\infty}\frac{\cos( x t)}{t} dt . 
\end{equation*}
The function $\Cin$ is entire, while the function $\Ci$ is even on $\R$ with the value $+\infty$ at 0,  and it turns out that for any $\omega\in\R\backslash\{0\}$,
\begin{equation*}
\Cin(\omega) = \Ci(\omega) +\log | \omega|+\gamma,
\end{equation*}
where $\gamma$ is the Euler constant.  In particular, since 
$\displaystyle\lim_{x\to\infty} \Ci(x) =0$, we have for any $\omega \in\R$, 
\begin{equation} \label{Cin_asymptotics}
\Cin(\omega/\epsilon) = \log^+(\omega/\epsilon) + \gamma+ \Oo{\epsilon\to0}{\omega} . 
\end{equation}

\begin{lemma}\label{thm:Cin}
Let $\Phi \in L^1(\R_+)$, continuous with $\Phi(0)=1$, and so that  the function $ \kappa\mapsto\frac{\Phi(\kappa)-\1_{\kappa\le1}}{\kappa}$ is integrable on $\R_+$.  
Then, the function 
\begin{equation*} 
\mathscr{E}_\Phi(\omega) = \int_{0}^{\infty}(1-\cos( \omega\kappa))\frac{\Phi(\kappa)-\1_{\kappa\le1}}{\kappa}\,d\kappa 
\end{equation*}
is continuously differentiable on $\R$, and
$\displaystyle
\lim_{\omega\to\infty} \mathscr{E}_\Phi(\omega) =  \int_{0}^{\infty}\frac{\Phi(\kappa)-\1_{\kappa\le1}}{\kappa}\,d\kappa. 
$\\\\
Moreover, for all $\omega\in\R$, 
\begin{equation} \label{Cin}
 \int_{0}^{\infty}\frac{1-\cos( \omega\kappa)}{\kappa}\Phi(\kappa)\,d\kappa = \Cin(\omega) + \mathscr{E}_\Phi(\omega) .
\end{equation}
\end{lemma}

\proof All the properties of the function $\mathscr{E}_\Phi$ are easy to check. In particular, we have
\begin{equation*} 
\mathscr{E}_\Phi'(\omega) = \int_{0}^{\infty}\sin(\omega \kappa )\left(\Phi(\kappa)-\1_{\kappa\le1}\right)d\kappa ,
\end{equation*}
and the limit of  $\mathscr{E}_\Phi(\omega)$ as $\omega\to\infty$ follows directly from the Riemann-Lebesgue  lemma. Finally, the identity \eqref{Cin} is an immediate consequence of the definition of the cosine integral.  \qed\\

\begin{proposition}\label{thm:covariance}

For any functions $\phi,\psi\in\mathcal{D}$, we have for all $u,v\in\R$,
$$
\E{G_{\phi, \epsilon}(u)G_{\psi, \epsilon}(v)}= Q_\epsilon(u-v) + \Oo{\epsilon\to0}{u,v}. 
 $$
\end{proposition}

\proof
Let  $\Phi =\Re\big\{ \hat{\phi}\overline{\hat\psi} \big\}$, $\Psi =\Im\big\{ \hat{\phi}\overline{\hat\psi} \big\}$, and $\epsilon' = \epsilon/2\pi$. 
We claim that the function $\Phi$ satisfies the assumption of Lemma~\ref{thm:Cin}. 
Namely, by the Cauchy-Schwarz inequality,  $\Phi\in L^1(\R)$ and by Plancherel's formula, for any $\kappa\in\R$,  
$$
\Phi(\kappa) =  \iint_{\R^2} \cos(2\pi \kappa x) \phi(x+t) \psi(t) dxdt .
$$
In particular, the  bound $|e^{i \omega} - 1| \le 2  | \omega|^\alpha$ valid for all $\omega\in\R$ and $0\le \alpha\le1$  implies that 
\begin{align*}
\left| \frac{\Phi(\kappa)-1}{\kappa}\right| 
&\le 2 \pi  \kappa^{\alpha-1}  \iint_{\R^2} |x|^\alpha \phi(x+t) \psi(t) dxdt  \\
&\le 2\pi   \kappa^{\alpha-1} \int |x|^\alpha \phi(x)dx \int|t|^\alpha \psi(t) dt
\end{align*}
Since both $\phi,\psi\in\mathcal{D}_\alpha$ for some $\alpha>0$,  this shows that the function $ \frac{\Phi(\kappa)-\1_{\kappa\le1}}{\kappa}$ is integrable on~$\R_+$. 
Similarly, the function $ \Psi(\kappa)/\kappa$ is also integrable on $\R_+$.
By the Riemann-Lebesgue lemma, this implies that for any $\omega\in\R$,
\begin{equation}  \label{integral_2}
\int_{0}^{\infty}\sin( \epsilon\omega\kappa)\Psi(\kappa)\frac{d\kappa}{\kappa}
=\Oo{\epsilon\to0}{\omega}. 
\end{equation}
By  formula~\eqref{covariance_4} and by definition of  $\hat{Q}$, we have
 \begin{align}  \label{covariance_5}
& \E{G_{\phi, \epsilon}(u)G_{\psi, \epsilon}(v)} 
 = 2 \int_0^\infty\Re\left\{ e^{-\pii (u-v) \kappa} \hat{\phi}(\epsilon\kappa)\overline{\hat\psi(\epsilon\kappa)}\right\}   \frac{\sin^2(\pi\ell\kappa)}{\kappa}    d\kappa \\
&\label{covariance_6}
=  \underbrace{\int_0^\infty 2\cos\left(\frac{(u-v)\kappa}{\epsilon' }\right) \frac{\sin^2\left(\ell \kappa/2\epsilon'\right)}{\kappa} \Phi(\kappa) d\kappa}_{\displaystyle:= I_1(u,v,\epsilon')} 
+ \underbrace{\int_0^\infty 2\sin\left(\frac{(u-v)\kappa}{\epsilon' }\right) \frac{\sin^2\left(\ell \kappa/2\epsilon'\right)}{\kappa} \Psi(\kappa)d\kappa}_{\displaystyle:= I_2(u,v,\epsilon')} .
 \end{align}
Using the trigonometric identity
$$
-2\cos(a)\sin^2(b/2)  = 1-\cos a - \frac{1-\cos(b+a)+ 1- \cos(b-a)}{2} ,
$$
with $a=(u-v)\kappa/\epsilon' $ and $b=\ell\kappa/\epsilon'$, 
 we see that we can apply Lemma~\ref{thm:Cin}. We obtain for all $u,v\in\R$,
\begin{align}
I_1(u,v,\epsilon')
=&\notag -\Cin\left(\frac{u-v}{\epsilon'}\right) +
\frac{1}{2} \left\{ \Cin\left(\frac{\ell+u-v}{\epsilon'}\right)
+\Cin\left(\frac{\ell+v-u}{\epsilon'}\right) \right\} \\
&\label{error_3}\qquad
- \mathscr{E}_\Phi\left(\frac{u-v}{\epsilon'}\right) +
\frac{1}{2} \left\{  \mathscr{E}_\Phi\left(\frac{\ell+u-v}{\epsilon'}\right)
+ \mathscr{E}_\Phi\left(\frac{\ell+v-u}{\epsilon'}\right) \right\} , 
\end{align}
and the terms in \eqref{error_3} combine as the error term $\Oo{\epsilon\to0}{u,v}$. Moreover, by formula \eqref{Cin_asymptotics}, this implies that
\begin{equation*}
I_1(u,v,\epsilon')
= -\log^+\left(\frac{u-v}{\epsilon'}\right) +
\frac{1}{2} \left\{ \log^+\left(\frac{\ell+u-v}{\epsilon'}\right)
+\log^+\left(\frac{\ell+v-u}{\epsilon'}\right) \right\} 
+ \Oo{\epsilon\to0}{u,v} ,
\end{equation*}
By definition of the functions $\log^+$ and $Q_\epsilon$, \eqref{Q_epsilon}, this may be written as
\begin{equation} \label{integral_1}
I_1(u,v,\epsilon')
= Q_\epsilon(u-v) + \Oo{\epsilon\to0}{u,v}. 
\end{equation}
We can also evaluate the integral $I_2(u,v,\epsilon')$ using a similar argument. Since
$$
2\sin(a)\sin^2(b/2)  =  \sin a - \sin(a+b)/2  +\sin(a-b)/2 ,
$$
the estimate \eqref{integral_2} shows that $I_2(u,v,\epsilon')
= \Oo{\epsilon\to0}{u,v}. $ Combining this fact with \eqref{integral_1}
and formula~\eqref{covariance_6}, this completes the proof.  \qed\\

\begin{proposition} \label{thm:mixed_covariance}
Let $\phi,\psi\in\mathcal{D}$, we have for all $u,v\in\R$,
$$
\E{G_{\phi, \epsilon}(u)G_{\psi, \delta}(v)} = Q_{\epsilon\vee\delta}(u-v) 
 + \Oo{\delta, \epsilon\to0}{u,v} , 
 $$
 where we can consider the limit as the parameters $\epsilon$ and $\delta$ converge to $0$ in an arbitrary way. 
\end{proposition}

\proof

We let $\epsilon'=\epsilon/2\pi$ and $\delta' = \delta /\epsilon$. Without loss of generality, we suppose that $\delta' \to \beta \in [0,1]$ as $\epsilon\to0$ (in particular, $\beta=0$ if we consider successive limits as $\delta \to0$ and $\epsilon\to0$). By formula \eqref{covariance_5}, we have
\begin{equation*} 
\E{G_{\phi, \epsilon}(u)G_{\psi, \delta}(v)}
= 2\int_0^\infty \Re\left\{  e^{-(u-v)\kappa/\epsilon'} \hat\phi(\kappa) \overline{\hat\psi(\kappa\delta')} \right\}\frac{\sin^2(\ell\kappa/2\epsilon')}{\kappa}    d\kappa  . 
\end{equation*}
In particular, this implies that
\begin{equation*} 
\left| \E{G_{\phi, \epsilon}(u)G_{\psi, \delta}(v)}-
\E{G_{\phi, \epsilon}(u)G_{\psi, \epsilon\beta}(v)} \right|
\le 2 \int_0^\infty \left| \hat\psi(\kappa\beta)- \hat\psi(\kappa\delta') \right| 
\big|\hat\phi(\kappa)\big| \frac{d\kappa}{\kappa} . 
\end{equation*}
Since 
$\displaystyle \left| \hat\psi(\kappa\beta)- \hat\psi(\kappa\delta') \right|  \le 4\pi |\kappa|^\alpha |\beta- \delta'|^\alpha \int |x|^\alpha \psi(x) dx$ for any $0\le \alpha\le1$, there exists a constant $C>0$ which depends only on $\psi \in \mathcal{D}_\alpha$ so that if $0<\alpha<1/2$ is sufficiently small, then
\begin{equation} \label{covariance_7}
\left| \E{G_{\phi, \epsilon}(u)G_{\psi, \delta}(v)}-
\E{G_{\phi, \epsilon}(u)G_{\psi, \epsilon\beta}(v)} \right|
\le C  |\beta- \delta'|^\alpha \int_0^\infty \big|\hat\phi(\kappa)\big| \frac{d\kappa}{\kappa^{1-\alpha}} . 
\end{equation}  
Since $\phi\in L^1\cap L^2(\R)$, by the Cauchy-Schwarz inequality, the RHS of \eqref{covariance_7} is finite and converges to 0 as $\epsilon\to0$ (by assumption, $\lim\delta'=\beta$).  Thus, it suffices to prove that given $\beta\in[0,1]$, 
\begin{equation} \label{covariance_8}
\E{G_{\phi, \epsilon}(u)G_{\psi, \epsilon\beta}(v)} = Q_\epsilon(u-v) 
 + \Oo{\epsilon\to 0}{u,v} .  
\end{equation}
If $\beta >0$,  this follows directly from Lemma~\ref{thm:covariance}, since 
$G_{\psi, \epsilon\beta} = G_{\psi_\beta, \epsilon}$. If $\beta=0$, since $\hat{\psi}(0)=1$, we have 
\begin{align*}
\E{G_{\phi, \epsilon}(u)G_{\psi, \epsilon\beta}(v)}
=\int_0^\infty \Re\left\{  e^{-(u-v)\kappa/\epsilon'} \hat\phi(\kappa) \right\}\frac{\sin^2(\ell\kappa/2\epsilon')}{\kappa}    d\kappa ,
\end{align*}
and the same computations as in the proof of Proposition~\ref{thm:covariance} shows that the function $\Re\hat\phi$ satisfies the assumptions of Lemma~\ref{thm:Cin} 
and the function $\kappa\mapsto\Im\hat\phi(\kappa)/\kappa$ is integrable on $\R_+$, so that the asymptotics \eqref{covariance_7} hold. \qed\\

\begin{corollary} \label{covariance_assumptions}
Let $\ell>0$, $G$ be the stationary Gaussian process on $\R$ with covariance function \eqref{cov}, and let  $\phi\in\mathcal{D}$. For any $u,v\in\R$, let $T(u,v)=Q(u-v)$, and for any $\epsilon,\delta>0$, define
$$
T_{\epsilon,\delta}(u,v) =  \E{G_{\phi, \epsilon}(u)G_{\phi, \delta}(v)}. 
$$
Then, the  function $T_{\epsilon,\delta}$ satisfies Assumption~\ref{kernel_asymp}. 
\end{corollary}

\proof By formula \eqref{Q_epsilon}, 
$\displaystyle \lim_{\epsilon\to0} Q_\epsilon(x) = Q(x)$ for almost all $x\in\R$, so that the condition \eqref{assumption_2}  follows directly from  Proposition~\ref{thm:mixed_covariance}.
Similarly, \eqref{assumption_5} is also an immediate consequence  Proposition~\ref{thm:mixed_covariance}.
To get the upper-bound, we check that if $\epsilon \le \ell \wedge 1$, directly from \eqref{Q_epsilon} we obtain
$$
Q_\epsilon(x) \le - \log\left(\epsilon \vee|x|\right)   +  \log(2\pi\sqrt{\ell^2+|x|^2}).
$$
We now use the identity $-\log^+(\epsilon^{-1} \wedge |x|^{-1}) = \log(\epsilon \vee |x|)+\log(|x|)\1_{|x|\geq 1}$, which follows from the definition \eqref{logplus} of $\log^+$ and the condition $\epsilon \leq 1$, to derive
$$
Q_\epsilon(x) -  \log^+\left(\epsilon^{-1}\wedge |x|^{-1}\right)   \le \sup\left\{ \log(2\pi\sqrt{\ell^2+|x|^2}) : |x| \le 1 \right\} \vee 
 \sup\left\{ - \log|x|  +  \log(2\pi\sqrt{\ell^2+|x|^2}) : |x| \ge 1 \right\} .
$$
Since the function $x\mapsto  -\log|x|  +  \log(2\pi\sqrt{\ell^2+|x|^2})$ is decreasing on the set $|x| \geq 1$, we have
\begin{equation*} 
Q_{\epsilon \vee \delta}(x) \le \log^+\left(|x|^{-1}\wedge \epsilon^{-1}\wedge\delta^{-1}\right) + \log(2\pi\sqrt{\ell^2+1})
\end{equation*}
and the condition \eqref{domination} also follows from Proposition \ref{thm:mixed_covariance}. \qed

\section{Asymptotic analysis} \label{sect:asymptotics}

\subsection{Proof of Proposition~\ref{thm:sine_asymp}} \label{sect:RHP}

The goal is to deduce the asymptotics of Proposition~\ref{thm:sine_asymp} 
from Theorem~\ref{thm:Deift} by performing the asymptotics of the solution to the Riemann-Hilbert problem \eqref{RHP_0}. Recall that for the sine process, we have
\begin{equation} \label{fg}
\f(x) =\begin{pmatrix} e^{i \pi N x} \\  e^{-i \pi Nx}     \end{pmatrix} ,
\hspace{1cm}
\g(x) = \begin{pmatrix} e^{i \pi N x} \\  -e^{-i \pi Nx}     \end{pmatrix} , 
\end{equation}
so that we look for the asymptotics of the solution to the problem:
\begin{itemize}
	\item $m(z) $ is analytic on $\mathbb{C}\backslash\R$.
	\item $m(z)$ satisfies the jump condition:
\begin{equation}  \label{RHP_1}
m_{+}(x) = m_{-}(x)  \begin{pmatrix} 1+ \varphi(x)  & - \varphi(x)e^{\pii N x}  \\  
 \varphi(x)e^{ -\pii N x} & 1- \varphi(x) \end{pmatrix} , \qquad x \in \R 
\end{equation}
\vspace{-.7cm}
	\item $m(z) \to \Id$ as $z \to \infty$ in $\mathbb{C}\backslash\R$. 
\end{itemize}

Note that we do not emphasize that the matrix $m$ depends on the dimension $N$
to keep the notation simple. Moreover, the solution of  \eqref{RHP_0} can be obtained from the solution of  \eqref{RHP_1} simply by replacing $\varphi$ by $\varphi_t= t\varphi$ for all $t\in [0,1]$. 
Finally, we will generalize slightly the setting of Proposition~\ref{thm:sine_asymp} and work with the following assumptions.

\begin{assumption} \label{psi_assumption}
Suppose that $\varphi$ is a function which is analytic in the strip $| \Im z | <\delta $,  so that $\varphi : \R \to \R$ and  $\varphi \in L^1 \cap L^\infty(\R\pm i s)$ for all $|s| \le \delta/2$. 
In particular, we define
$$C_\varphi = 
  \sup\big\{ \| \varphi\|_{L^\infty(\R+i s)} \vee  \| \varphi\|_{L^1(\R+i s)}  : s \le\delta/2\big\} .$$
 Let us also assume that there exists a constant $c>0$ so that
\begin{equation}  \label{condition_alpha}
\inf_{  | \Im z | <\delta, x<-1}|\varphi(z) - x|  \ge c. 
\end{equation}
In particular, the function $\psi = \log(1+\varphi)$ is also analytic  in the strip $| \Im z | <\delta$. Finally, we assume that $\psi \in L^1\cap L^\infty(\R)$ and we let  $C_\psi = c^{-1} \exp  \| \psi \|_{L^\infty(\R)}$. 
\end{assumption}

\begin{lemma} \label{thm:RHP}
Suppose that the function $\varphi$ satisfies Assumption~\ref{psi_assumption} and  that $   C_\varphi C_\psi^2 e^{-\pi \delta N} \to 0$ as $N\to\infty$.
Then, the solution of the Riemann-Hilbert problem \eqref{RHP_1} satisfies for all $x\in \R$, 
\begin{equation}  \label{asymptotics} 
  m_+(x) =   R(x)    
  \begin{pmatrix} e^{\mathscr{C}(\psi)_+(x)} &  -\frac{\varphi(x)}{1+\varphi(x)} e^{ \mathscr{C}(\psi)_+(x)+ \pii N x}  \\  0 & e^{-\mathscr{C}(\psi)_+(x)}\end{pmatrix}  , 
 \end{equation}   
where  $\mathscr{C}(\psi)$ denotes the Cauchy transform of the function $\psi = \log(1+\varphi)$ and the  $2 \times 2$ matrix $R(z) $ is analytic in the strip $|\Im z| <  \delta/4$  and satisfies the bound:
\begin{equation}\label{R}
\| R(z) -\Id \| \ll \delta^{-1/2}  C_\varphi C_\psi^2 e^{-\pi \delta N} 
\end{equation}
for the matrix supremum norm.
\end{lemma}

The proof of Lemma~\ref{thm:RHP} will be given at the end of this section and it follows closely the proof of the Strong Szeg\H{o} theorem given by Deift in \cite[Example~3]{Deift99}. 
Given this result, let us first complete the proof of Proposition~\ref{thm:sine_asymp}.\\

{\it Proof of Proposition~\ref{thm:sine_asymp}.} First of all, 
we claim that,  if the function $h(z)$ satisfies Assumption~\ref{h_assumption}, then the function $\varphi_t(z) = t(e^{h(z)} -1)$ satisfies Assumption~\ref{psi_assumption} for all $t\in[0,1]$.  Indeed, we have
$$
|\varphi_t (z)| \le  \bigg|\int_0^{h(z)} e^w dw\bigg| \le C_\infty C_1
$$
and for all $|s|\le \delta/2$,
\begin{equation} \label{estimate_2}
\int_{\R} |\varphi_t (x+ is)| dx  \le C_\infty\int_{\R}  \big|h(x+i s)\big| dx 
\le C_\infty C_1 . 
\end{equation}
On the other hand, the condition $ |\Im h(z)| <  \alpha $ guarantees we can choose $c= |\sin \alpha |/2 $ in \eqref{condition_alpha}. Thus, for all $t\in[0,1]$, the functions $\psi_t(z) = \log(1-t + te^{h(z)} )$ are analytic in the strip $| \Im z | <\delta$, 
and  since the $\log$ function is increasing on $\R_+$, we have
for all $x\in\R$,
$$
\psi_t(x)\vee 0 \le  \log(e^{h(x)}\vee 1 ) = h(x)\vee 0
\quad\text{ and }
- \psi_t(x)\wedge 0 \le - \log(e^{h(x)}\wedge 1 ) \le h(x) \wedge 0 .
$$
 These inequalities show that 
\begin{equation} \label{estimate_1}
|\psi_t(x)| \le |h(x)| 
\end{equation}
and it follows that $\psi _t\in L^1\cap L^\infty(\R)$  (in fact there is equality in \eqref{estimate_1} if and only if $t=1$, in which case $\psi_1=h$).   
The bottom line is that, for any $t\in [0,1]$, the  function $\varphi_t(z) = t(e^{h(z)} -1)$ satisfies Assumption~\ref{psi_assumption} with
\begin{equation} \label{constants_2}
C_{\varphi_t} = C_\infty C_1
\quad\text{ and } \quad
C_{\psi_t} = \frac{2C_\infty}{|\sin\alpha|} .  
\end{equation}
In the rest of the proof, we will  denote  for all $t\in[0,1]$ and for all $x\in\R$, 
$$\psi_t(z) = \log(1+t \varphi) = \log(1-t + te^{h(z)} )
\quad\ \text{ and }\quad\
H_t(x) =  \mathscr{C}(\psi_t)_+(x). 
$$
By  Theorem~\ref{thm:Deift}, formula \eqref{asymptotics} implies that 
$$
\mathrm{F}_t(x)= m_+(x)\f(x) = R_t(x) \widetilde{\mathrm{F}}_t(x)
\quad\text{ where }\quad
\widetilde{\mathrm{F}}_t(x)=
 \begin{pmatrix}  \frac{e^{ i \pi N x +H_{t}(x)}}{1+ t\varphi(x)}  \\  e^{- i \pi  N x-H_{t}(x)} \end{pmatrix} ,
$$
and  
$$
\mathrm{G}_t(x) =  (m_+^{-1})^*(x)\g(x) = (R_t(x)^{-1})^*  \widetilde{\mathrm{G}}_t(x)
\quad\text{ where }\quad
\widetilde{\mathrm{G}}_t(x)=
 \begin{pmatrix} e^{i \pi  N x + H_{t}(x)} \\   
-  \frac{e^{-i \pi  N x -H_{t}(x)} }{1+ t\varphi(x)}\end{pmatrix} . 
$$ 
Since the function $h' \in L^1(\R)$,  the function $H_t$ is continuously differentiable on $\R$ for all $t\in[0,1]$. Moreover, by the Plemelj-Sokhotski formula \eqref{PS}, we obtain 
\begin{equation} \label{H'}
2H_{t}' = \frac{t \varphi'}{1+t \varphi}  
+ i \Hi\left(\psi_t'\right) . 
\end{equation}
So, if we differentiate the expression of $ \widetilde{\mathrm{F}}_t(x)$, we obtain
 for all $x\in\R$, 
$$
 \widetilde{\mathrm{F}}_t'(x)= \begin{pmatrix}  
\big( \pi i N  + H_{t}'(x) - \frac{t \varphi'(x)}{1+t \varphi(x)} \big)
\frac{1}{1+ t\varphi(x)} e^{ \pi i N x +H_{t}(x)} \\  
-\big( \pi i N  + H_{t}'(x) \big)e^{-\pi i N x-H_{t}(x)} \end{pmatrix} .
$$
Since, $\mathrm{F}_t^* \mathrm{G}_t =0 $, this shows that 
\begin{align} \notag
&\left(\frac{d\sqrt{\varphi}\mathrm{F}_t}{dx}  (x)\right)^* \left(\sqrt{\varphi(x)} \mathrm{G}_t(x)\right) 
= \varphi(x) \mathrm{F}_t'(x)^*\mathrm{G}_t(x) \\
& \label{FG'} 
\hspace{1cm}= \frac{\varphi(x) }{1+t \varphi(x)} \left(- \pii N  + 2\overline{H_{t}'(x)} - \frac{t \varphi'(x)}{1+t \varphi(x)}  \right)
+\varphi(x)\left(U_t(x) \widetilde{\mathrm{F}}_t(x)\right)^* \widetilde{\mathrm{G}}_t(x) ,
\end{align}
where $U_t(x) = R_t(x)^{-1}  R'_t(x)$. 
Since the matrix  $R_t(z) $ is analytic in the strip $|\Im z| < \delta/4$, 
by Cauchy's formula, $\| R_t'(x) \| \le \delta^{-1} \| R_t(x) -\Id \|$. Thus, the estimate \eqref{R} and \eqref{constants_2} show that 
$$
\|U_t(x)\| \ll \delta^{-3/2}  C_{\varphi_t} C_{\psi_t} e^{-\pi \delta N} 
\ll \frac{C_\infty^3 C_1}{\delta^{3/2}|\sin\alpha|}  e^{-\pi \delta N}  . 
$$
On the other hand, $\Re\{ H_{t}\} = \frac{1}{2}\log(1+t \varphi) =  \frac{\psi_t}{2}$ so that, by the estimate \eqref{estimate_1},  for all $x\in \R$ and all $t\in[0,1]$,  
$$
\| \widetilde{\mathrm{F}}_t(x)\| \le \sqrt{C_\infty}
\quad\text{ and }\quad
 \| \widetilde{\mathrm{G}}_t(x)\| \le C_\infty^{3/2} . 
$$
This proves that the last term in formula \eqref{FG'} is bounded by 
\begin{equation} \label{error_1}
\left|\varphi(x) \left(U_t(x) \widetilde{\mathrm{F}}_t(x)\right)^* \widetilde{\mathrm{G}}_t(x)\right| 
\ll   |\varphi(x)|  \frac{C_\infty^5 C_1}{\delta^{3/2}|\sin\alpha|}  e^{-\pi \delta N}  .
\end{equation}
 By \eqref{estimate_2}, the RHS of  \eqref{error_1} is integrable on $\R\times[0,1]$ and its $L^1$-norm contributes to the   error term in formula~\eqref{Laplace_asymp}. 
 Therefore, by formula \eqref{Laplace_T} and \eqref{FG'},  in order to complete the proof, it remains to show that
\begin{equation} \label{main_asymp}
\int_0^1 \int_\R \frac{\varphi(x) }{1+t \varphi(x)} \left(  \pii N  - 2\overline{H_{t}'(x)}  +\frac{t \varphi'(x)}{1+t \varphi(x)}  \right) dxdt = \pii N \int  h(x) dx 
+ i\pi  \| h \|_{H^{1/2}(\R)}^2  . 
\end{equation}
 This is a rather easy computation that we split in two steps.  First observe that, since 
 $\varphi(x) =e^{h(x)}-1$,  we have
 $$
 \int_0^1 \frac{\varphi(x) }{1+t \varphi(x)}dt 
 = \big[ \log(1+t \varphi(x))\big]_{t=0}^1
 = h(x) . 
 $$
 By Fubini's theorem, this yields the first term in 
 formula \eqref{main_asymp}. Secondly, by formula \eqref{H'}, we have
\begin{align*}
  \frac{\varphi }{1+t \varphi} \left( 2H_{t}'  - \frac{t \varphi'}{1+t \varphi}  \right)
&= i\frac{d\psi_t}{dt} \Hi\left(\frac{d\psi_t}{dx}\right). 
 \end{align*}
By Plancherel's  formula and the linearity of the Fourier transform, this implies that
\begin{align*}
\int_\R \frac{d\psi_t}{dt}(x) \Hi\left(\frac{d\psi_t}{dx}\right)(x) dx
&=  2\pi \int_\R  |\kappa|  \frac{d \widehat{\psi_t}}{dt}(\kappa) \overline{\widehat{\psi_t}(\kappa)} d\kappa \\
&= \pi  \frac{d}{dt} \underbrace{\left(\int_\R  |\kappa|  \widehat{\psi_t}(\kappa) \overline{\widehat{\psi_t}(\kappa)} d\kappa \right)}_{
\displaystyle = \| \psi_t \|_{H^{1/2}(\R)}^2}
\end{align*}
Note that we can pull the differential 
$ \frac{d}{dt}$ out of the integral because both functions $\psi_t , \frac{d\psi_t}{dt} \in L^2(\R)$ by our assumptions.
We conclude that
\begin{align*}
\int_0^1\int_\R  \frac{\varphi (x)}{1+t \varphi(x)} \left( 2H_{t}'(x)  - \frac{t \varphi'(x)}{1+t \varphi(x)}  \right) dxdt
&= i\pi \left( \| \psi_1 \|_{H^{1/2}(\R)}^2 - \| \psi_0 \|_{H^{1/2}(\R)}^2 \right). 
\end{align*}
Since $\psi_1= h$ and $\psi_0=0$, this yields the second term in formula  \eqref{main_asymp}. \qed\\

In order to get the uniform  estimate \eqref{R} in Lemma~\ref{thm:RHP},
  we will need the following correspondence between a well-posed RHP and a singular integral equation.

\begin{theorem}[Kuijlaars \cite{K03}, Theorem 3.1] \label{thm:R}
Let $\Sigma$ be an oriented contour in the complex plane and let $\Delta \in L^2 \cap L^\infty(\Sigma)$ be a $n\times n$ matrix. Suppose that $R$ is a $n\times n$  matrix which is analytic in $\C\backslash\Sigma$ and satisfies
\begin{equation}  \label{RHP_5} 
\begin{cases}
  R_+ = R_- (\Id +\Delta) &\text{on }\Sigma \\
 R(z)\to \Id  &\text{as } z\to\infty 
 \end{cases} .
\end{equation}
We associate to $\Delta$ an operator $\mathscr{C}_\Delta$ acting on $n\times n$ matrices in $L^2(\Sigma)$ and defined by
\begin{equation*}
\Cau_\Delta Y (z) = \lim_{w \to z_+} \frac{1}{2\pi i} \int_\Sigma \frac{Y(s) \Delta(s) }{s-w} ds . 
\end{equation*}
Suppose that $\| \Delta\|_{L^\infty(\Sigma)}$ is sufficiently small so that the equation 
\begin{equation} \label{Cauchy_equation}
X - \Cau_\Delta X  =  \Cau_{\Delta}\Id
\end{equation}
 has a unique solution $X \in L^2(\Sigma)$. Then, there exists a constant $C>0$ which only depends on the contour $\Sigma$ such that
 \begin{equation} \label{X}
 \| X \|_{L^2(\Sigma)} \le  \frac{C\| \Delta\|_{L^2(\Sigma)}}{1- C\| \Delta\|_{L^\infty(\Sigma)} } ,
 \end{equation}
and  the RHP \eqref{RHP_5} has a unique solution which is given by 
\begin{equation*}
R(z) = \Id +  \frac{1}{2\pi i} \int_\Sigma \frac{(\Id + X(s)) \Delta(s) }{s-z} ds 
\end{equation*}
for any $z\in \C \backslash \Sigma$. 
\end{theorem}
Armed with the above theorem, we can construct the solution of the Riemann-Hilbert problem~\eqref{RHP_1} and check the estimate \eqref{R}. \\

{\it Proof of Lemma~\ref{thm:RHP}.}
If $\psi(x) = \log(1+\varphi(x))$, then we can rewrite the jump matrix as
\begin{equation*} 
v(x)  = \begin{pmatrix} e^{\psi(x)}  &-  \varphi(x)e^{\pii N x}  \\  
 \varphi(x)e^{-\pii N x} & ( 1- \varphi(x)^2)  e^{-\psi(x)}\end{pmatrix} .
\end{equation*}

In order to apply the Deift-Zhou steepest method, we use  the decomposition:
\begin{equation} \label{v_2}
v(x)
 =   \underbrace{\begin{pmatrix} 1 & 0\\    \varphi(x)e^{-\psi(x)-\pii N x} &1 \end{pmatrix}}_{\displaystyle=\tilde{A}(x)}
 \underbrace{\begin{pmatrix}  e^{\psi(x)}    &0 \\  0 &  e^{-\psi(x)}   \end{pmatrix}}_{\displaystyle=\Psi(x)}
  \underbrace{ \begin{pmatrix} 1 &   -\varphi(x) e^{-\psi(x) + \pii N x}  \\  0 &1 \end{pmatrix}}_{\displaystyle=A(x)} .
\end{equation}

By assumption, the matrices $A$ and $\tilde{A}$ are analytic in the domain $|\Im z|<\delta$ and we  can define a matrix $M$  on $\C\backslash \{\R \cup \Gamma_\pm \}$ by setting:
\begin{center}
\begin{tikzpicture}[scale=1]
\draw[thick, red] (-4, 1) -- (4, 1) ;
\draw[thick] (-4, 0) -- (4, 0) ;
\draw[thick, red] (-4, -1) -- (4, -1) ;

  \node[red] at ( 0, 1.5) {$M(z) = m(z) $};
    \node at ( 0, .5) {$M(z) =  m(z) A^{-1}(z) $};
    \node at ( 0, -.5) {$M(z) =  m(z) \tilde{A}(z) $};
  \node[red] at ( 0, -1.5) {$M(z) = m(z) $};
  
 \node at ( 4.2, 0) {$\R$};
  \node[red] at ( 5.2, 1) {$\Gamma_+ = \R + i \delta/2$};
  \node[red] at ( 5.2, -1) {$\Gamma_-= \R - i \delta/2$};
  \node at (0, 0) {0};
\end{tikzpicture} .
\end{center}
We deduce from the Riemann-Hilbert problem \eqref{RHP_1},  that the matrix $M$ has the following properties:
\begin{itemize}
	\label{RHP_2}
	\item $M(z) $ is analytic on $\mathbb{C}\setminus (\Gamma_{\pm}\cup \R)$.
	\item $M(z)$ satisfies the following jump conditions
	\begin{equation}\label{MRHP}
	\begin{cases}
		M_{+}(z) &= M_{-}(z)A(z), \qquad z \in \Gamma_{+}\\
		M_{+}(x) &= M_{-}(x)\Psi(x), \qquad x \in \mathbb{R}\\
		M_{+}(z) &= M_{-}(z)\tilde{A}(z), \qquad z \in \Gamma_{-}
	\end{cases}
	\end{equation} \vspace{-.7cm}
	\item $M(z) \to \Id$ as $z \to \infty$.
\end{itemize}
Since $\psi\in L^1(\R)$, its Cauchy transform is well-defined, c.f.~\eqref{Cauchy_T}, and we claim that the global parametrix is given by 
\begin{equation}\label{P_matrix}
P(z)= \begin{pmatrix} e^{\mathscr{C}(\psi)(z)}  &0 \\  0 & e^{-\mathscr{C}(\psi)(z)}  \end{pmatrix} . 
\end{equation}
Indeed, it is straightforward to check using formula \eqref{PS}, that it solves the RHP:
\begin{itemize}
	\item $P(z) $ is analytic on $\mathbb{C}\backslash\R$.
	\item $P(z)$ satisfies the condition $  P_+ = P_- \Psi$ on $\R$
	\item $P(z) \to \Id$ as $z \to \infty$.
\end{itemize}
The matrix $P(z)$ is invertible on $\C\backslash\R$ and, if we let $R=MP^{-1}$, 
then the matrix $R$ solves the RHP:
\begin{itemize}
\item $R(z)$ is analytic on $\mathbb{C}\setminus(\mathbb{R}\cup \Gamma_{\pm})$.
	\item If we let $\Delta = P A P^{-1} - \Id $ and $\tilde{\Delta} = P \tilde{A} P^{-1}-\Id $, then $R(z)$ satisfies the  jump conditions
 \begin{equation} \label{RHP_4}
 \begin{cases} 
 R_+ = R_- ( \Id + \Delta )  &\text{on } \Gamma_+\\
R_+ = R_- ( \Id + \tilde{\Delta})  &\text{on } \Gamma_- 
\end{cases}
\end{equation}\vspace{-.7cm}
\item $R(z) \to \Id$ as $z \to \infty$.
\end{itemize}
 Moreover observe that by formulae \eqref{v_2} and \eqref{P_matrix}, 
\begin{align*}
&\Delta (z)=  \begin{pmatrix} 0 &    - \varphi(z) e^{-\psi(z) + \pii N z +2\mathscr{C}(\psi)(z)}
\\ 0 &0 \end{pmatrix} \ ,\hspace{.5cm}\forall z\in \Gamma_+\\
&\tilde{\Delta} (z)=  \begin{pmatrix} 0 &0 \\   \varphi(z) e^{-\psi(z) - \pii N z -2\mathscr{C}(\psi)(z) } &0 \end{pmatrix} \ ,\hspace{.8cm}\forall z\in \Gamma_- .
\end{align*}

The function $\psi$ is real valued on $\R$ and for any $z\in\Gamma_{\pm}$, we have
\begin{equation*}
 \Re\big\{ \mathscr{C}(\psi)(z) \big\} =  \frac{\delta}{4\pi } \int_{\R} \frac{ \psi (x)}{(\Re z-x)^2 + \delta^2/4 } dx ,
\end{equation*} 
so that
\begin{equation*}
 \big|\Re\big\{ \mathscr{C}(\psi)(z) \big\} \big| \le \| \psi \|_{L^\infty(\R)} /2 .
 \end{equation*}
Combined with Assumption~\ref{psi_assumption}, this  implies that for any $z\in \Gamma_+$,
$$
\|\Delta(z)\| = \left|\frac{\varphi(z)}{1+\varphi(z) }\right|  e^{\Re\{ \pii N z +2\mathscr{C}(\psi)(z) \}} \le |\varphi(z)| C_\psi^2 e^{-\pi \delta N} , 
$$
so that the matrix $\Delta \in L^1 \cap L^\infty(\Gamma_+) $ and 
\begin{equation} \label{Delta_estimate}
 \| \Delta \|_{L^\infty(\Gamma_+)} \vee 
  \| \Delta \|_{L^1(\Gamma_+)}   \le    C_1 C_\psi^2 e^{- \pi\delta N} .
  \end{equation}
Note that, since $\tilde{\Delta} (z) =  -\Delta(\bar{z})^*$ for any $z\in \Gamma_-$, the matrix $\tilde{\Delta}$ also satisfies the estimate \eqref{Delta_estimate}.  Hence, by Theorem~\ref{thm:R}, we obtain that
the solution of the RHP \eqref{RHP_4} is unique and is given by
\begin{equation} \label{R_N}
R(z) = \Id +  \frac{1}{2\pi i} \int_{\Gamma_+} \frac{(\Id + X(s)) \Delta(s) }{s-z} ds 
+  \frac{1}{2\pi i} \int_{\Gamma_-} \frac{(\Id + \tilde{X}(s)) \tilde{\Delta}(s) }{s-z} ds 
\end{equation}
where the  $2\times 2$ matrices $X$ and $\tilde{X}$ solve appropriate singular integral equations, c.f.~\eqref{Cauchy_equation}.  Moreover, a simple estimate using the Cauchy-Schwarz inequality shows that for any $z \in \C$ such that $| \Im z | \le \delta/4$, 
 \begin{equation*}
\left\|  \frac{1}{2\pi i} \int_{\Gamma_+} \frac{(\Id + X(s)) \Delta(s) }{s-z} ds\right\| 
\le \frac{1}{2\sqrt{\pi \delta}}  \left(   \| \Delta\|_{L^2(\Gamma_+)}  + 
 \| \Delta\|_{L^\infty(\Gamma_+)} \| X\|_{L^2(\Gamma_+)} \right) .
\end{equation*}
Hence, as  $\| \Delta \|_{L^\infty(\Gamma_+)}\to0$ as $N\to\infty$, combining the bounds \eqref{X}  and \eqref{Delta_estimate}, we have proved that
 \begin{equation*}
\left\|  \frac{1}{2\pi i} \int_{\Gamma_+} \frac{(\Id + X(s)) \Delta(s) }{s-z} ds\right\| 
\ll \delta^{-1/2}C_\psi^2 C_\varphi e^{-\pi \delta N}  .
\end{equation*}
A similar estimate holds for the integral over $\Gamma_-$ and, by formula \eqref{R_N},  this proves the bound~\eqref{R}.
Going back to the original problem, we see that for all $x\in \R$, 
$m_+(x) = M_+(x) A(x)$ where $M_+(x) = R(x) P_+(x) $. Finally, we deduce \eqref{asymptotics} from formulae \eqref{v_2} and \eqref{P_matrix}.    \qed  

\subsection{Strong Gaussian approximation for the CUE} \label{sect:BO}

The goal of this section is to prove the Gaussian asymptotics \eqref{qmoms} for the CUE. We work with the statistic \eqref{smoothing} and it will be convenient to use the notation
\begin{equation}
h_{\mathbf{u},\boldsymbol{\epsilon}}(x) = \gamma \pi\sum_{k=1}^{q}t_{k}\int_{u_k-\ell/2}^{u_k+\ell/2}\phi_{\epsilon_k}(x-t)\,dt
\end{equation}
We shall suppose that the mollifier $\phi$ belongs to the Schwartz class, \textit{i.e.} $\phi$ is a smooth function such that for all $\eta>0$ we have $\phi(x) = O(|x|^{-\eta})$ as $x \to \pm \infty$. The main goal of this section will be to prove 
\begin{proposition}
	\label{prop:cue}
Consider a CUE matrix $U$ of size $N \times N$ with eigenangles $\theta_1,\ldots,\theta_N \in [-\pi,\pi)$ and mesoscopic scale $0 < \alpha < 1$. Take $\epsilon^{*} = \mathrm{min}\{\epsilon_1,\ldots,\epsilon_q\}$. Let $\delta > 0$ and suppose that we have the following bound $N^{\alpha}/(N\epsilon^{*}) = O(N^{-\delta})$. Then for all Schwartz $\phi$, the following Gaussian estimate holds as $N \to \infty$
\begin{equation}
\mathbb{E}\left[\mathrm{exp}\left(\sum_{j=1}^{N}h_{\mathbf{u},\boldsymbol{\epsilon}}(N^{\alpha}\theta_j)-\mathbb{E}\left(\sum_{j=1}^{N}h_{\mathbf{u},\mathbf{\epsilon}}(N^{\alpha}\theta_j)\right)\right)\right] = \mathrm{exp}\left(\frac{\|h_{\mathbf{u},\boldsymbol{\epsilon}}\|^{2}_{H^{1/2}}}{2}\right)(1+\mathcal{E}^{\mathrm{S}}_{N})e^{\mathcal{E}_{N}^{\mathrm{glob}}} \label{cueGaussEst}
\end{equation}
where for any $\eta>0$, the \textit{smoothing} error term satisfies the bound
\begin{equation}
|\mathcal{E}_{N}^{\mathrm{S}}| \leq C\gamma^{2}\|t\|_{1}^{\eta_2}N^{-\eta}\mathrm{exp}\left(\gamma^{2}\|t\|_{1}^{\eta_2}N^{-\eta}\right) \label{gausscueerror}
\end{equation}
for some $\eta_2>0$, while the \textit{global} error term satisfies
\begin{equation}
|\mathcal{E}_{N}^{\mathrm{glob}}| \leq C\|t\|_{1}^{2}\log(\ell/\epsilon^{*})\ell N^{-\alpha} \label{globerr}
\end{equation}
uniformly in compact subsets of the parameters $u_{1},\ldots,u_{q} \in \mathbb{R}$.
\end{proposition}
With the Gaussian approximation above we may complete the proofs of the main theorems for the CUE.
\begin{proof}[Proof of Theorems \ref{th:gmc} and \ref{th:moments}]
Proposition \ref{prop:cue} shows that the regularized statistic \eqref{smoothing} satisfies the strong Gaussian approximation as in Assumption \ref{exp_moments} for all $q\in \N$ (hence Assumption \ref{cvg} holds as well). The fact that the associated covariance kernel satisfies Assumption \ref{kernel_asymp} is a consequence of the computations carried out in section~\ref{sect:covariance}, namely Theorem \ref{thm:covariance}. Therefore we deduce Theorems \ref{th:gmc} and \ref{th:moments} as corollaries of Theorems \ref{thm:L^1_phase} and \ref{thm:L^2_phase} respectively.
\end{proof}

\begin{remark}
The error \eqref{gausscueerror} is called the smoothing error because it fails when $\epsilon \to 0$ quickly enough that $\epsilon \sim N^{\alpha}/N$. Beyond this regime the asymptotics in \eqref{cueGaussEst} are no longer valid and one enters the transition regime of Fisher-Hartwig asymptotics, see \cite{Kra11, DIK13} for a review. On the other hand, \eqref{globerr} is controlled by the relative numbers of eigenvalues sampled by the statistic and is only small in the mesoscopic regime $0 < \alpha < 1$. 
\end{remark}
As for the sine process and Proposition \ref{thm:sine_asymp}, the proof of Proposition \ref{prop:cue} is also based on the existence of integrable and determinantal structures in the CUE. In fact the Laplace transform of \eqref{smoothing} could also be written as a Fredholm determinant which could then be analysed with an appropriate Riemann-Hilbert problem (as was done in \cite{Deift99} for the global scale $\alpha=0$ and $\epsilon>0$ fixed). However, for the CUE we will give a more elementary proof that does not involve Riemann-Hilbert techniques, but instead relies on an `algebraic miracle' known (for historical reasons) as the Borodin-Okounkov-Case-Geronimo formula. Unlike our Riemann-Hilbert computation, the proof we give here does not impose any analyticity condition on the mollifier $\phi$, but we do require it to be smooth in general.	Another crucial difference is that the CUE has a macroscopic regime (defined by \eqref{smoothing} with $\alpha=0$) which has no analogue for the sine process.

When working with the CUE it is convenient to expand such functions in a Fourier series. Therefore in what follows we are going to work with the periodisation
\begin{equation}
h^{(2\pi)}_{\mathbf{u},\boldsymbol{\epsilon}}(\theta) = \sum_{a=-\infty}^{\infty}h_{\mathbf{u},\boldsymbol{\epsilon}}(N^{\alpha}(\theta+2\pi a)).
\end{equation}
Then $h^{(2\pi)}_{\mathbf{u},\boldsymbol{\epsilon}}(\theta+2\pi) = h^{(2\pi)}_{\mathbf{u},\boldsymbol{\epsilon}}(\theta)$ is $2\pi$-periodic. The periodisation has the convenient property that its Fourier coefficients are given explicitly in terms of the Fourier transform of $h_{\mathbf{u},\boldsymbol{\epsilon}}$. 
\begin{equation}
\frac{1}{2\pi}\int_{-\pi}^{\pi}h^{(2\pi)}_{\mathbf{u},\boldsymbol{\epsilon}}(\theta)e^{-ik\theta}\,d\theta = N^{-\alpha}\int_{\mathbb{R}}h_{\mathbf{u},\boldsymbol{\epsilon}}(2\pi x)e^{-2i\pi kxN^{-\alpha}}\,dx \label{fouriercoeffs}
\end{equation}
Furthermore, the linear statistics of $h^{(2\pi)}_{\mathbf{u},\boldsymbol{\epsilon}}(\theta)$ are uniformly close to those of $h_{\mathbf{u},\boldsymbol{\epsilon}}(\theta)$. This follows from the rapid decay of $\phi$, since the difference between the two is deterministically bounded by 
\begin{equation}
\sum_{|a|>0}\sum_{j=1}^{N}\sum_{k=1}^{q}\epsilon_{k}^{\eta-1}\int_{x_{j}-u_{k}-l/2}^{x_{j}-u_{k}+l/2}|s|^{-\eta}\,ds\leq \|t\|cN(\epsilon_{*}/N^{\alpha})^{\eta} \label{deterror}
\end{equation}
where $x_{j} = N^{\alpha}(\theta_j+2\pi a)$ and we used that $lN^{-\alpha} \to 0$ as $N \to \infty$. Choosing $\eta>0$ large enough we can always ensure that \eqref{deterror} goes to $0$ as $N \to \infty$. Hence it will suffice to always work with the periodisation. 
\begin{proposition}[Macroscopic approximation]
	\label{macrocue}
Define the quantity
\begin{equation}
E_{\mathbf{u},\mathbf{\epsilon}} := \frac{1}{(2\pi)^{2}}N^{-\alpha}\sum_{k=1}^{\infty}kN^{-\alpha}|\hat h_{\mathbf{u},\boldsymbol{\epsilon}}(k/(2\pi N^{\alpha}))|^{2} \label{macrocov}
\end{equation}
and suppose that the hypotheses of Proposition \ref{prop:cue} are satisfied. Then we have the Gaussian estimate
\begin{equation}
\label{expcue}
\mathbb{E}\left[\mathrm{exp}\left(\sum_{j=1}^{N}h^{(2\pi)}_{\mathbf{u},\boldsymbol{\epsilon}}(\theta_j)-\mathbb{E}\left(\sum_{j=1}^{N}h^{(2\pi)}_{\mathbf{u},\mathbf{\epsilon}}(\theta_j)\right)\right)\right] = \mathrm{exp}\left(E_{\mathbf{u},\boldsymbol{\epsilon}}\right)(1+\mathcal{E}_{N}^{\mathrm{S}})
\end{equation}
where the error term $\mathcal{E}_{N}^{\mathrm{S}}$ satisfies \eqref{gausscueerror} and is uniform in the variables $u_{1},\ldots,u_{q}$ varying in compact subsets of $\mathbb{R}$.
\label{strongcueprop}
\end{proposition}
To prove Proposition \ref{strongcueprop}, the main idea is to exploit the fact that for the CUE, the left-hand side of \eqref{expcue} can be written exactly as an $N \times N$ Toeplitz determinant involving the Fourier coefficients of the periodic function $w_{\mathbf{u},\mathbf{\epsilon}}(\theta) = e^{h^{(2\pi)}_{\mathbf{u},\mathbf{\epsilon}}(\theta)}$. The representation as a determinant follows from the following well known but remarkable chain of equalities for the (un-centered) left-hand side of \eqref{cueGaussEst}:
\begin{align}
&\mathbb{E}\left[\mathrm{exp}\left(\sum_{j=1}^{q}X_{N,\epsilon_N}(u_j)\right)\right] \label{expcuetoweyl}\\
&= \frac{1}{(2\pi)^{N}N!}\int_{-\pi}^{\pi}\ldots \int_{-\pi}^{\pi}\prod_{j=1}^{N}w(\theta_{j})\prod_{1 \leq p < q \leq N}|e^{i\theta_{q}}-e^{i\theta_{p}}|^{2}d\theta_{1} \ldots d\theta_{N} \label{weyl}\\
&=\frac{1}{(2\pi)^{N}N!}\int_{-\pi}^{\pi}\ldots\int_{-\pi}^{\pi}\det\{w(\theta_{j})e^{i(k-1)\theta_{j}}\}_{j,k=1}^{N}\det\{e^{-i(k-1)\theta_{j}}\}_{j,k=1}^{N}d\theta_{1} \ldots d\theta_{N} \label{vander}\\
&=\det\bigg\{\frac{1}{2\pi}\int_{-\pi}^{\pi}w(\theta)e^{-i(k-j-2)\theta}\,d\theta\bigg\}_{j,k=1}^{N} \label{andre}\\
&= \det T_{N}(w) \label{toeplitz}
\end{align}
where $T_{N}(w)$ is the $N \times N$ Toeplitz matrix $\{\hat{w}_{k-j-2}\}_{j,k=1}^{N}$. That \eqref{expcuetoweyl} equals \eqref{weyl} is a consequence of the Weyl integration formula for the unitary group. Then \eqref{vander} writes the product of differences in \eqref{weyl} as a product of two determinants and finally \eqref{andre} is a consequence of the Andrejeff identity. 

Thus our task will be to calculate the asymptotics of the Toeplitz determinant in \eqref{toeplitz} as $N \to \infty$. The function $w(\theta)$ is called the \textit{symbol}. For smooth and $N$-independent symbols, such asymptotics are well known from the strong Szeg\H{o} limit theorem. However, our symbol is $N$-dependent and furthermore the quantity $E_{\mathbf{u},\mathbf{\epsilon}}$ is divergent in the limit $N \to \infty$. This is due to the fact that our symbol becomes discontinuous as $\epsilon_{N} \to 0$ and therefore does not belong to the $H^{1/2}$-space. The following formula is our main tool for establishing the strong Gaussian approximation in the CUE.
\begin{theorem}[Borodin-Okounkov-Case-Geronimo formula]
Let $w(\theta)$ a periodic function on the interval $\theta \in [0,2\pi)$ and that $\log w(\theta)$ has Fourier coefficients $\hat{L}_{k}$. Suppose that we have the expansion
\begin{equation}
\log w(\theta) = \sum_{k=-\infty}^{\infty}\hat{L}_{k}e^{ik\theta}, \qquad \sum_{k=1}^{\infty}k|\hat{L}_{k}|^{2} < \infty
\end{equation}
In terms of the quantities $b(\theta) = 1/c(\theta)$ and
\begin{equation}
c(\theta) = \mathrm{exp}\left(\sum_{k=1}^{\infty}\hat{L}_{k}e^{ik\theta}-\sum_{k=1}^{\infty}\hat{L}_{-k}e^{-ik\theta}\right) \label{ctheta}
\end{equation}
we have the following identity
\begin{equation}
\frac{\det(T_{N}(w))}{\mathrm{exp}\left(N\hat{L}_{0}+\sum_{k=1}^{\infty}k|\hat{L}_{k}|^{2}\right)} = \mathrm{det}(I-R_{N}H(b)H(\tilde{c})R_{N}) \label{BOformula}
\end{equation}
where $R_{N}$ is the projection operator on $\ell^{2}(N+1,N+2,\ldots)$ and $H(b)$, $H(\tilde{c})$ are infinite Hankel matrices corresponding to the sequence of Fourier coefficients of $b(\theta)$ and $c(\theta)$,
\begin{equation}
	H(b) = \begin{pmatrix} 
		b_{1} & b_{2} & b_{3} \ldots\\
		b_{2} & b_{3} & b_{4} \ldots\\
		b_{3} & b_{4} & b_{5} \ldots\\
		\ldots & \ldots & \ldots
		\end{pmatrix}, 
		\qquad
		H(\tilde{c}) = \begin{pmatrix} 
			c_{-1} & c_{-2} & c_{-3} \ldots\\
			c_{-2} & c_{-3} & c_{-4} \ldots\\
			c_{-3} & c_{-4} & c_{-5} \ldots\\
			\ldots & \ldots & \ldots.
		\end{pmatrix}
\end{equation}
\end{theorem}
\begin{proof}
There are many proofs in the literature, for example the one of Borodin and Okounkov \cite{BO00} led to an increased interest in the formula. For a comprehensive proof under the conditions mentioned here see Simon \cite[Theorem 6.2.14]{Sim05}, which also provides a detailed historical discussion of \eqref{BOformula} and its original discovery in \cite{GC79}.
\end{proof}
Note that here $\log w(\theta) = h^{(2\pi)}_{\mathbf{u},\boldsymbol{\epsilon}}$ and the quantity $\sum_{k=1}^{\infty}k|\hat{L}_{k}|^{2}$ in \eqref{BOformula} is precisely $E_{\mathbf{u},\boldsymbol{\epsilon}}$ in \eqref{macrocov}. Furthermore, an easy computation shows that $N\hat{L}_{0} = \mathbb{E}(\sum_{j=1}^{N}h^{(2\pi)}_{\mathbf{u},\boldsymbol{\epsilon}}(\theta_j))$ which corresponds to a re-centering by the expectation. Hence formula \eqref{BOformula} implies
\begin{equation}
\mathbb{E}\left[\mathrm{exp}\left(\sum_{j=1}^{N}h^{(2\pi)}_{\mathbf{u},\boldsymbol{\epsilon}}(\theta_j)-\mathbb{E}\left(\sum_{j=1}^{N}h^{(2\pi)}_{\mathbf{u},\mathbf{\epsilon}}(\theta_j)\right)\right)\right] = \mathrm{exp}\left(E_{\mathbf{u},\boldsymbol{\epsilon}}\right)\mathrm{det}(I-R_{N}H(b)H(\tilde{c})R_{N}).
\end{equation}
The next lemma gives us the necessary estimate on the above determinant, thus concluding the proof of Proposition \ref{macrocue}.
\begin{lemma}
Suppose the hypotheses of Proposition \ref{prop:cue} are satisfied. Then for any $\eta>0$, there exists $\eta_2>0$ such that we have the following estimate
\begin{equation}
|\det(1-R_{N}H(b)H(\tilde{c})R_{N})-1| \leq C\gamma^{2}\|t\|_{1}^{\eta_2}N^{-\eta}\mathrm{exp}\left(\gamma^{2}\|t\|_{1}^{\eta_2}N^{-\eta}\right)
\end{equation}	
uniformly for variables $u_{1},\ldots,u_{q}$ belonging to a compact subset of $\mathbb{R}$.
\label{bo-estimate}
\end{lemma}

\begin{proof}
The following inequality is an easy consequence of standard properties of the Fredholm determinant
\begin{equation}
|\mathrm{det}(I+A)-1| \leq e^{\|A\|_{1}}-1 \label{detbound}
\end{equation}
In our case $A=-R_{N}H(b)H(\tilde{c})R_{N} = -|H(\tilde{c})R_{N}|^{2} \leq 0$ using the fact that $H(b) = H(\tilde{c})^{\dagger}$, so $-A$ is a positive operator. Thus we can write the bound \eqref{detbound} in terms of the Hilbert-Schmidt norm
\begin{equation}
|\mathrm{det}(I-R_{N}H(b)H(\tilde{c})R_{N})-1| \leq e^{\|H(\tilde{c})R_{N}\|_{2}^{2}}-1 \leq \|H(\tilde{c})R_{N}\|_{2}^{2}e^{\|H(\tilde{c})R_{N}\|_{2}^{2}}
\end{equation}
which can be computed explicitly (see 6.2.57 in \cite{Sim05})
\begin{equation}
\|H(\tilde{c})R_{N}\|_{2}^{2} = \sum_{k=1}^{\infty}k|c_{k+N}|^{2}. \label{HSformula}
\end{equation}
We now proceed to estimate the Fourier coefficients of the function $c(\theta)$ in \eqref{ctheta}, which clearly satisfies $|c(\theta)|=1$. We have
\begin{equation}
c_{k+N} = \frac{1}{2\pi}\int_{-\pi}^{\pi}e^{-i(k+N)\theta}e^{2iu(\theta)}\,d\theta \label{stationaryphase}
\end{equation}
where
\begin{equation}
u(\theta) = \Im\bigg\{\sum_{k=1}^{\infty}\hat{L}_{k}e^{ik\theta}-\sum_{k=1}^{\infty}\hat{L}_{-k}e^{-ik\theta}\bigg\}. \label{udef}
\end{equation}
The idea is to exploit cancellations in the integral \eqref{stationaryphase} for large $N$ coming from rapid oscillations of the factor $e^{-i(k+N)}$. To this end, we integrate by parts $p$ times, obtaining
\begin{equation}
c_{k+N} = \frac{2^{p}}{2\pi}\frac{1}{(k+N)^{p}}\int_{-\pi}^{\pi}e^{-i(k+N)}e^{-u(\theta)}\left(\frac{d^{p}}{d\theta^{p}}e^{u(\theta)}\right)e^{2iu(\theta)}\,d\theta \label{parts}
\end{equation}
The function $e^{-u(\theta)}\left(\frac{d^{p}}{d\theta^{p}}e^{u(\theta)}\right)$ is a polynomial in $u(\theta)$ and all its derivatives up to order $p$. We have the explicit formula
\begin{equation}
e^{-u(\theta)}\left(\frac{d^{p}}{d\theta^{p}}e^{u(\theta)}\right) = \sum_{m=1}^{p}\sum_{\substack{k_{1}+2k_{2}+\ldots+pk_{p}=p\\k_{1}+k_{2}+\ldots+k_{p}=m}}c_{p,\vec{k}}\prod_{l=1}^{p}u^{(l)}(\theta)^{k_{l}} \label{expderiv}
\end{equation}
where $c_{p,\vec{k}}$ are some combinatorial coefficients. We now proceed to estimate $u^{(l)}(\theta)$. Clearly the interchange of derivative and sum in \eqref{udef} is valid as the partial sums of all derivatives are uniformly convergent. To calculate the coefficients $c_{k+N}$, note that by \eqref{fouriercoeffs} and the convolution theorem, we have
\begin{align}
|u^{(l)}(\theta)| &\leq 2\sum_{k=1}^{\infty}k^{l}|\hat{L}_{k}| \leq \frac{2\pi\gamma}{N^{\alpha}}\sum_{k=1}^{\infty}k^{l}\sum_{j=1}^{q}|t_{j}|\bigg{|}\frac{e^{-2i\pi(u_{j}+l/2)kN^{-\alpha}}-e^{-2i \pi(u_{j}-l/2)kN^{-\alpha}}}{2\pi kN^{-\alpha}}\bigg{|}|\hat{\phi}(\epsilon_{j} k/(2\pi N^{\alpha}))|\\
&\leq 2\gamma\sum_{j=1}^{q}\rho_{j}^{l-1}\sum_{k=1}^{\infty}(k/\rho_j)^{l-1}|t_{j}||\hat{\phi}(k/(2\pi\rho_j))|\\
&\sim 2\gamma\sum_{j=1}^{q}|t_{j}|\rho_{j}^{l}\int_{0}^{\infty}k^{l-1}|\hat{\phi}(k/(2\pi))|\,dk
\end{align}
where $\rho_{j} = N^{\alpha}/\epsilon_{j} \to \infty$. Hence there exists a constant $C>0$ such that
\begin{equation}
|u^{(l)}(\theta)| \leq C\sum_{j=1}^{q}\rho_{j}^{l}|t_j| \leq C(N^{\alpha}/\epsilon^{*})^{l}\sum_{j=1}^{q}|t_{j}| \label{derivbound}
\end{equation}
uniformly in $u_{1},\ldots,u_{q}$. Inserting \eqref{derivbound} into \eqref{expderiv} yields the estimate
\begin{equation}
\bigg{|}e^{-u(\theta)}\left(\frac{d^{p}}{d\theta^{p}}e^{u(\theta)}\right)\bigg{|} \leq \sum_{m=1}^{p}\sum_{\substack{k_{1}+2k_{2}+\ldots+pk_{p}=p\\k_{1}+k_{2}+\ldots+k_{p}=m}}c_{p,\vec{k}}\prod_{l=1}^{p}(\|t\|N^{\alpha}/\epsilon^{*})^{lk_{l}} = C(\|t\|N^{\alpha}/\epsilon^{*})^{p}
\end{equation}
Inserting this into \eqref{parts} gives the following bound on the Fourier coefficients of $c(\theta)$
\begin{equation}
|c_{k+N}| \leq C\gamma\left(\frac{\|t\|N^{\alpha}/\epsilon^{*}}{k+N}\right)^{p}
\end{equation}
and the corresponding bound on the Hilbert-Schmidt norm
\begin{equation}
\begin{split}
\sum_{k=1}^{\infty}k|c_{k+N}|^{2} &\leq C\gamma\sum_{k=1}^{\infty}k\left(\frac{\|t\|N^{\alpha}/\epsilon^{*}}{k+N}\right)^{2p} \leq C\int_{1}^{\infty}k\left(\frac{\|t\|N^{\alpha}/\epsilon^{*}}{k+N}\right)^{2p}\,dk\\
&= O(N^{2}(\|t\|N^{\alpha}/(\epsilon^{*}N))^{2p}).
\end{split}
\end{equation}
This completes the proof of the lemma.
\end{proof}

The quantity \eqref{macrocov} is close to a Riemann sum. We need good estimates on the error in the approximation, which is the purpose of the next lemma. This error is generically of order $N^{-\alpha}$ in the mesoscopic regime. More generally, for the case of diverging interval length $l = L(N) \to \infty$ the error becomes of order $\log(L(N))L(N)N^{-\alpha}$.
\begin{proposition}[Macroscopic to mesoscopic]
\label{mesoapprox}
Consider the quantity \eqref{macrocov}. We have the uniform approximation
	\begin{equation}
E_{\mathbf{u},\mathbf{\epsilon}} = \frac{1}{(2\pi)^{2}}\int_{0}^{\infty}k|\hat{h}_{\mathbf{u},\mathbf{\epsilon}}(k)|^{2}\,dk+O(\log(L(N)) L(N)N^{-\alpha})
\end{equation}
	\end{proposition}
\begin{proof}
	We use the fact that the error in a Riemann sum approximation is given by the step size (here $N^{-\alpha}$) multiplied by the total variation norm of the function in question. Hence we have to estimate the quantity
	\begin{align} \mathcal{E} := N^{-\alpha}\int_{0}^{\infty}\bigg{|}\frac{d}{dk}\left(k\bigg{|}\sum_{j=1}^{q}t_{j}\frac{e^{-2\pi i k(u_{j}+l/2)}-e^{-2\pi ik(u_{j}-l/2)}}{2\pi k}\hat{\phi}(k\epsilon_j)\bigg{|}^{2}\,\right) \,\bigg{|}dk
	\end{align}
Using the identity
	\begin{equation}
	\begin{split}
		&\bigg{|}\sum_{j=1}^{q}\hat{\phi}(k\epsilon_j)t_{j}(e^{-2\pi i k(u_{j}+l/2)}-e^{-2\pi ik(u_{j}-l/2)})\bigg{|}^{2}\\
		&= 8\sin^{2}(2\pi k l/2)\sum_{j_1 \leq j_2}t_{j}t_{k}\hat{\phi}(k\epsilon_{j_1})\hat{\phi}(-k\epsilon_{j_2})\cos(2\pi k(u_{j_1}-u_{j_2}))\\
		\end{split}
	\end{equation}
	and changing variables $k \to k/\epsilon^{*}$, we see that it is sufficient to bound the quantity
	\begin{equation}
		I_{\epsilon,u} := \int_{0}^{\infty}\bigg{|}\frac{d}{dk}\frac{\sin^{2}(\pi k l)\cos(2\pi k(u_{j_2}-u_{j_1}))}{k}\hat{\phi}(k\epsilon_{j_1})\hat{\phi}(-k\epsilon_{j_2})dk\bigg{|} \label{I1eps}
	\end{equation}
	uniformly in $u$ as $\epsilon \to 0$. Computing the derivative yields four terms $I = I_{1}+I_{2}+I_{3}+I_{4}$ coming from differentiating $1/k$, the two trigonometric terms and the functions $\hat{\phi}$, respectively. The contribution coming from the derivative of $1/k$ is bounded by
	\begin{equation}
		I_{1} \leq \int_{0}^{\infty}\frac{\sin^{2}(\pi k l)}{k^{2}}\,dk = \frac{\pi^{2}}{2}l
	\end{equation}
To compute the contribution coming from the trigonometric terms we change variables $k \to k/\epsilon^{*}$ so that the argument of the Fourier transform is diverging for $k>1$. Then the contribution is dominated by the interval $k \in [0,1]$ due to the rapid decay of $\hat{\phi}(k)$ and we get the bounds
	\begin{equation}
	\begin{split}
		I_{2} &\leq Cl\int_{0}^{1}\bigg{|}\frac{\sin(2\pi k l/\epsilon^{*})}{k}\bigg{|}\,dk = O(l\log(l/\epsilon^{*}))\\
		I_{3} &\leq C2\pi|u_{j_2}-u_{j_1}|\int_{0}^{1}\frac{\sin^{2}(\pi k l/\epsilon^{*})}{k}\,dk = O(|u_{j_2}-u_{j_1}|\log(l/\epsilon^{*}))
		\end{split}
	\end{equation}
A similar estimate yields
\begin{equation}
I_{4} \leq \max\{\epsilon_{j_1},\epsilon_{j_2}\}\int_{0}^{1}\frac{\sin^{2}(\pi k l/\epsilon^{*})}{k}\,dk = O(\max\{\epsilon_{j_1},\epsilon_{j_2}\}\log(l/\epsilon^{*}))
\end{equation}
Multiplying these estimates by the step size $N^{-\alpha}$ we get the error in the Riemann sum approximation
\begin{equation}
\mathcal{E} \leq CN^{-\alpha}\sum_{j_1 \leq j_2}|t_{j_2}t_{j_1}|l\log(l/\epsilon^{*})
\end{equation}
which completes the proof of the Proposition.
\end{proof}

\appendix

\section{Properties of the Hilbert transform, $H^{1/2}$--noise, and a log--correlated Gaussian process} \label{appendix}

We define the Fourier transform of any function $f\in L^1(\R)$ by
\begin{equation} \label{Fourier_T}
\F(f)(\kappa) = \hat{f}(\kappa) =  \int_\R e^{-\pii \kappa x } f(x) dx . 
\end{equation}
The operator $\F$ can be extended to a unitary transformation on $L^2(\R)$ with the Plancherel formula:
\begin{equation*} 
\int_\R \hat{f}(\kappa) \overline{\hat{g}(\kappa)} d\kappa  =  \int_\R  f(x) \overline{g(x)} dx,  
\end{equation*}
for any functions $f, g \in L^2(\R\to\C)$. \\

We define the Hilbert transform of any function $f\in L^1(\R)$, 
\begin{equation} \label{Hilbert_T}
\Hi(f)(x) =  \frac{1}{\pi}\dashint_\R \frac{f(u)}{x-u} du ,
\end{equation}
where the integral is defined in the principal value sense. The Hilbert transform can also  be extended to a bounded operator on $L^2(\R)$ which satisfies:
\begin{equation}\label{Stein}
\widehat{\Hi(f)}(\kappa) =  - i \sgn(\kappa)   \hat{f}(\kappa)
\end{equation}
where $\sgn(\cdot)$ is the sign function. In particular, the identity \eqref{Stein} implies that $\Hi$ is invertible on $L^2(\R)$ and $\Hi^{-1} = -\Hi$.  
Moreover, let us mention that if $f\in L^1\cap L^2(\R)$ is absolutely continuous (i.e.~$f'\in L^1(\R)$), then the Hilbert transform of the function $f$ is differentiable on $\R$ and
\begin{equation} \label{Hilbert_1}
\frac{d\Hi(f)}{d\kappa} = \Hi(f') . 
\end{equation}


We define the Cauchy transform of any function $f\in L^1(\R)$, 
\begin{equation} \label{Cauchy_T}
\mathscr{C}(f)(z) =\frac{1}{\pii}\int_\R \frac{f(x)}{x-z} dx .
\end{equation}
This function is analytic in both the lower and upper half planes, denoted $\C_{\pm}$. Moreover, its boundary values are given (in $L^2$ or pointwise if the limits make sense) by 
 the  Plemelj-Sokhotski formula, for all $x\in\R$,
\begin{equation} \label{PS} 
\mathscr{C}(f)_\pm(x) = \pm \frac{f(x)}{2}  + \frac{i}{2} \Hi(f)(x)  . 
\end{equation}

We define the Sobolev space 
\begin{equation} \label{Sobolev_space}
H^{1/2}(\R)
=\left\{ f \in L^2(\R\to \R) :  \int_{\mathbb{R}} |\kappa| \big|\hat{f}(\kappa)\big|^2 d\kappa <\infty  \right\} . 
\end{equation}
This is a Hilbert space equipped with the  inner-product
\begin{equation} \label{IP_1}
\langle f; g \rangle_{H^{1/2}} = \int_\R |\kappa| \hat{f}(\kappa) \hat{g}(-\kappa) d\kappa .
\end{equation}
There are other formulae for the inner product \eqref{IP_1} which do not involve the Fourier transform. For any functions $f, g\in C^1(\R)$ such that $f, g \in L^2(\R)$, we claim that
\begin{align}
\langle f; g \rangle_{H^{1/2}} 
&\label{IP_2}
= \frac{1}{4\pi^2} \iint_{\R^2}  \frac{f(x)-f(y)}{x-y}\frac{g(x)-g(y)}{x-y} dxdy \\
&\label{IP_3}
= \frac{-1}{2\pi} \int_\R  f'(u) \Hi(f)(u) du .
\end{align}
It can be checked that  formula \eqref{IP_2} holds for any functions $f, g\in H^{1/2}(\R)$, while \eqref{IP_3} holds as long as $\F(f') \in L^2(\R)$ and    
$\F(f')(\kappa) =  \pii \kappa \hat{f}(\kappa)$. \\

We define $\Xi$ to be the Gaussian noise associated with the Hilbert space $H^{1/2}(\R)$, see for instance~\cite{Janson97}. That is $\Xi = \{ \Xi(f)\}_{f\in H^{1/2}(\R)}$
is a Gaussian process with covariance structure:
\begin{equation} \label{covariance_Xi}
\E{\Xi(f) \Xi(g)} = \langle f; g \rangle_{H^{1/2}} . 
\end{equation}
Observe that, if $\J_u(x) = \pi \1_{|x-u| \le \ell/2}$, then  
\begin{equation}  \label{J_hat} 
\widehat{\J_u}(\kappa) = e^{-\pii u \kappa} \frac{\sin(\pi\ell\kappa)}{\kappa} . 
\end{equation}
So that, if we define $\widehat{Q}(\kappa)=\frac{\sin^{2}(\pi \ell \kappa)}{|\kappa|}$,  it is easy to check that despite the fact the functions $\J_u \notin H^{1/2}(\R)$, we have
\begin{equation}\label{Q_covariance}
\langle \J_u ; \J_v \rangle_{H^{1/2}} = Q(u-v)
\end{equation}
for all $u\neq v$. Hence, in a \textit{formal} sense, the log-correlated Gaussian field $G$ with zero mean and covariance function $Q$ can be realized as $G(u) = \Xi(\J_u)$. This is only formal because the function $\J_u$ does not belong to the domain of the noise $\Xi$.  
However we may rigorously define regularizations of the field $G$ using this procedure.
 For any $\phi \in\mathcal{D}_0$ as in \eqref{D}, and $\epsilon>0$, we let
 \begin{equation} \label{G_regularization}
 G_{\phi,\epsilon}(u) =\Xi( \J_u\ast \phi_\epsilon) . 
\end{equation}
Then, by formulae \eqref{IP_1}, \eqref{J_hat}, and the definition of $Q$,  the field  $ G_{\phi,\epsilon}$ has the correlation structure: 
\begin{align} \notag
\E{ G_{\phi,\epsilon}(u),  G_{\psi,\delta}(v)}
& = \langle \J_u \ast \phi_\epsilon  ; \J_v \ast \psi_\delta \rangle_{H^{1/2}}  \\
&\label{covariance_3}
=\int_\R e^{-\pii (u-v)\kappa} \hat\psi(\delta\kappa) \overline{\hat\phi(\epsilon\kappa)} \widehat{Q}(\kappa)  d\kappa 
\end{align}
for any $\phi, \psi \in \mathcal{D}_0$, $\epsilon, \delta>0$ and $u,v\in\R$. 
Note that by Plancherel's formula, we may also rewrite formula \eqref{covariance_3} as 
\begin{align*}
\E{ G_{\phi,\epsilon}(u),  G_{\psi,\delta}(v)}= \iint_{\R^2} \phi_\epsilon(u-x) \psi_\delta(v-y) Q(x-y) dxdy  . 
\end{align*}
Then, when it is convenient, we shall also denote the field  $G_{\phi,\epsilon}$ by $G\ast \phi_\epsilon$. Moreover, the log-correlated field $G$ can be realized as $G=\lim_{\epsilon\to0}G_{\phi,\epsilon}$ in the sense of random distributions. It is not difficult to make this convergence rigorous, \textit{e.g.} by using the asymptotics of Section~\ref{sect:covariance} for the covariance kernel \eqref{covariance_3}, but we do not pursue this here.

\section{Gaussian Multiplicative Chaos} \label{sect:GMC}

In this section, we review in further detail the theory of Gaussian Multiplicative Chaos (GMC) with respect to the Lebesgue measure on a compact subset $\A\subset \R^d$. 
This theory which originates in the work of Mandelbrot and Kahane \cite{Mandelbrot74a, Mandelbrot74b, KP76, Kah85} aims at defining the exponential of a log-correlated random field $G$, denoted formally by
\begin{equation}\label{chaos_measure}
\nu^\gamma(dx) = e^{\gamma G(x) - \frac{\gamma^{2}}{2}\mathbb{E}(G(x)^{2})} dx . 
\end{equation}
The original motivation to study such an object goes back to the work of Kolmogorov who proposed that the measure $\nu^\gamma$ should describe the energy dissipation in a turbulent fluid; c.f.~\cite{PGC16} for a modern reference.  Another motivation comes from the fact that $\nu^\gamma$ can be interpreted as the Boltzmann-Gibbs measure associated with the {\it random Hamiltonian} $G$.
Then, this measure describes the equilibrium configuration of a particle in a very rough landscape and $\gamma>0$ plays the role of the inverse-temperature and is usually called the intermittency parameter. In fact, sampling from the measure  $\nu^\gamma$ gives information about the points where the field $G$ takes unusually high values known as {\it $\gamma$-thick points}, see~\cite{HMP10}. In particular, there exists a critical value of $\gamma$ above which no such points exist and the measure \eqref{chaos_measure} needs to be renormalized in a different way and becomes purely atomic. This is known as the freezing transition in the theory of spin glasses and it has been observed that the behavior of $\nu^\gamma$ at criticality is related to the law of the maximum of the field $G$ \cite{FB08, FK14}. Recently, there have also been intense developments in the case where $G$ is the Gaussian Free Field associated to a domain in the complex plane. Then, the chaos measure \eqref{chaos_measure} is known as the Liouville measure and it has been one of the key inputs in a program aiming at giving a mathematically rigorous construction of Liouville quantum gravity and Liouville quantum field theory, c.f.~the recent results of \cite{DKRV16} and \cite{MS16}. The latter paper is concerned with developing an important program of {\it imaginary geometries} and {\it Liouville quantum gravity and the Brownian map} which aims at proving that the Liouville measures are central objects which arise, for instance to describe the scaling limit of random planar maps. In addition various KPZ relations have been established \cite{BS09, DS11, RV11}. For a more detailed introduction and further references to these topics, we refer to the lecture notes \cite{Garban13, Berestycki16, RV16} or the survey \cite{RV14}. Measures of the type \eqref{chaos_measure} also started to play an increasingly important role in random matrix theory \cite{W15,BWW16} and in statistics of the Riemann zeta function high up on the critical line \cite{SW16}.

 Usually, the random measure $\nu^\gamma$ is defined by first regularizing the field $G$ in some way and then by taking a limit as the regularization tends to zero. There have been many important developments in understanding this procedure and the so-called subcritical case is now well understood for Gaussian regularizations \cite{Kah85, RV10, Shamov16, Berestycki15}. 

To be specific, let $G$ be a Gaussian process on $\A$ with a covariance kernel:
$$T(x,y) = - \log|x-y| + g(x,y) $$
where the function $g:\A^2\to\R\cup\{-\infty\}$ is continuous as an extended function and there exists a constant $C>0$ so that for all $x,y\in \A$,
\begin{equation} \label{g_bound}
 g(x,y) -\log^+|x-y|   \le C .
\end{equation}
We introduce this general setting since our main example is a stationary Gaussian process on $\R$ with covariance kernel $Q$ given by \eqref{cov}. 
Note that one usually assumes that $g$ is a continuous and bounded function, but these assumptions can be relaxed without changing the general theory, as long as a condition such as \eqref{g_bound} holds. 
Because of the singularity of the kernel $T$ on the diagonal, $G$ is not defined pointwise and needs to be interpreted as a random generalized function. 
Formally $G= \{ G(f)\}_{f\in C(\A)}$  is a Gaussian process with covariance structure:
$$
\E{G(f)G(g)} = \iint_{\R^2} f(x) g(y) T(x,y) dxdy . 
$$
In general, the definition of $G$ can be extended to a more general class of test functions than $C(\A)$, or even to certain classes of measures. To define the exponential measure \eqref{chaos_measure}, one can consider a regularization of the process $G$ coming from the convolution with an approximate delta function. This approach was introduced by Robert and Vargas in \cite{RV10} and developed further by Berestycki in \cite{Berestycki15}.    Namely, given $\epsilon>0$ and a mollifier $\phi$ (i.e.~a sufficiently smooth and light-tailed probability density function on $\R$),  we define
$$G_{\phi,\epsilon} =  G \ast \phi_\epsilon
\quad\text{ where }\quad
 \phi_\epsilon(x) = \epsilon^{-1} \phi(x/\epsilon), 
$$
and for any $\gamma>0$,
\begin{equation} \label{GMC_1}
\nu^\gamma_{\epsilon}(dx) = \exp\left(
\gamma G_\epsilon(x) - \frac{\gamma^2}{2} \E{G_\epsilon(x)^2}\right) dx. 
\end{equation}
Then, for any function $w\in L^1(\A\to\R_+)$ uniformly bounded, if $\gamma^{2} < 2d$, we let
\begin{equation} \label{GMC_2}
\nu^\gamma(w) =  \lim_{\epsilon\to0 } \nu^\gamma_{\epsilon}(w) .  
\end{equation} 
The random measures $\nu^\gamma$ are called the multiplicative chaos measures associated to the Gaussian process with  covariance kernel $T$. 
The major achievements of the GMC theory are that the 
limit \eqref{GMC_2} exists in probability (and almost surely in certain cases) and that it does not depend on $\phi$ for a large class of mollifiers. Moreover, the measure $\nu^\gamma$ is non-trivial if and only if $\gamma^2<2d$.  
In the critical ($\gamma=\sqrt{2d}$) and supercritical ($\gamma^2>2d$) regimes, one needs different normalizations than \eqref{GMC_1} to make sense of the GMC in a non-trivial way, c.f.~\cite[Section~6]{RV14} and reference therein.
In this paper, we will focus on the subcritical regime, in which case by Theorem~\ref{th:GMCintro},   the limit \eqref{GMC_2} holds in $L^1(\mathbb{P})$ and the  normalization is such that
\begin{equation*}
\E{\nu^\gamma(w) } =  \int w(u) du . 
\end{equation*} 
Moreover, for any $q\in\N$ such that $q\gamma^2 <2$, it is not difficult to show that
\begin{equation} \label{G_moments}
\E{\nu^\gamma(w)^q}  
:= \lim_{\epsilon\to0} \E{\nu_{\epsilon}^\gamma(w)^q}
=  \int_{\A^q} \exp\Bigg(\gamma^2\hspace{-.15cm}\sum_{1\le j<k\le q}\hspace{-.15cm} T(u_j,u_k) \Bigg)\prod_{k=1}^q w(u_k) du_k .
\end{equation}
In particular, in dimension $d=1$,  by a change of variables, this implies that for any $x_0\in\R$, 
\begin{equation*} 
\E{\nu^\gamma\big([x_0-\tfrac{r}{2}, x_0+\tfrac{r}{2}]\big)^q} 
=  r^{\xi(q)} \int_{[0,1]^q}\prod_{1\le j<k\le q}\hspace{-.15cm} |u_j-u_k|^{-\gamma^2}e^{\gamma^2 g\big(x_0+ r(u_j-\frac{1}{2}), x_0+ r(u_k-\frac{1}{2})\big)} \prod_{k=1}^q du_k ,
\end{equation*}
where $\xi(q) =  q- \gamma^2 \frac{q(q-1)}{2}$,
so that
$$
\E{\nu^\gamma\big([x_0-\tfrac{r}{2}, x_0+\tfrac{r}{2}]\big)^q}  \sim  r^{\xi(q)} \S(q;\gamma^2/2) e^{\gamma^2\binom{q}{2}g(x_0,x_0)} 
\qquad\text{as } r\to0 ,
$$
where
\begin{equation} \label{Selberg}
\S(n;\tilde{\gamma})
:=\int_{[0,1]^n}\prod_{i\neq j} |u_j-u_k|^{-\tilde{\gamma}} \prod_{k=1}^n du_k
= \prod_{j=0}^{n-1} \frac{\Gamma(1+j \tilde{\gamma})^2 \Gamma(1+\tilde{\gamma}+j\tilde{\gamma}))}{\Gamma(2+(n+j-1)\tilde{\gamma})\Gamma(1+\tilde{\gamma})} 
\end{equation}
is a Selberg integral. 
The quadratic function $\xi(q)$ is called the structure exponent and it describes the multi-fractal properties of the random measure $\nu^\gamma$, c.f.~\cite[section~2.3]{RV14}. 
Finally, since the Selberg integral converges if and only if $\Re\{n\tilde{\gamma}\}<1$, 
 this shows that the condition $q\gamma^2 <2$ in \eqref{G_moments} is sharp. \\
 
 Let us conclude this section by stating a result about the convergence of the GMC measure in the so-called $L^2$-phase ($\gamma^2<d$).
  In particular, we will need this result 
 in Section~\ref{sect:L^2} in order to identify the law of the multiplicative chaos measure coming from counting statistics of the CUE or sine process. 
 In addition, the strategy of the proof will be re-used  and generalized in Sections~\ref{sect:L^2} and~\ref{sect:L^1} to construct multiplicative chaos measures for the asymptotically Gaussian fields discussed in the introduction.
  
 \begin{proposition} \label{L^2_GMC}
Let $q$ be an even integer and $\gamma>0$ so that $q\gamma^2<2d$. For any $w\in L^1(\A)$ uniformly bounded, the random variable $\nu_{Q,\epsilon}^\gamma(w)$ given by \eqref{GMC_1} converges as $\epsilon\to0$  in $L^q(\mathbb{P})$ to a random variable $\nu^\gamma_Q(w)$ which does not depend on the mollifier $\phi$ subject to the conditions that $\phi$ is smooth and $\phi \in \mathcal{D}_\alpha$ for some  sufficiently large $\alpha>0$, \eqref{D}.
\end{proposition}

\proof
This follows from a standard argument; c.f.~for instance the proof of \cite[Theorem~2.3]{RV16}.
 First, given a mollifier $\phi$, one can prove that $\nu_{\epsilon}^\gamma(w)$ is a Cauchy sequence in $L^q(\mathbb{P})$. When $q$ is an even integer, this just boils down to checking  that 
\begin{equation}\label{condition_1}
\E{G_{\phi,\epsilon}(x) G_{\phi,\epsilon'}(x')} \to T(x,x')
\end{equation}
as $\epsilon, \epsilon'\to0$ in such a way that we can apply the dominated convergence theorem. To prove  that the limit does not depend on $\phi$, if $\nu_{\epsilon}^{\prime\gamma}$ is the measure associated to the Gaussian process $G_{\psi,\epsilon}$ for a second mollifier $\psi$, then it suffices to show that 
$$
\lim_{\epsilon\to0}\E{| \nu_{\epsilon}^{\prime\gamma}(w) - \nu_{\epsilon}^{\gamma}(w)|^q} =0.
$$
This limit follows in a similar fashion  by checking that as $\epsilon\to0$, 
\begin{equation*}
\E{G_{\phi,\epsilon}(x) G_{\psi,\epsilon}(x')} \to T(x,x') .  \qed\\
\end{equation*}

\begin{remark}\label{rk:Berestycki}
In what follows, the condition \eqref{condition_1} is replaced
by Assumption~\ref{kernel_asymp}. In fact, by going carefully through the proof in \cite{Berestycki15}, it is not difficult to check that  if  
$$ T_{\epsilon,\delta}(x,x') = \E{G_{\phi,\epsilon}(x) G_{\phi,\delta}(x')}
$$
satisfies Assumption~\ref{kernel_asymp}, then for any $\gamma <\sqrt{2d}$, 
 $\nu_{\epsilon}^\gamma(w)$  converges in $L^1(\mathbb{P})$ to $\nu^\gamma(w)$  as $\epsilon\to0$. 
Moreover, for the stationary Gaussian process on $\R$ with covariance kernel $Q$ given by \eqref{cov} which arises in Theorem~\ref{th:gmc} for the mesoscopic CUE, as well as in Theorem~\ref{thm:sine_GMC} for the sine process,
it is proved in Section~\ref{sect:covariance} that Assumption~\ref{kernel_asymp} holds for any mollifier $\phi \in \mathcal{D}_\alpha$  for any $\alpha>0$. 
\end{remark}

\bibliographystyle{plain}
\bibliography{LamOstSim}
\end{document}